\documentclass[reqno,11pt]{amsart}
\usepackage{amsmath,amsfonts,amssymb,amsxtra,latexsym,amscd,enumerate,amsthm,verbatim}

\usepackage{hyperref}
\usepackage{graphicx}
\usepackage{physics}
\usepackage{bbm}

\usepackage{hyperref} 
\hypersetup{
    colorlinks=true,
    citecolor= blue,
    linkcolor=black,
    filecolor=magenta,      
    urlcolor=cyan
    }

\usepackage[margin=1in]{geometry}
\numberwithin{equation}{section}

\newcommand{\R}{\mathbb{R}}

\newcommand{\T}{\mathbb{T}}
\newcommand{\C}{\mathbb{C}}
\newcommand{\Z}{\mathbb{Z}}

\numberwithin{equation}{section} 

\newcommand{\tp}{\mathfrak{p}}
\newtheorem{theorem}{Theorem}[section]

\newtheorem{assumption}{Assumption}[section]
\newtheorem{definition}{Definition}[section]
\newtheorem{lemma}{Lemma}[section]
\newtheorem{proposition}{Proposition}[section]
\newtheorem{corollary}{Corollary}[section]
\newtheorem{remark}{Remark}[section]

\newcommand{\p}{\partial} 
\newcommand{\lb}{\langle}
\newcommand{\rb}{\rangle} 
\newcommand{\ept}{\epsilon^{-\frac{1}{3}}}
\newcommand{\epbp}{\epsilon^{\frac{1 + 2 \beta}{3}}}
\newcommand{\epbn}{\epsilon^{-\frac{1 + 2 \beta}{3}}}
\newcommand{\epbv}{\epsilon^{-\frac{1 - \beta}{3}}}
\newcommand{\red}[1]{\textcolor{red}{#1}}

\newcommand{\tc}[2]{T_{#1#2}} 
\newcommand{\gr}{\mathcal{G}_{k} (y,z;\Lambda)} 
\newcommand{\grl}{\mathcal{G}_{k_\ell} (y,z;\Lambda_\ell)}

\newcommand{\ai}{A_{\Theta}^{-1}}
\newcommand{\ail}{A_{\Theta_\ell}^{-1}}

\newcommand{\sd}[1]{S_{ #1  \delta(\Lambda)}}

\newcommand{\blue}[1]{\textcolor{blue}{#1}}
\usepackage[shortlabels]{enumitem}

\begin{document}

\title{Uniform vorticity depletion and inviscid damping for periodic shear flows in the high Reynolds number regime}

\author{Rajendra Beekie}
\address{Duke University}
\email{rajendra.beekie@duke.edu}
\author{Shan Chen}
\address{University of Minnesota}
\email{chen7081@umn.edu}
\author{Hao Jia}
\address{University of Minnesota}

\email{jia@umn.edu}


\thanks{HJ and SC  are supported in part by NSF grant DMS-1945179. RB is supported in part by NSF grant DMS-2202974}

\begin{abstract}
{\small}

We study the dynamics of the two dimensional Navier-Stokes equations linearized around a shear flow on a (non-square) torus which possesses exactly two non-degenerate critical points. We obtain linear inviscid damping and vorticity depletion estimates for the linearized flow that are uniform with respect to the viscosity, and enhanced dissipation type decay estimates. The main task is to understand the associated Rayleigh and Orr-Sommerfeld equations, under the natural assumption that the linearized operator around the shear flow in the inviscid case has no discrete eigenvalues. The key difficulty is to understand the behavior of the solution to Orr-Sommerfeld equations in three distinct regimes depending on the spectral parameter: the non-degenerate case when the spectral parameter is away from the critical values, the intermediate case when the spectral parameter is close to but still separated from the critical values, and the most singular case when the spectral parameter is inside the viscous layer. 
\end{abstract}

\maketitle
\setcounter{tocdepth}{1}
\pagestyle{plain}

\tableofcontents

\section{Introduction and main results}\label{sec:imr}
Assume that $\nu\in(0,1)$ and $\tp\in(2\pi,\infty)$. Consider the equation  
\begin{equation}\label{int1}
\left\{\begin{array}{rl}
\partial_t\omega-\nu\Delta\omega+b(y)\partial_x\omega-b''(y)\partial_x\psi=0,&\\
\Delta\psi=\omega,&
\end{array}\right.
\end{equation}
for $(x,y,t)\in\T\times\T_{\tp}\times[0,\infty)$. 

Equation \eqref{int1} can be obtained as the linearization of the two dimensional Navier-Stokes equations on $\T\times\T_{\tp}$ around the meta-steady state $(b(y), 0)$. Alternatively, one can view the shear flow $(b(y),0)$ as an exact steady state to the Navier-Stokes equations with a small external forcing $F:=(-\nu b''(y),0), y\in\T_{\tp}$. Equation \eqref{int1} was studied classically, in an attempt to explain the onset of turbulence. See for example \cite{Eckert} for an interesting history on this problem. We refer also to \cite{Almog,CWZ2,JiaUM} for more recent results. 

We note that it is natural to study \eqref{int1} on a non-square torus (for our purposes), since on a square torus some of the most interesting shear flows such as Kolmogorov flows are not stable \cite{CotiZelatiKol}. To establish asymptotic stability, a more natural object to study on a square torus seems to be the dipole flow, see Remark 1.3 in \cite{Dongyi3}. 

Our main goal in this paper is to understand the precise dynamics of \eqref{int1} starting from relatively smooth initial data, and obtain bounds on the solution that are uniform with respect to the viscosity $\nu\in(0,\nu_0)$ for some small $\nu_0>0$. In particular we shall capture some of the 
more refined dynamics associated with periodic shear flows, and prove uniform vorticity depletion and linear inviscid damping results. Such problems are well known and were discussed for example in section 11, chapter 6 of the book \cite{BeVi} by Bedrossian and Vicol .

\subsection{Review of recent results on inviscid damping and main problems} The study of stability problems in mathematical analysis of fluid dynamics has a long and distinguished history, dating back to the work of Kelvin \cite{Kelvin}, Orr \cite{Orr} and Rayleigh \cite{Ray} among many others, and continuing to the present day. Despite many significant recent progresses, many important questions still remain open. The main issue is that very often the most interesting physical settings involve high Reynolds number (or equivalently the viscosity has to be taken very small), where the Navier-Stokes equations degenerate to Euler equations and the continuous spectrum plays a dominant role.  As a starting point to a more complete and precise theory of hydrodynamic stability, it is therefore important to understand mathematically the property of the linearized flow in the high Reynolds number regime, including the limiting inviscid case. Below we briefly review some recent works on the linearized flow for both Euler and Navier-Stokes equations with small viscosity, focusing on two dimensional shear flows which is our main object of study in this paper, and mentioning nonlinear inviscid damping and other extensions only briefly. 

\subsubsection{The case of monotonic shear flow}
For the Euler equations, there are significant recent progresses on the asymptotic stability of monotonic shear flows and vortices, assuming spectral stability, see for example \cite{Bed2,Grenier,JiaVortex,JiaL,JiaG,Xiaoliu,Stepin,dongyi,Zillinger1,Zillinger2} for linear results. The main mechanism of stabilization is the so called ``inviscid damping", which refers to the transfer of energy of vorticity to higher and higher frequencies leading to decay of the stream and velocity functions, as $t\to\infty$.  Extending the linearized stability analysis for inviscid fluid equations to the full nonlinear setting is a challenging problem, and the only available results are on spectrally stable monotonic shear flows \cite{BM,IJacta,IOJI,NaZh}, and on point vortices \cite{IOJI2}. We refer also to the recent review article \cite{IJICM} for a more in-depth discussion of recent developments of both linear and nonlinear inviscid damping. 


In view of the results on linear inviscid damping for the Euler equation, it is natural to expect that even for Navier-Stokes equations with small viscosity we should have the same inviscid damping results with explicit rates of decay, in time, of the stream functions and velocity fields. There are several results in this direction, all of which are for {\it monotonic} shear flows, see for example \cite{BH20, Bed6,CLWZ,MZ22} for the Couette flow with full nonlinear analysis or precise linear results, and \cite{Grenier} for the spectrally restricted stream function for ``mixing layer" type shear flows (but a description of the full solution without spectral restrictions is not available).  More recently, Chen, Wei and Zhang \cite{CWZ2} established uniform inviscid damping with quantitative bounds for general spectrally stable monotonic shear flows in a periodic channel. More precisely, they proved that the stream function decays like $1/t$ as $t\to\infty$, uniformly in viscosity. We remark that it remains an interesting question if the stream function decays as in the Euler equation (with rate $1/t^2$). It seems that one needs to understand the boundary layer more clearly to settle this quesiton.  The third author \cite{JiaUM} proved uniform (in viscosity) Gevrey bounds on the profile of the vorticity, for spectrally stable monotonic shear flows in $\T\times\R$ (thus avoiding the boundary layer issue), which yields sharp bounds on the stream function (and in particular $1/t^2$ decay rate).

\subsubsection{The case of non-monotonic shear flows}
Many physically important shear flows are not monotonic, such as Poiseuille flow and Kolmogorov flows. For such flows on the linear inviscid level, there is an additional significant physical phenomenon called ``vorticity depletion" which refers to the asymptotic vanishing of vorticity as $t\to\infty$ near the critical points of the shear flow, first predicted in Bouchet and Morita \cite{Bouchet}, and proved rigorously in Wei-Zhang-Zhao \cite{Dongyi2}.  A similar phenomenon for the case of vortices was observed earlier by Bassom and Gilbert \cite{Bassom} and demonstrated in Bedrossian-Coti Zelati-Vicol \cite{Bed2}. In \cite{JiaVortex} Ionescu and the third author obtained a refined description of the dynamics of the linearized flow around monotonic vortices in Gevrey spaces, which is a critical step towards proving nonlinear vortex symmetrization. 

In \cite{Dongyi2} by Wei-Zhang-Zhao, sharp linear inviscid damping estimates and quantitative depletion estimates were obtained for an important class of ``symmetric shear flows" in a periodic channel (see also \cite{Dongyi3} by Wei-Zhang-Zhao for a similar result for Kolmogorov flow). When no symmetry is assumed, Ionescu, Iyer and the third author proved optimal inviscid damping rates and quantitative vorticity depletion estimates in \cite{Iyer}.

As in the case of monotonic shear flows, it is also natural to consider the linearized equation around non-monotonic shear flows for the Navier-Stokes equations and study the corresponding vorticity depletion and linear inviscid damping estimates that are uniform with respect to the viscosity $0<\nu\ll1$. This problem is more complicated than the monotonic case and, as far as we know, is completely open, since the corresponding Orr Sommerfeld equations are {\it genuinely} fourth order singular ODEs and the estimates needed to establish optimal vorticity depletion and linear inviscid damping are very strong. We note also that the Orr-Sommerfeld equations have been studied classically, with a focus on distribution of eigenvalues, see \cite{Gre1,Gre2} and references therein for recent results in this direction. We also mention the essay \cite{Eckert} for an interesting account on the history of research in this direction. Our main task here is quite different, which is to establish very precise control for the resolvent for the spectral parameter on the continuous spectrum (of the inviscid Euler equation).

\subsection{Main equations and assumptions}

In this paper we address this problem and study the uniform vorticity depletion and linear asymptotic stability around a periodic shear flow $(b(y),0)$ on $\T\times\T_{\tp}$, together with enhanced dissipation estimates. The choice of the domain is motivated by our goal of understanding the interaction of viscosity and critical points of shear flows, and by the desire to avoid the boundary layers which arise when there is a boundary. We note that the issues of critical points and boundary layers are somewhat decoupled, and for the sake of clarity of argument, it makes sense to tackle these difficulties separately. 

Strictly speaking, $(b(y),0)$ is not a steady state for the Navier-Stokes equations, and becomes steady state only with a small external force $F=(-\nu b''(y),0)$. In our setting, we consider the viscosity $\nu$ to be very small. The effect of diffusion on the background shear flow is negligible, at least up to the diffusion time scale $T\ll \frac{1}{\nu}$. After the diffusion time, due to the enhanced dissipation effect, the flow is essentially dominated by a heat evolution, at least heuristically. Therefore our analysis below still captures the main dynamics of the Navier-Stokes equations near shear flows, even when we do not add the external forcing.  

\subsubsection{Main spectral assumption on the background shear flow} 

Our main assumptions on the background flow $(b(y),0), y\in \T_{\tp}$ are the following.
\begin{assumption}\label{MaAs}
We assume that the shear flow (b(y),0), $y\in \T_{\tp}$ satisfies the following conditions:
\begin{itemize}
\item $b\in C^4(\T_{\tp})$. The set $S:=\{y\in\T_{\tp}: b'(y)=0\}=\{y_{1\ast}, y_{2\ast}\}$, and for some $\varkappa\in(0,1)$, we have
\begin{equation}\label{asp1}
|b''(y)|\in[\varkappa,1/\varkappa]\,\,{\rm for\,\,}y\in S, \quad{\rm and}\quad\|b\|_{C^4(\T_\tp)}\leq1/\varkappa.
\end{equation}

\item The linearized operator $L_k: L^2(\T_{\tp})\to L^2(\T_{\tp})$ with $k\in\Z\backslash\{0\}$ defined for $g\in L^2(\T_{\tp})$ as
\begin{equation}\label{asp2}
L_kg(y):=b(y)g(y)-b''(y)\varphi, \quad{\rm where}\,\,(-k^2+\partial_y^2)\varphi(y)=g(y),\,\,y\in\T_{\tp},
\end{equation}
has no eigenvalues, nor generalized embedded eigenvalues (also called singular limiting modes) in $\{b(y), y\in\T_{\tp}\}$. 
\end{itemize}
\end{assumption}

In contrast to the spectral property of the linearized operator around monotonic shear flows, the spectral property of $L_k$ is less understood, especially on the mechanism of the generation of discrete eigenvalues and embedded eigenvalues. From general spectral theory, we know that the spectrum of $L_k$ consists of the continuous spectrum 
 \begin{equation}\label{continspe1}
 \Sigma:=\big\{b(y):\,y\in \T_\tp\big\},
 \end{equation} 
 together with some discrete eigenvalues with nonzero imaginary part which can only accumulate at the set of continuous spectrum $\Sigma$. Unlike the case of monotonic shear flows where the discrete eigenvalues can accumulate only at inflection points of the background shear flow, there appears no simple characterization of the possible accumulation points for non-monotonic shear flows.
 
 Recall that $\lambda\in\Sigma$ is called an embedded eigenvalue if there exists a nontrivial $g\in L^2(\T_{\tp})$, such that 
 \begin{equation}\label{F3.20}
 L_kg=\lambda g. 
 \end{equation}
 
 For non-monotonic shear flows, this definition is too restrictive, as accumulation points of discrete eigenvalues may no longer be embedded eigenvalues. To capture the discrete eigenvalues, we recall the following definition of ``generalized embedded eigenvalues", which can be found already in \cite{Dongyi2}, adapted to our setting. 
 
 \begin{definition}\label{emb1}
 We call $\lambda\in\Sigma$ a generalized embedded eigenvalue, if one of the following conditions is satisfied. 
 \begin{itemize}
 \item  $\lambda$ is an embedded eigenvalue.
 
 \item $\lambda\not\in\{ b(y_{1\ast}), b(y_{2\ast})\}$ and there exists a nontrivial $\psi\in H^1(\T_{\tp}): \T_\tp \to \mathbb{C}$ such that in the sense of distributions on $\T_{\tp}$,
 \begin{equation}\label{emb2}
 (k^2-\partial_y^2)\psi(y)+{\rm P.V.}\frac{b''(y)\psi(y)}{b(y)-\lambda}+i\pi\sum_{z\in \T_\tp, \,b(z)=\lambda}\frac{b''(z)\psi(z)}{|b'(z)|}\delta(y-z)=0.
 \end{equation}
 
 \end{itemize}
 
 \end{definition}
 
 We remark that our assumption that the critical points $y_{1\ast}, y_{2\ast}$ of $b(y)$ being non-degenerate implies that the sum in \eqref{emb2} is finite. The spectral assumption \ref{MaAs} is satisfied by a large class of shear flows $(b(y),0)$, including $\sin\frac{2\pi y}{\tp}$.

\subsubsection{Main equations}
Taking the Fourier transform in $x$ in the equation \eqref{int1}, we obtain for each $k\in\Z$ that
\begin{equation}\label{int2}
\left\{\begin{array}{rl}
\partial_t\omega_k+\nu(k^2-\partial_y^2)\omega_k+ikb(y)\omega_k-ikb''(y)\psi_k=0,&\\
(-k^2+\partial_y^2)\psi_k=\omega_k,&
\end{array}\right.
\end{equation}
for $(y,t)\in\T_{\tp}\times[0,\infty)$.
 For $k\in\Z\backslash\{0\}$, define the operator $L_{k,\nu}: H^2(\T_{\tp})\to L^2(\T_{\tp})$ as follows. For $g\in H^2(\T_{\tp})$ and $y\in\T_{\tp}$,
\begin{equation}\label{int5}
L_{k,\nu}g(y):=(\nu/k)\partial_y^2g-ib(y)g+ib''(y)\varphi, \quad{\rm with}\,\,(-k^2+\partial_y^2)\varphi=g.
\end{equation}
We can rewrite equation \eqref{int2} more abstractly as
\begin{equation}\label{int3}
\partial_t\omega_k^\ast=kL_{k,\nu}\omega_k^\ast, \quad{\rm for}\,\,t\ge0,
\end{equation}
where we have set for $y\in\T_{\tp}$ and $t\ge0$,
\begin{equation}\label{int4}
\omega_k^\ast(t,y)=e^{\nu k^2t}\omega_k(t,y).
\end{equation}
By relatively standard spectral theory and our bounds below (see section \ref{sec:pmt} for more details), for sufficiently small $\nu_0>0,\sigma_\sharp>0$ and $\alpha\ge-\sigma_\sharp|\nu/k|^{1/2}$, (denoting $\iota={\rm sign}\, k$) we have the following representation formula for $y\in\T_{\tp}$ and $t>0$,
\begin{equation}\label{int6}
\begin{split}
\omega_k^\ast(t,y)&=\frac{1}{2\pi i}\int_{i\R+\iota\alpha } e^{\mu tk}\big[(\mu-L_{k,\nu})^{-1}\omega_{0k}\big](y)\,d\mu\\
&=\frac{1}{2\pi }\int_{\R} e^{i\lambda tk+\alpha |k|t}\big[(i\lambda+\iota\alpha-L_{k,\nu})^{-1}\omega_{0k}\big](y)\,d\lambda\\
&=-\frac{1}{2\pi }e^{\alpha|k|t}\int_{\R} e^{-i\lambda tk}\big[(i\lambda-\iota\alpha+L_{k,\nu})^{-1}\omega_{0k}\big](y)d\lambda.
\end{split}
\end{equation}
In the above, $\omega_{0k}$ denotes the $k-$th Fourier coefficient of the initial data $\omega|_{t=0}$.

Define for $y\in\T_{\tp}, \lambda\in\R$ and $\iota={\rm sign}\, k$,
\begin{equation}\label{int7}
\omega_{k,\nu}(y,\lambda+i\iota\alpha):=\big[(i\lambda-\iota\alpha+L_{k,\nu})^{-1}\omega_{0k}\big](y).
\end{equation}
It follows from \eqref{int7} that 
\begin{equation}\label{int8}
(i\lambda-\iota\alpha)\omega_{k,\nu}(y,\lambda+i\iota\alpha)+L_{k,\nu}\omega_{k,\nu}(y,\lambda)=\omega_{0k}(y), \quad{\rm for}\,\,y\in\T_{\tp}, \lambda\in\R.
\end{equation}
Therefore $\omega_{k,\nu}(y,\lambda)$ satisfies the equation for $k\in\Z\backslash\{0\}$, $y\in\T_{\tp}, \lambda\in\R$, $\iota={\rm sign}\, k$,
\begin{equation}\label{int9}
\left\{\begin{array}{rl}
\Big[\frac{\nu}{k}\partial_y^2+i(\lambda-b(y))-\iota\alpha\Big]\omega_{k,\nu}(y,\lambda+i\iota\alpha)+ib''(y)\psi_{k,\nu}(y,\lambda+i\iota\alpha)=\omega_{0k}(y),&\\
(-k^2+\partial_y^2)\psi_{k,\nu}(y,\lambda+i\iota\alpha)=\omega_{k,\nu}(y,\lambda+i\iota\alpha).&
\end{array}\right.
\end{equation}

We summarize our calculations in the following proposition.
\begin{proposition}\label{intP1}
Assume that $\nu\in(0,\nu_0)$ with sufficiently small $\nu_0>0$ and $k\in\Z\backslash\{0\}$. Suppose that $\omega_k(t,y)$ satisfies the regularity condition $\omega_k(t,y)\in C^\infty((0,\infty)\times\T_{\tp})$ and $\omega_k(t,\cdot)\in C([0,\infty), L^2(\T_{\tp}))$ is the solution to the system of equations 
\begin{equation}\label{intP2}
\left\{\begin{array}{rl}
\partial_t\omega_k(t,y)+\nu(k^2-\partial_y^2)\omega_k(t,y)+ikb(y)\omega_k(t,y)-ikb''(y)\psi_k(t,y)&=0,\\
(-k^2+\partial_y^2)\psi_k(t,y)&=\omega_k(t,y),
\end{array}\right.
\end{equation}
for $(y,t)\in\T_{\tp}\times[0,\infty)$, with initial data $\omega_k(0,y)=\omega_{0k}(y)\in C^\infty(\T_{\tp})$. Define the operator $L_{k,\nu}: H^2(\T_{\tp})\to L^2(\T_{\tp})$ as follows. For any $g\in H^2(\T_{\tp})$, 
\begin{equation}\label{intP3}
L_{k,\nu}g(y)=(\nu/k)\partial_y^2g-ib(y)g+ib''(y)\varphi, \quad{\rm where\,\,}(-\partial_y^2+k^2)\varphi=g\,\,{\rm for}\,\,y\in\T_{\tp}. 
\end{equation}
 Set for $y\in\T_{\tp}, \lambda\in\R$, $\iota={\rm sign}\, k$, $\alpha\ge-\sigma_\sharp|\nu/k|^{1/2}$ with sufficiently small $\sigma_\sharp\in(0,1/10)$, 
\begin{equation}\label{intP4}
\omega_{k,\nu}(y,\lambda+i\iota\alpha)=\big[(i\lambda-\iota\alpha+L_{k,\nu})^{-1}\omega_{0k}\big](y).
\end{equation}
Then $\omega_{k,\nu}(y,\Lambda)$ satisfies the equation for $k\in\Z\backslash\{0\}$, $y\in\T_{\tp}, \lambda\in\R$,
\begin{equation}\label{intP5}
\left\{\begin{array}{rl}
\Big[\frac{\nu}{k}\partial_y^2+i(\lambda-b(y))-\iota\alpha\Big]\omega_{k,\nu}(y,\lambda+i\iota\alpha)+ib''(y)\psi_{k,\nu}(y,\lambda+i\iota\alpha)=\omega_{0k}(y),&\\
(-k^2+\partial_y^2)\psi_{k,\nu}(y,\lambda+i\iota\alpha)=\omega_{k,\nu}(y,\lambda+i\iota\alpha).&
\end{array}\right.
\end{equation}
We have the representation formulae for $y\in\T_{\tp}, t>0$,
\begin{equation}\label{intP6}
\begin{split}
\omega_k(t,y)&=-\frac{1}{2\pi }e^{\alpha|k|t}e^{-\nu k^2t}\int_{\R} e^{-ik\lambda t}\omega_{k,\nu}(y,\lambda+i\iota\alpha)\,d\lambda,\\
\psi_k(t,y)&=-\frac{1}{2\pi }e^{\alpha|k|t}e^{-\nu k^2t}\int_{\R} e^{-ik\lambda t}\psi_{k,\nu}(y,\lambda+i\iota\alpha)\,d\lambda.
\end{split}
\end{equation}

\end{proposition}

Strictly speaking the vorticity function $\omega_k(t,y)$ and stream function $\psi_k(t,y)$ also depend on $\nu$. We omit this dependence from our notations for the sake of simplicity, since there is no danger of confusion.

\subsection{Main results}

Our main result is the following theorem.
\begin{theorem}\label{thm}
Assume the spectral condition \ref{MaAs}. Fix $\gamma\in[15/8,2)$. There exist $\sigma_{\sharp}\in(0,1)$ and $\nu_0\in(0,1/10)$ sufficiently small depending on $\gamma$, such that the following statement holds for $\nu\in(0,\nu_0)$.  Assume that 
$$\omega(t,x,y)\in C([0,\infty), H^3(\T\times\T_\tp))$$ with the associated stream function $ \psi(t,x,y)$ is the unique solution to \eqref{int1}, with initial data $\omega_0\in H^3(\T\times\T_\tp)$ satisfying for all $y\in\T_\tp$,
\begin{equation}\label{thm0}
\int_{\T}\omega_0(x,y)\,dx=0.
\end{equation}

Then we have the following bounds.

(i) Uniform inviscid damping estimates:
for $(x,y,t)\in\T\times\T_\tp\times[0,\infty)$,
\begin{equation}\label{th0.5}
|\psi(t,x,y)|\lesssim_{\gamma} \frac{|y-y_{1\ast}|^{\gamma-2}+|y-y_{2\ast}|^{\gamma-2}}{\langle t\rangle^2}e^{-\sigma_{\sharp}\nu^{1/2}t}\|\omega_0\|_{H^3(\T\times\T_\tp)},
\end{equation}
\begin{equation}\label{thm1}
\begin{split}
&|u^x(t,x,y)|\lesssim_{\gamma} \frac{1}{\langle t\rangle}e^{-\sigma_\sharp\nu^{1/2}t}\|\omega_0\|_{H^3(\T\times\T_\tp)}, \\
&|u^y(t,x,y)|\lesssim_{\gamma} \frac{|y-y_{1\ast}|^{\gamma-2}+|y-y_{2\ast}|^{\gamma-2}}{\langle t\rangle^2}e^{-\sigma_\sharp\nu^{1/2}t}\|\omega_0\|_{H^3(\T\times\T_\tp)}.
\end{split}
\end{equation} 

(ii) Uniform vorticity depletion estimates: for $j\in\{1,2\}$ there exists a decomposition
\begin{equation}\label{thm2}
\omega(t,x,y):=\omega^j_{\rm loc}(t,x,y)+\omega^j_{\rm nloc}(t,x,y),
\end{equation}
where for $(x,y,t)\in\T\times\T_\tp\times[0,\infty)$,
\begin{equation}\label{thm3}
\begin{split}
&|\omega^j_{\rm loc}(t,x,y)|\lesssim_{\gamma} \big(|y-y_{j\ast}|+\nu^{1/4}\big)^{\gamma}e^{-\sigma_\sharp\nu^{1/2}t}\|\omega_0\|_{H^3(\T\times\T_\tp)},\\
& |\omega^j_{\rm nloc}(t,x,y)|\lesssim_{\gamma} \frac{1}{\langle t\rangle^{\gamma/2}}e^{-\sigma_\sharp\nu^{1/2}t}\|\omega_0\|_{H^3(\T\times\T_\tp)}.
 \end{split}
\end{equation}
In the above, $\langle t\rangle:=\sqrt{1+t^2}$, and the implied constants depend, in addition to $\gamma$, also on $\varkappa$ from \eqref{asp1} and the structure constant $\kappa\in(0,1)$ from Proposition \ref{lap_main}.
\end{theorem}

\bigskip

We have the following remarks on Theorem \ref{thm}. 

\begin{enumerate}
 \item The exponent $\gamma$ is almost sharp. The endpoint case $\gamma=2$  seems to require new ideas to treat. \\

    \item The restriction that $\tp\in(2\pi,\infty)$ is only needed for the spectral assumption \ref{MaAs} to hold. An important example that satisfies \ref{MaAs} is the Kolmogorov flow $(\sin(2\pi y/\tp), 0)$. See also Meshalkin and Sinai \cite{Sinai} for an investigation of spectral stability of Kolmogorov flows on general torus.\\

    \item The condition \eqref{thm0} removes the zero mode in $x$ of vorticity from our results, since the evolution of the zero mode corresponding to $k=0$ in equation \eqref{intP2} is trivial. \\
    
     \item The bounds \eqref{th0.5}-\eqref{thm1} provide uniform-in-viscosity inviscid damping estimates. We note that to get $1/t^2$ decay as $t\to\infty$, one needs to use integration by parts in \eqref{intP6} twice with respect to $\lambda$, while the decay bound on $u^x$ involves only one integration by parts in $\lambda$ and one derivative in $y$. In the neighborhood of the critical points $y_{1\ast}, y_{2\ast}$, taking a derivative with respect to $\lambda$ is roughly equivalent to $\frac{1}{b'(y)+i\nu^{1/4}}\partial_y$, which explains the unbounded factor $|y-y_{1\ast}|^{\gamma-2}+|y-y_{2\ast}|^{\gamma-2}$ in the bounds for $\psi$ and $u^y$. \\

     \item The bounds \eqref{thm3} confirm the ``vorticity depletion" phenomenon which predicts the asymptotic vanishing of vorticity near the critical points $y_{1\ast}$ and $y_{2\ast}$ as $t\to\infty$. In addition, we note that there is a cut-off distance $\nu^{1/4}$ to the vorticity depletion effect in contrast to the Euler equation case, which is natural since the viscous effect dominates in small scales and the dissipation it causes offsets the vorticity depletion over the viscous region. \\
     

     \item  We comment on the regularity assumptions on the initial data. By a simple dimension counting argument, we expect that \eqref{thm3} to be quite sharp. On the other hand,  in \eqref{th0.5}-\eqref{thm1}, one might be able to reduce the regularity assumption to $\omega_0\in H^{5/2}(\T\times\T_\tp)$. For the sake of conciseness we have refrained from optimization in this direction.\\

     \item It is a natural problem to also consider the problem in $\T\times[0,1]$, to include a boundary. There are essential new difficulties due to the presence of boundary layers. We refer to the recent work \cite{Almog,CWZ2} and references therein for significant results on monotonic shear flows in this direction. 
     The boundary effect is completely decoupled from the issue of critical points of $b(y)$ away from the boundary, and therefore can be treated separately. We plan to address this issue in a future work. \end{enumerate}

\subsection{Main ideas of the proof} In this section we briefly outline the main ideas of the proof. To present the key elements in the simplest settings, we assume that $k=1$ and $\alpha=0$ in \eqref{intP5}. 

The fact that we can take $\alpha$ as $-\sigma_\sharp|\nu/k|^{1/2}$ is crucial to obtain the enhanced dissipation bound in our approach, which requires that we show stability of our bounds with respect to moving from $\alpha=0$ to $\alpha=-\sigma_\sharp|\nu/k|^{1/2}$. The higher $k$ cases should, in principle, make the problem easier since the problem becomes more ``elliptic". In practice, however, the analysis becomes considerably more complicated as we need to track the dependence on $k$ of various quantities, as optimally as possible. Nonetheless these are secondary issues that we neglect here to emphasize the main points. 

To further simplify the discussion, we assume also that $\omega_{0k}$ vanishes in a neighborhood of $\{y_{1\ast}, y_{2\ast}\}$. The depletion estimates are still nontrivial and manifest as a repelling effect for the vorticity away from the critical points. The general case can be reduced to a case similar to this. 

With the help of the representation formula \eqref{intP6}, it suffices to study the ``spectral density functions" for the stream function $\psi(y,\lambda)$ and $\omega(y,\lambda)$, see \eqref{intP5}. For the sake of conciseness, we have suppressed the dependence on $\nu$ (recall that we assumed $k=1$, $\alpha=0$). Indeed, the uniform-in-viscosity inviscid damping bounds follow from integration by parts in $\lambda$ using \eqref{intP6}, and the vorticity depletion bounds follow from weighted estimates on $\psi(y,\lambda)$ which imply asymptotic vanishing of $\psi(y,\lambda)$ as $y\to \{y_{1\ast}, y_{2\ast}\}$. 

Roughly speaking, the bounds on $\psi(y,\lambda)$ are that 
\begin{equation}\label{oop1}
\begin{split}
    &|\psi(y,\lambda)|\sim \big(|\lambda-b(y)|+\min\{|\lambda-b(y_{1\ast})|, |\lambda-b(y_{2\ast})|\}+|\nu|^{1/2}\big)^{\gamma/2},\\
    & |\partial_y\psi(y,\lambda)|\sim \big(|\lambda-b(y)|+\min_{j\in\{1,2\}}\{|\lambda-b(y_{j\ast})|\}+|\nu|^{1/2}\big)^{\frac{\gamma-1}{2}}\sum_{j=1}^2\Big\langle\log\Big|\frac{\lambda-b(y)}{\lambda-b(y_{j\ast})}\Big|\Big\rangle,\\
    &|\partial^2_y\psi(y,\lambda)|\sim \big(|\lambda-b(y)|+\min_{j\in\{1,2\}}\{|\lambda-b(y_{j\ast})|\}+|\nu|^{1/2}\big)^{\frac{\gamma-2}{2}}\sum_{j=1}^2 \bigg|\frac{\lambda-b(y_{j\ast})}{\lambda-b(y)}\bigg|. 
    \end{split}
\end{equation}
We also have similar bounds on $\partial_\lambda \psi(y,\lambda)$, with the heuristic that in terms of size for $\beta\in\{1,2\}$,
\begin{equation}\label{oop2}
    \partial^\beta_\lambda\psi(y,\lambda)\sim \big(|\lambda-b(y)|+\min\{|\lambda-b(y_{1\ast})|, |\lambda-b(y_{2\ast})|\}+|\nu|^{1/2}\big)^{-\beta/2}\partial^\beta_y\psi(y,\lambda). 
\end{equation}
The bounds \eqref{oop1} show that $\psi(y,\lambda)$ becomes very small as $y$ approaches the critical points $\{y_{1\ast}, y_{2\ast}\}$ when $\lambda$ is close to one of the critical values $\{ b(y_{1*}), b(y_{2*}) \}$. These decay bounds are crucial for obtaining the explicit depletion rate for the vorticity functions, see \eqref{thm3}. The bounds on $\partial_y^2\psi(y,\lambda)$ and $\partial_\lambda^2\psi(y,\lambda)$ have to be refined near $\{y\in\T_\tp:b(y)-\lambda=0\}$ so that they can be integrated, see Lemma \ref{RDS1} below. 

The first critical step in proving the bounds \eqref{oop1}-\eqref{oop2} is to reformulate the equation \eqref{intP5} as 
\begin{equation}\label{oop3}
    (-1+\partial_y^2)\psi(y,\lambda)+\Big[\nu\partial_y^2+i(\lambda-b(y))\Big]^{-1}\big(ib''\psi(\cdot,\lambda)\big)(y)=\Big[\nu\partial_y^2+i(\lambda-b(y))\Big]^{-1}\omega_{01}(y), 
\end{equation}
and to justify the heuristic that 
\begin{equation}\label{oop4}
    \Big[\nu\partial_y^2+i(\lambda-b(y))\Big]^{-1}\approx \frac{1}{i(\lambda-b(y))+\ell_\nu(\lambda)},
\end{equation}
where 
\begin{equation}\label{oop5}
    \ell_\nu(\lambda):=\left\{\begin{array}{ll}
         \nu^{1/3},& {\rm if}\,\, \min_{j\in\{1,2\}}\{|\lambda-b(y_{j\ast})|\}\gtrsim1, \\
         \min\limits_{j\in\{1,2\}}\{|\lambda-b(y_{j\ast})|\}^{1/3}\nu^{1/3},&  {\rm if}\,\, \nu^{1/2}\ll\min_{j\in\{1,2\}}\{|\lambda-b(y_{j\ast})|\}\ll1, \\
         \nu^{1/2},&{\rm if}\,\, \min_{j\in\{1,2\}}\{|\lambda-b(y_{j\ast})|\}\lesssim\nu^{1/2}.
    \end{array}\right.
\end{equation}

To rigorously show \eqref{oop4} in various senses, we prove pointwise estimates for the kernel of $\Big[\nu\partial_y^2+i(\lambda-b(y))\Big]^{-1}$, see Proposition \ref{Airy_main}, as well as obtain refined characterizations of the (quantitative) singular behavior near $\{y\in\T_\tp:b(y)-\lambda=0\}$. The factor $\ell_\nu(\lambda)$ in \eqref{oop5} is determined by the kernel estimates.

To study \eqref{oop3}, we use \eqref{oop4} and obtain the following approximate equation for $y\in\T_\tp$, 
\begin{equation}\label{oop6}
    (-1+\partial_y^2)\psi(y,\lambda)+\frac{ib''(y)\psi(y,\lambda)}{i(\lambda-b(y))+\ell_\nu(\lambda)}\approx\frac{\omega_{01}(y)}{i(\lambda-b(y))+\ell_\nu(\lambda)}.   
\end{equation}
As in \cite{Iyer}, depending on how close $\lambda$ is to the critical values $\{b(y_{1\ast}), b(y_{2\ast})\}$, the nonlocal term $\frac{ib''(y)\psi(y,\lambda)}{i(\lambda-b(y))+\ell_\nu(\lambda)}$ plays very different roles. 
\begin{itemize}
    \item Case I: non-degenerate region $\min_{j\in\{1,2\}}\{|\lambda-b(y_{j\ast})|\}\gtrsim1$. In this case, the nonlocal term can be viewed as a compact perturbation of the main part $-1+\partial_y^2$ of the equation and can be treated using our spectral assumptions in \ref{asp2}. 

    \item Case II: the intermediate region $\nu^{1/2}\ll\min_{j\in\{1,2\}}\{|\lambda-b(y_{j\ast})|\}\lesssim 1$. In this case, as observed in \cite{JiaVortex} and \cite{Iyer}, assuming for concreteness that $\nu^{1/2}\ll|\lambda-b(y_{1\ast})|\ll 1$, we have by simple computations that for $|\lambda-b(y_{1\ast})|^{1/2}\ll|y-y_{1\ast}|\ll1$, the nonlocal term is
\begin{equation}\label{oop7}
\frac{ib''(y)}{i(\lambda-b(y))+\ell_\nu(\lambda)}\approx -\frac{2}{|y-y_{1\ast}|^2},
\end{equation}
which is both {\it critically singular} and {\it favorable}. This term can not be considered as a perturbation. In section \ref{sec:modifiedgreen} we incorporate this term into the main part (in addition to $-1+\partial_y^2$) and study the associated Green's functions which we call ``modified Green's function" as in \cite{JiaVortex,Iyer}. After taking off the non-perturbative part of the nonlocal term, the rest of the terms can be treated perturbatively, using the spectral assumptions in \ref{asp2}. 

\item Case III: the viscous region $\min_{j\in\{1,2\}}\{|\lambda-b(y_{j\ast})|\}\lesssim \nu^{1/2}$. As is expected, the viscosity dominates in small regions which in our setting limits the depletion effect to a limit length scale. The scale $\ell_\nu(\lambda)$ is  determined by the bounds we obtained for the kernel of $\Big[\nu\partial_y^2+i(\lambda-b(y))\Big]^{-1}$. On the technical level, in the viscous region, the analogous computation \eqref{oop7} is valid only for $|y-y_{1\ast}|\gg |\nu|^{1/4}$, where we note that the lower bound is independent of $\lambda$. This cutoff is ultimately responsible for the factor $\nu^{1/4}$ that appeared in \eqref{thm3}.

\end{itemize}

After identifying the main terms using the above analysis, we then show that the resulting equations can be solved, by proving suitable limiting absorption principle in weighted Sobolev spaces which are designed to capture both the regularity and vanishing of the spectral density functions $\psi(y,\lambda)$. 

In order to prove the limiting absorption principle, we reformulate the left hand side of \eqref{oop6} as  
\begin{equation}
    \label{oop8}
    \psi(y, \lambda) + T_{\Theta} \psi (y,\lambda)  
\end{equation}
where $T_{\Theta}$ is defined based on whether $\lambda$ is in the non-degenerate or intermediate/viscous region (see \eqref{lapj1} and \eqref{lapj2}). Written in this way, the limiting absorption principle can then be understood as the statement that ``$I+T_\Theta$ is invertible" in suitable weighted Sobolev spaces. To establish this, we proceed via a contradiction argument after showing that $T_\Theta$ is ``small" when the frequency is high by proving strong regularity estimates. The contradiction then follows as a consequence of the spectral assumption \ref{MaAs}. 

An essential difficulty that arises when the viscosity $\nu>0$ in contrast to the Euler equation case is that the singularities we need to handle are no longer explicit. The complexity of the singularities can be seen from the kernel of $\Big[\nu\partial_y^2+i(\lambda-b(y))\Big]^{-1}$, which although on the level of regularity is comparable with the simpler term in \eqref{oop4}, behaves nonetheless rather differently. Indeed, the kernel is highly oscillatory and does not admit a simple expression. 

The lack of explicit formula for the operator $\Big[\nu\partial_y^2+i(\lambda-b(y))\Big]^{-1}$ presents serious technical challenges especially when estimating $\partial_\lambda^2\psi(y,\lambda)$ since we need to take two derivatives in $\lambda$ in equation \eqref{oop3} and the resulting singularities involve very singular terms of the order $\Big[\nu\partial_y^2+i(\lambda-b(y))\Big]^{-3}$. To resolve this issue, instead of taking derivatives in $\lambda$, we use the {\it good derivative} $D_\lambda$ defined in \eqref{mainprop0.2}, which does not worsen the singularities. Using the observation that $\partial_y^\beta\psi(y,\lambda)$ with $\beta\in\{1,2\}$ can be understood using the equation \eqref{intP5} directly, we can then obtain sufficient control on $\partial_\lambda^\beta\psi(y,\lambda)$ for our purposes.

\subsection{Notations and conventions}\label{sec:nota}
Here we summarize some important notations and conventions used throughout the rest of the paper, for easy references. 

In the sections below we assume that $k\in\Z\cap[1,\infty)$ as the case of $k\in\Z\cap(-\infty,-1]$ follows by a complex conjugation. Define for $k\in\Z\backslash\{0\}, m\in\Z\cap[1,\infty)$, and $h(y)\in H^m(\T_\tp)$,
\begin{equation}\label{Hmk}
\|h\|^2_{H^m_k(\T_\tp)}:=\sum_{j=0}^m|k|^{2(m-j)}\|\partial_y^jh\|_{L^2(\T_\tp)}^2.
\end{equation}

Below we list some notations that we reserve for important quantities and briefly explain their roles. 
\begin{itemize}
    
    \item {\it Fundamental parameters $\lambda,\, \nu, \,k, \,\alpha, \,\epsilon,\, j$}:
    These parameters are directly given by the equations as in Proposition \ref{intP1}. $\nu\in(0,1)$ is the viscosity and will be fixed throughout the paper. $k\in\Z\cap[1,\infty)$ denotes the Fourier mode in $x$. We sometimes also use $\epsilon:=\nu/k\in(0,1)$ to simplify notations. 
We will use $\lambda\in\R$ to denote our main spectral parameter, and $\alpha\ge-\sigma_\sharp|\nu/k|^{1/2}$ as an auxiliary spectral parameter which are useful when we shift the contour of integral, as in Proposition \ref{intP1}, where $\sigma_\sharp\in(0,1)$ will be specified below. $j\in\{1,2\}$ is used to specify the the critical point $y_{j\ast}$ under consideration. \\

\item {\it Derived parameters $\delta_0,\,\delta_1(\lambda),\,\delta_2(\lambda), \, \Sigma_{j,\delta_0}, \,S^j_d$}:
These parameters are introduced to simplify notations. It follows from Assumption \ref{MaAs} that there exists a $\delta_0\in(0,1/8)$ such that
\begin{equation}\label{IntN4.0}
\begin{split}
&|y_{1\ast}-y_{2\ast}|>10\delta_0,\sup_{y\in\cup_{y_\ast\in\{y_{1\ast},y_{2\ast}\}}(y_\ast-4\delta_0,y_\ast+4\delta_0)}|b'''(y)|\delta_0<|b''(y_\ast)|/10.
\end{split}
\end{equation}
Define
\begin{equation}\label{IntN4.1}
\begin{split}
 \delta_1(\lambda):=&8\min_{y_\ast\in\{y_{1\ast},y_{2\ast}\}}\sqrt{|\lambda-b(y_\ast)|/b''(y_\ast)},\\ \delta_2(\lambda):=&\frac{1}{8}\min_{y_\ast\in\{y_{1\ast},y_{2\ast}\}}\sqrt{|\lambda-b(y_\ast)|/b''(y_\ast)},
 \end{split}
 \end{equation}
and for $j\in\{1,2\}$ the set
\begin{equation}\label{IntN4.2}
\Sigma_{j,\delta_0}:=\Big[b(y_{j\ast})-|b''(y_{j\ast})|\delta_0^2/16,\, b(y_{j\ast})+|b''(y_{j\ast})|\delta_0^2/16\Big].
\end{equation}
We introduce also for $j\in\{1,2\}, d\in(0,1/10)$ the shorthand
\begin{equation}\label{IntN4.3}
S^j_d:=[y_{j\ast}-d,y_{j\ast}+d]. 
\end{equation}

\item {\it Composite notations $\Lambda, \Theta$}:
Below we use the composite symbol $\Lambda$ and $\Theta$ to indicate dependence of certain quantities on $\Lambda:=(\lambda, \alpha, \epsilon)$ (or $\Lambda =(\lambda,\alpha,k,\nu)$ recalling that $\epsilon =\nu/k$). For example, $h(y; \Lambda)$ denotes a function $h$ of $y$ that depends on the parameters $\lambda, \alpha$ and $\epsilon$. We use $\Theta:=(\lambda,\alpha,\epsilon)$ (or $\Theta =(\lambda,\alpha,k,\nu)$) for similar purposes. The motivation for using both  $\Lambda$ and $\Theta$ with identical purpose is purely aesthetic. \\

\item {\it Weights}: To capture the precise depletion phenomena, we need to prove various weighted estimates. For this purpose, define 
\begin{equation}\label{Intd1}
\delta(\Lambda):=|\alpha|^{1/2}+C^\dagger|\epsilon|^{1/4}+8\min_{y_\ast\in\{y_{1\ast},y_{2\ast}\}}\sqrt{|\lambda-b(y_\ast)|/b''(y_\ast)}.
\end{equation}
The constant $C^\dagger\gg1$ depends only on $\delta_0$ and will be chosen below in Proposition \ref{prop:k:bounds}. 
For $\lambda\in\R$ and $j\in\{1,2\}$ with 
\begin{equation}\label{Intd2}
   \delta(\Lambda)\leq \tp/8,
\end{equation}
we define the weights for $y\in\T_\tp$,
\begin{equation}\label{Intd3}
    \varrho_j(y ; \Lambda):=|y-y_{j\ast}|+\delta(\Lambda),\quad \varrho_{j,k}(y ; \Lambda):=\varrho_j(y ; \Lambda)\wedge \frac{1}{k},\quad d_{j,k}:=\delta(\Lambda)\wedge \frac{1}{k}.
\end{equation}
In the above, $\alpha\wedge\beta$ is the minimum of two real numbers $\alpha,\beta$.

\item {\it Weighted spaces}: For $\mathfrak{M}:=(\lambda,\alpha,k,\nu,j)$ and $\sigma_1,\sigma_2\in\R$, we define the weighted space $X^{\sigma_1,\sigma_2}(\mathfrak{M})$ using the norm for any function $g$ on $\T_\tp$,
\begin{equation}\label{Intd4}
\begin{split}
\|g\|_{X^{\sigma_1,\sigma_2}(\mathfrak{M})}:=&\sum_{\beta\in\{0,1\}}(\delta(\Lambda))^{-1/2+\sigma_1}\Big\|\big[\delta(\Lambda)\wedge\frac{1}{k}\big]^{\sigma_2+\beta}\partial_y^\beta g\Big\|_{L^2(S^j_{\delta(\Lambda)})}\\
&+\sum_{\beta\in\{0,1\}}\Big\|\varrho_j^{\sigma_1}(\cdot;\Lambda)\varrho_{j,k}^{\sigma_2+\beta}(\cdot;\Lambda)\partial^\beta_yg\Big\|_{L^\infty(\T_{\tp}\backslash S^j_{\delta(\Lambda)})}.
\end{split}
\end{equation}

\end{itemize}

Lastly we remark that all the implied constants are uniform with respect to the parameters $\lambda, \alpha, k,\nu$ and only depend on the $\varkappa$ from \eqref{asp1} and the structure constant $\kappa$ from Proposition \ref{lap_main} in subsection \ref{sec:bspd}.

\subsection{Organization of the paper}
The rest of the paper is organized as follows. 

\begin{itemize}
    
    \item 
In section \ref{sec:Green} we study the Green's function and modified Green's function for the main elliptic operators from the Rayleigh equation. 

\item
In section \ref{sec:pbk} we prove pointwise bounds on the kernel of the inverse for the generalized Airy operators that appear in the Orr-Sommerfeld equations, and use them to bound the inverse of the generalized Airy operators in various weighted spaces.

\item

In section \ref{sec:mbsd} we state the limiting absorption principle and use them to derive the main bounds on the spectral density functions. 

\item

In section \ref{sec:pmt} we prove our main result, Theorem \ref{thm}.

\item
In sections \ref{sec:nondeg2} and \ref{sec:deg2} we give the proof of the limiting absorption principles assuming various bounds on $T_\Theta$. 

\item 
In section \ref{sec:T} we prove the technical propositions needed on $T_{\Theta}$ for the proofs of the limiting absorption principles.  

\end{itemize}

\section{Bounds on the Green's function and modified Green's function}\label{sec:Green}
In this section we recall and prove some useful estimates for Green's functions associated with Rayleigh equations, including the Green's function for the standard ellitpic operator $k^2-\partial_y^2$ on $\T_\tp$ and for the elliptic operator $k^2-\partial_y^2+V(y)$ on $\T_\tp$ where $V$ is a potential term coming from the nonlocal term in the Rayleigh equations. Throughout the section we assume that $k\in\Z\cap[1,\infty)$. 
\subsection{Elementary properties of the standard Green's function}\label{sub1}
Recall that the Green's function $G_k(y,z)$ solves  for $y, z\in\T_{\tp}$,
\begin{equation}\label{eq:Helmoltz}
-\frac{d^2}{dy^2}G_k(y,z)+k^2G_k(y,z)=\delta(y-z).
\end{equation}
 $G_k$ has the symmetry
\begin{equation}\label{Gk2}
G_k(y,z)=G_k(z,y), \qquad {\rm for}\,\,y, z\in\T_{\tp}.
\end{equation}
We note the following point-wise bounds for $G_k$ with $y,z\in\T_\tp$,
\begin{equation}\label{Helm1.0}
    |G_k(y,z)|+|k|^{-1}|\partial_yG_k(y,z)|\lesssim |k|^{-1}e^{-|k||y-z|},
    \end{equation}
and the integral bounds
\begin{equation}\label{Gk1.1}
\begin{split}
\sup_{y\in\T_{\tp},\beta\in\{0,1\}}\bigg[|k|^{3/2-\beta}\left\|\partial_{y,z}^{\beta}G_k(y,z)\right\|_{L^2(z\in\T_{\tp})}\bigg]\lesssim 1.
\end{split}
\end{equation}

Define $F_k(y,z)$ for $ y, z\in\T_\tp$ as
\begin{equation}\label{GkFk}
F_k(y,z):=\partial_z\partial_yG_k(y,z)-\delta(y-z),\qquad{\rm for}\,\,y,z\in\T_{\tp}.
\end{equation}
Then $F_k$ satisfies the bounds
\begin{equation}\label{Gk3.1}
\begin{split}
&\sup_{y\in\T_{\tp},\beta\in\{0,1\}}\bigg[|k|^{-1/2-\beta}\left\|\partial_{y,z}^{\beta}F_k(y,z)\right\|_{L^2(z\in\T_{\tp})}\bigg]\lesssim 1.
\end{split}
\end{equation}
These estimates are well known and follow also from the argument below in section \ref{sec:modifiedgreen}. 


\subsection{Bounds on the modified Green's function}\label{sec:modifiedgreen}
For applications below, we need to study the ``modified Green's function" $\mathcal{G}^j_{k}(y,z;\Lambda)$ defined for the parameters $j\in\{1,2\}, \lambda\in \Sigma_{j,\delta_0}$, and $\delta(\Lambda)\leq1/8$, which satisfies for $y,z\in\T_{\tp},$
\begin{equation}\label{mGk1}
(k^2-\partial_y^2)\mathcal{G}^j_{k}(y,z;\Lambda)+\frac{b''(y)}{b(y)-\lambda -i\alpha}\Big[\varphi_0\big(\frac{y-y_{j\ast}}{\delta_0}\big)-\varphi_0\big(\frac{y-y_{j\ast}}{\delta(\Lambda)}\big)\Big]\mathcal{G}^j_{k}(y,z;\Lambda)=\delta(y-z).
\end{equation}
In the above, $\varphi_0\in C_c^\infty(-2,2)$ with $\varphi_0\equiv 1$ on $[-1,1]$.

The bounds we need for the modified Green's function $\mathcal{G}^j_k(y,z;\Lambda)$ are the following.

\begin{lemma}\label{mGk50} 

Let $\mathcal{G}^{j}_{k}(y,z;\Lambda)$ be defined as in \eqref{mGk1} for $j\in\{1,2\}, \lambda\in \Sigma_{j,\delta_0}$ and $\delta(\Lambda)\leq1/8.$ 
Then we have for $\beta\in\{0,1\}$ and $y, z\in\T_{\tp}$, $\mathcal{G}^j_{k}(y,z;\Lambda)=\mathcal{G}^j_{k}(z,y;\Lambda)$, and

\begin{equation}\label{mGk51}
\varrho_{j,k}^\beta(y;\Lambda) \big|\partial_y^{\beta}\mathcal{G}^j_{k}(y,z;\Lambda)\big| \lesssim  \varrho_{j,k}(z;\Lambda) \min\Big\{e^{-|k||y-z|}, \frac{\varrho_j^2(y;\Lambda)}{\varrho_j^2(z;\Lambda)},\,\frac{\varrho_j(z;\Lambda) }{\varrho_j(y;\Lambda)}\Big\}.
\end{equation}

\end{lemma}

As a corollary of Lemma \ref{mGk50}, we also have the following bounds.

\begin{lemma}\label{mGk30}
Under the assumptions of Lemma \ref{mGk50},
 define for $y,z\in \T_{\tp}$,
 \begin{equation}\label{mGk31}
\mathcal{H}^j_k(y,z;\Lambda):=\Big[\partial_z+\varphi_0\big(\frac{y-y_{j\ast}}{10\delta(\Lambda)}\big)\partial_y\Big]\mathcal{G}^j_{k}(y,z;\Lambda).
\end{equation}
Then we have the bounds for $y\in\T_\tp, z\in S^j_{4\delta(\Lambda)}$ and $\beta\in\{0,1\}$,
\begin{equation}\label{mGk52}
\begin{split}
&\varrho_{j,k}^\beta(y;\Lambda)\big|\partial_y^\beta\mathcal{H}^j_{k}(y,z;\Lambda)\big|\lesssim \min\Big\{e^{-|k||y-z|}, \frac{\varrho_j^2(y;\Lambda)}{\varrho_j^2(z;\Lambda)},\,\frac{\varrho_j(z;\Lambda) }{\varrho_j(y;\Lambda)}\Big\},
\end{split}
\end{equation}
and for $y,z\in\T_\tp$,
\begin{equation}\label{lem:mGK30:New}
\varrho_{j,k}^\beta(y;\Lambda)\big|\partial_y^\beta(\partial_z + \partial_y) \mathcal{G}^{j}_k(y,z;\Lambda) \big|\lesssim \min \left\{e^{-(|k|/2)|y-z|}, \frac{\varrho_{j,k}(y;\Lambda)}{\varrho_{j}(z;\Lambda)}, \frac{\varrho_{j,k}(z;\Lambda)}{\varrho_{j}(y;\Lambda)}    \right\}.
\end{equation}



\end{lemma}


The rest of the subsection is devoted to the proof of Lemma \ref{mGk50} and Lemma \ref{mGk30}.

\subsubsection{Proof of Lemma \ref{mGk50}}
For simplicity, we suppress the dependence of $\mathcal{G}_k^j$ on $z,\Lambda$, $k$, $j$ which are fixed, and set for $y\in\T_{\tp}$,
\begin{equation*}
    h(y) := \mathcal{G}_k^j(y,z;\Lambda),\quad V(y):=\frac{b''(y)}{b(y)-\lambda - i\alpha}\Big[\varphi_0\big(\frac{y-y_{j\ast}}{\delta_0}\big)-\varphi_0\big(\frac{y-y_{j\ast}}{\delta(\Lambda)}\big)\Big].
\end{equation*}
We also set for $y\in\T_\tp$,
\begin{equation*}
    \varrho(y):=\varrho_{j}(y;\Lambda),\quad \varrho_k(y):=\varrho_{j,k}(y ; \Lambda),\quad y_\ast:=y_{j\ast}. 
\end{equation*}

We separate the proofs into three steps, corresponding to different configurations of $y$ and $z$. With a slight abuse of notation we shall prove the bounds \eqref{mGk51} for $y=y_1\in\T_\tp, z\in\T_\tp$ and use $y$ as a free variable. 

{\bf Step 1: local estimates near $z\in\T_\tp$}.
Multiplying $\overline{h}$ to \eqref{mGk1} and integrating over $\T_{\tp}$, we get that
\begin{equation}\label{mGk_1_1}
    \int_{\T_{\tp}}|\partial_y h(y)|^2+|k|^2|h(y)|^2 dy+\int_{\T_{\tp}}V(y)|h(y)|^2 dy = \bar{h}(z).
\end{equation}
Notice that $\Re V(y)\geq0$ and for $y\in S^j_{\delta_0}$ satisfying
$ |y-y_*|\gg \delta(\Lambda),$
that
\begin{equation}\label{mGk_1_2}
    1+\Re V(y)\gtrsim\frac{1}{\varrho^2(y)}.
\end{equation}
It follows from \eqref{mGk_1_1} that for large $C_0>1$,
\begin{equation}\label{mGk_1_3}
    \int_{\T_{\tp}}|\partial_y h(y)|^2+|k|^2|h(y)|^2 dy+\int_{y\in S^j_{\delta_0},|y-y_*|>C_0\delta(\Lambda)}\frac{1}{\varrho^2(y)}|h(y)|^2 dy\lesssim|h(z)|.
\end{equation}
Using the inequality that
\begin{equation}\label{mGk_sob}
    \|h\|_{L^\infty(J)}\lesssim\|h\|_{L^2(J_*)}|J|^{-\frac{1}{2}}+\|\partial_y h\|_{L^2(J)}|J|^\frac{1}{2},
\end{equation}
for any interval $J,\ J_*$ with $J_*\subseteq J$ and $|J_*|\gtrsim|J|$, and choosing the interval 
\begin{equation*}
    J = \Big\{y\in \T_\tp: |y-z|\leq 2C_0\min\{1/|k|,\varrho(z)\}\Big\},
\end{equation*}
we obtain that
\begin{equation}\label{mGk_1_4}
    |h(z)|\lesssim\min \{1/|k|,\varrho(z)\}.
\end{equation}
It follows from Sobolev type inequality and \eqref{mGk_1_3} that
\begin{equation}\label{mGk_1_5}
    \sup_{y,z\in\T_{\tp},\, |y-z|\leq \min\{\varrho(z), 1/|k|\}}|\mathcal{G}^j_{k}(y,z;\lambda + i\alpha)|\lesssim\min\{\varrho(z), 1/|k|\}. 
\end{equation}
Then, by \eqref{mGk1} we also obtain that
\begin{equation}
    \sup_{y,z\in\T_{\tp},\, |y-z|\leq \min\{\varrho(z), 1/|k|\}}|\partial_y\mathcal{G}^j_{k}(y,z;\Lambda)|\lesssim1.
\end{equation}


{\bf Step 2: Global bounds (I) large $k$ case.}
In this step we prove, under the assumption that
\begin{equation}\label{mGk510.01}
    1/k\leq \varrho_{j}(z;\Lambda),
\end{equation}
for $\beta\in\{0,1\}$ and $y, z\in\T_{\tp}$,
\begin{equation}\label{mGk510}
\varrho_{j,k}^\beta(y;\Lambda) \big|\partial_y^{\beta}\mathcal{G}^j_{k}(y,z;\lambda + i\alpha)\big| \lesssim  \varrho_{j,k}(z;\Lambda) e^{-|k||y-z|}. 
\end{equation}
We will remove the assumption \eqref{mGk510.01} at the end of the proof. 

Let $\varphi\in C^1_p(\T_{\tp})$, the space of piecewise $C^1$ functions, with $\varphi(z)=0$, to be chosen more precisely below. We multiply $\varphi^2\overline{h}$ to \eqref{mGk1} and integrate over $\T_{\tp}$ to get that
\begin{equation}
    \label{eq:energy1}
    \int_{\T_{\tp}} |\varphi(y) \partial_y h(y)|^2dy  + \left( k^2 + V(y) \right) |\varphi(y) h(y)|^2 dy+ \int_{\T_{\tp}} 2 \varphi'(y) \varphi(y) \overline{h}(y) \partial_y h(y) dy =0. 
\end{equation}
Taking the real part of \eqref{eq:energy1}, using Cauchy-Schwarz inequality and the fact that $\Re V(y)\geq0$ we obtain that  
\begin{equation}
    \int_{\T_{\tp}} \Big[|\varphi '(y)|^2 - |k|^2 |\varphi(y)|^2  \Big]|h(y)|^2 dy \geq 0.
\end{equation}

We assume, without loss of generality that $y_1$ lies in the positive direction to $z$ counterclockwise, and that the distance from $z$ to $y_1$ is greater than $1/k$ (otherwise the desired bound \eqref{mGk510} follows from \eqref{mGk_1_5}). We now specify $\varphi$ as follows
\begin{equation}
    \begin{split}
        &\varphi(z) = 0, \quad \varphi''(y) = 0 \quad \text{for}\ y \in \left[z, z+c/k\right]\backslash\{0\},\quad \varphi|_{y\in\{z-c/k, z+c/k\}}  = 1,\\
        &\varphi'(y) = k\varphi(y) \quad \text{for}\ y \in \left[ z+c/k, y_1 - c/k \right],\quad \varphi'(y) = 0 \quad \text{for}\ y \in \left[  y_1 - c/k, z+\mathfrak{p}/2 \right].
    \end{split}
\end{equation}

In the above $0<c<1$ is a small constant. And then extend $\varphi$ to $[z-\mathfrak{p}/2,z]$ by an even reflection across $z$. We denote $I = [y_1-c/k,y_1+c/k]$, then
$\varphi(y)\approx e^{|k||y_1-z|}$ for $y\in I$ and 
\begin{equation}
    \int_{I} |h(y)|^2 dy\lesssim e^{-2|k||y_1-z|}\int_{z-c/k}^{z+c/k}|h(y)|^2dy.
\end{equation}
By equation \eqref{mGk1} and a local energy inequality, we obtain with $I':=[y_1-c/(2k), y_1+c/(2k)]$, 
\begin{equation}\label{mGk_2_15}
    \int_{I'}\varrho^{-2}(y)|h(y)|^2+ |\partial_y h(y)|^2dy\lesssim k^2\int_I|h(y)|^2dy.
\end{equation}
By \eqref{mGk_sob}, we obtain that for $y\in I'$,
\begin{equation}\label{mGk_k_h}
    |h(y)|\lesssim \varrho_k(z)e^{-|k||y_1-z|}.
\end{equation}
The rest of the bounds in \eqref{mGk510} then follow from equation \eqref{mGk1} and \eqref{mGk_k_h}.

{\bf Step 3: Global estimates (II) large $V$ case.}
In this step we prove, under the assumption that 
\begin{equation}\label{mGk51'.001}
    1/k\ge \varrho_j(z; \Lambda),
\end{equation}
the bounds for $\beta\in\{0,1\}$ and $y, z\in\T_{\tp}$,
\begin{equation}\label{mGk51'}
\varrho_{j,k}^\beta(y;\Lambda) \big|\partial_y^{\beta}\mathcal{G}^j_{k}(y,z;\Lambda)\big| \lesssim  \varrho_{j,k}(z;\Lambda) \min\Big\{\frac{\varrho_j^2(y;\Lambda)}{\varrho_j^2(z;\Lambda)},\,\frac{\varrho_j(z;\Lambda) }{\varrho_j(y;\Lambda)}\Big\}.
\end{equation}
We will remove the assumption \eqref{mGk51'.001} at the end of the proof.

 Without loss of generality, we assume the critical point $y_{j*}=0$, and consider only the main case $y_1,z\in [0,\delta_0]$ and $z>\delta(\Lambda)$. In addition, we assume that $y_1>z$, the other case following from the symmetry of the modified Green's function.  
 
Multiplying $\varphi^2\overline{h}$ to \eqref{mGk1} and integrating over $\T_{\tp}$, we have
\begin{equation}\label{mGk_3_1}
    \int_{\T_{\tp}} |\varphi(y) \partial_y h(y)|^2  + \left( k^2 + V(y) \right) |\varphi(y) h(y)|^2 dy+ \int_{\T_{\tp}} 2 \varphi'(y) \varphi(y) \overline{h}(y) \partial_y h(y)\,dy =0. 
\end{equation}
Set for $y\in\T_\tp$, 
\begin{align}\label{mGk_3_2}
    h(y) = r(y) h^*(y),
\end{align}
where $r(y)\in C^1_p(\T_\tp)$ is positive and is chosen to satisfy the following properties 
\begin{equation}
    r(y) = \left\{\begin{array}{lll}
        |y|^{1/2}&\text{on}\ [-\delta_0,\delta_0]\backslash J, \ J:=[-z,z],\\
    (\delta_0)^{\frac{1}{2}} &\text{on}\, \ \T_\tp\backslash [-\delta_0,\delta_0].
    \end{array}\right.
\end{equation}

From \eqref{mGk_3_1} and \eqref{mGk_3_2} we obtain that
\begin{equation}\label{eq:entanglement}
\begin{split}
     &\int_{\T_\tp} \frac{1}{r^2(y)} \Big[ (r^2(y)\varphi'(y))^2  - \left( r^2(y) |r'(y)|^2 + r^4(y) \Re V(y)   \right)  \varphi^2(y) \Big] |h^\ast(y)|^2  dy\\
     &\geq -\int_{\T_\tp} \varphi^2(y)|h^\ast(y)|^2  (r'(y)r(y))' dy.
     \end{split}
\end{equation}
and therefore for $\varphi$ satisfying $\varphi\equiv0$ on $[-z,z]$,
\begin{equation}\label{eq:entangle2}
    \int_{\T_{\tp}\backslash J}\frac{1}{r^2(y)} \Big[ (r^2(y)\varphi'(y))^2  - \left( r^2(y) |r'(y)|^2 + r^4(y) \Re V(y)   \right)  \varphi^2(y) \Big] |h^*(y)|^2 dy\geq0.
\end{equation}

We notice the pointwise bounds for $y\in[-\delta_0,\delta_0]\backslash J$ and some $C_2>1$ that
\begin{equation}
   m^2(y):= r^2(y) |r'(y)|^2 + r^4(y) \Re V(y) \geq\max\left\{\frac{1}{4} ,\frac{9}{4}-C_2\frac{\delta(\Lambda)^2}{y^2}-C_2|y|\right\}.
\end{equation}
It is clear that there exists a $y_1'\in\T_\tp$ such that 
\begin{equation}\label{eq:entangle2.1}
    \int_{[z,y_1-\varrho(y_1)/8]}|m(y)|dy=\int_{[y_1'+\varrho(y_1')/8,-z]}|m(y)|dy. 
\end{equation}
In the above, we identify $\T_\tp$ with the circle of circumference $\tp$, and for $a,b\in\T_\tp$ where $b$ is in the positive direction of $a$ (counterclockwise), $[a,b]$ denotes the interval from $a$ to $b$.

We now specify $\varphi\in C^1_p(\T_\tp)$ as the following 
\begin{align}\label{entan2.10}
    &\varphi|_{[-z,z]} = 0, \quad  \varphi ''(y) = 0, \quad y \in [-z-c\varrho(z), z + c\varrho(z)]\backslash[-z,z],\quad  \varphi|_{\{-z-c\varrho(z),z+c\varrho(z)\}} = 1;  \notag\\
    &r(y)\varphi'(y)  = -\Big[ r^2(y) |r'(y)|^2 + r^4(y) \Re V(y)\Big]^{1/2} \varphi(y), \quad y \in [y_1'+\varrho(y_1')/8, -z]; \\
    &r(y)\varphi'(y)  = \Big[ r^2(y) |r'(y)|^2 + r^4(y) \Re V(y)\Big]^{1/2} \varphi(y), \quad y \in [z,y_1-\varrho(y_1)/8]; \\
    & \varphi'(y)=0,\ y\in[y_1-\varrho(y_1)/8,y_1'+\varrho(y_1')/8].
\end{align}
From the definition of $\varphi$, we obtain the bound for $y\in[y_1-\varrho(y)/8,y_1+\varrho(y)/8]$ for a small $c>0$,
\begin{equation}\label{eq:entangle2.2}
    \varphi(y)\ge c\Big[\varrho(y_1)/\varrho(z)\Big]^{3/2}. 
\end{equation}
Then the bound \eqref{eq:entangle2} and the definiton \eqref{entan2.10} imply that 
\begin{equation}\label{entan2.3}
\begin{split}
  &\Big[\varrho(z)/\varrho(y_1)\Big]^3  \int_{[y_1-\varrho(y_1)/8,y_1+\varrho(y_1)/8]}\frac{1}{r^2(y)}|h^\ast(y)|^2dy\\
  &\lesssim \int_{[z-c\varrho(z), z + c\varrho(z)]} \varrho^{-2}(z)|h(y)|^2dy\lesssim \varrho(z).
  \end{split}
\end{equation}
The desired bounds \eqref{mGk51'} then follow from \eqref{entan2.3}, equation \eqref{mGk1} and local energy inequality. We omit the routine details. 

{\bf Step 4: Removing assumptions \eqref{mGk510.01} and \eqref{mGk51'.001}. } We consider first the assumption \eqref{mGk510.01}. The key is to establish the bound 
\begin{equation}\label{step4.1}
    k\int_{z-c/k}^{z+c/k}h^2(y)dy\lesssim \varrho^2(z),
\end{equation}
in the case that $1/k\ge \varrho(z)$. We can use the bound \eqref{mGk51'} and then \eqref{step4.1} follows. 

We now turn to the assumption \eqref{mGk51'.001}. The key is to prove the bound 
\begin{equation}\label{step4.2}
    \int_{[z-c\varrho(z), z + c\varrho(z)]} \varrho^{-1}(z)|h(y)|^2dy\lesssim 1/k^2,
\end{equation}
in the case that $1/k\leq\varrho(z)$. In this case, we can use the bound \eqref{mGk510} and \eqref{step4.2} follows. 

Lemma \ref{mGk50} is now proved. 

\subsubsection{Proof of Lemma \ref{mGk30}}
This proof of \eqref{mGk52} is almost the same as the one in \cite{Iyer}. We briefly outline the main ideas below.

Denote for $y\in\T_{\tp}$,
\begin{equation}
    \varphi^{\dag}(y):=\varphi_0 \Big(\frac{y-y_{j*}}{10\delta(\Lambda)} \Big).
\end{equation}
By direct computation, we see that $\mathcal{H}_k^j(y,z;\Lambda)$ satisfies for $y\in\T_\tp$,
\begin{equation}\label{eq:Hpoisson}
\begin{split}
    &(k^2-\partial_y^2)\mathcal{H}_k^j(y,z;\Lambda)+V(y)\mathcal{H}_k^j(y,z;\Lambda) \\
    &=-2\partial_y\varphi^{\dag}(y)\partial_y^2\mathcal{G}^j_k(y,z;\Lambda)-\partial_y^2\varphi^{\dag}(y)\partial_y\mathcal{G}^j_k(y,z;\Lambda)-\varphi^{\dag}(y)\partial_yV(y)\mathcal{G}^j_k(y,z;\Lambda)\\
    &:=F(y,z;\Lambda). 
\end{split}
\end{equation}
Notice the inequality that for $z\in S^j_{4\delta}, y\in\T_\tp$, 
\begin{equation*}
    |F(y,z;\Lambda)|\lesssim\max\Big\{|k|(\delta(\Lambda))^{-1}, (\delta(\Lambda))^{-2}\Big\}e^{-|k||y-z|},\quad\ \text{Supp}\,F\subseteq[-20\delta(\Lambda),20\delta(\Lambda)].
\end{equation*}
Then the desired bound \eqref{mGk52} follows from \eqref{mGk51}. 

We now turn to the proof of \eqref{lem:mGK30:New}. Notice that $(\partial_y+\partial_z)\mathcal{G}^j_k(y,z;\Lambda)$ satisfies the equation for $y,z\in\T_\tp$,
\begin{equation}\label{entang400}
    \Big[k^2-\partial_y^2+V(y)\Big](\partial_y+\partial_z)\mathcal{G}^j_k(y,z;\Lambda)=-\partial_yV(y)\mathcal{G}^j_k(y,z;\Lambda). 
\end{equation}
The desired bounds \eqref{lem:mGK30:New} then follow from \eqref{mGk51}. 




\section{Pointwise bounds on the kernel of generalized Airy operators and applications}\label{sec:pbk}

In this section we consider the generalized Airy operator $\epsilon \partial_y^2 -\alpha + i (\lambda-b(y)) $, where $b(y)$ satisfies Assumption \ref{asp1}, $\epsilon\in(0,1), \lambda\in\R$ and $\alpha\ge -\sigma_0 \epsilon^{1/2}$ with a sufficiently small $\sigma_0\in(0,1)$. Our main goal is to prove pointwise bounds for the kernel of the generalized Airy operator and then obtain estimates on various operators involving the generalized Airy operator. 

\subsection{Pointwise estimates on the kernel of generalized Airy operators}
Depending on the value of $\lambda\in\R$, we divide into three regions.
\begin{enumerate}
    \item  the non-degenerate region:  $\min_{j\in\{1,2\}}\{|b(y_{j*})-\lambda|\} \gtrsim1$;
    \item the intermediate region: $\epsilon^{\frac{1}{2}}\ll \min_{j\in\{1,2\}}\{|b(y_{j*})-\lambda|\} \ll 1 $;
    \item the viscous region:  $\min_{j\in\{1,2\}}\{|b(y_{j*})-\lambda|\} \lesssim \epsilon^{\frac{1}{2}} $. 
\end{enumerate}

\begin{proposition}\label{Airy_main}
   Assume that  $b(y)$ satisfies Assumption \ref{asp1} and that $\epsilon\in(0,1)$. Then there exist $\sigma_0,c_0\in(0,1)$ such that the following statement holds. 
   
  For $\lambda\in\R$, $y,z\in\T_\tp$, denote $\Lambda:=(\epsilon,\lambda,\alpha)$ and let $k(y,z; \Lambda)$ be the fundamental solution to the generalized Airy operator $\epsilon \partial_y^2 -\alpha + i (\lambda-b(y)) $ on $\T_\tp$, where $\alpha\ge -\sigma_0\epsilon^{1/2}$. More precisely, for $y,z\in\T_\tp$, in the sense of distributions,
    \begin{equation}\label{eq:airy}
        \epsilon \partial_y^2 k(y,z; \Lambda)  -\alpha k(y,z; \Lambda) + i (\lambda-b(y)) k(y,z; \Lambda) = \delta(y -z).
    \end{equation}
    We have the following bounds for $k(y,z; \Lambda)$ with $\lambda\in\R$, $y,z\in\T_\tp$. 
    \begin{enumerate}
        \item if $\min_{j\in\{1,2\}}\{|b(y_{j*})-\lambda|\} \gtrsim1$, then
        \begin{equation}
        \label{eq:kernel:bounds:3}
        \begin{split}
            |k(y,z; \Lambda)| \lesssim   &
             \frac{\epsilon^{-\frac{2}{3}}}{\lb \ept\alpha , \ept( b(z) - \lambda) \rb^{\frac{1}{2}} }\\
             &\times \exp\left[-c_0 \lb \ept\alpha , \ept( b(y)-\lambda) , \ept( b(z)-\lambda)  \rb^{\frac{1}{2}} |y -z|\ept \right];
             \end{split}
        \end{equation}
        \item if $\epsilon^{\frac{1}{2}}\ll  \min_{j\in\{1,2\}}\{|b(y_{j*})-\lambda|\} \ll 1 $, then
        \begin{equation}
        \label{eq:kernel:bounds:2}
        \begin{split}
           |k(y,z; \Lambda)|\lesssim   &\frac{\epsilon^{-\frac{2 + \beta}{3}}}{\lb \epbn\alpha,  \epbn(b( z)-\lambda) \rb^{\frac{1}{2}} } \\
           &\times \exp\left[-c_0 \lb  \epbn\alpha, (b( y)-\lambda) \epbn, (b( z)-\lambda)\epbn \rb^{\frac{1}{2}} |y -z|\epbv\right];
            \end{split}
        \end{equation}
        where $\beta$ is determined by $\min_{i\in\{1,2\}}\{|b(y_{i*})-\lambda|\} \approx \epsilon^{2 \beta}$, $\beta \in [0,1/4]$.
        \item if $\min_{j\in\{1,2\}}\{|b(y_{j*})-\lambda|\} \lesssim  \epsilon^{\frac{1}{2}} $, then
        \begin{equation}
        \label{eq:kernel:bounds:1}
        \begin{split}
            |k(y,z; \Lambda)| \lesssim &  
            \frac{\epsilon^{-\frac{3}{4}}}{\lb  \epsilon^{-\frac{1}{2}}\alpha,  \epsilon^{-\frac{1}{2}}(b( z)-\lambda)) \rb^{\frac{1}{2}} } \\
            &\times\exp\left[-c_0 \lb \epsilon^{-\frac{1}{2}}\alpha, \epsilon^{-\frac{1}{2}}(b( y)-\lambda), \epsilon^{-\frac{1}{2}}(b( z)-\lambda)  \rb^{\frac{1}{2}} |y -z|\epsilon^{-\frac{1}{4}} \right] .
            \end{split}
        \end{equation}
    \end{enumerate}
\end{proposition}

The rest of the subsection is devoted to the proof of Proposition \ref{Airy_main}. 
Without loss of generality, below we only consider the case $\min_{j\in\{1,2\}}\{|b(y_{j*})-\lambda|\} = |b(y_{1*})-\lambda|$, and denote $y_*=y_{1*}$ for simplicity. Through the rest of this section, we also denote $k(y,z)=k(y,z; \Lambda)$ for the ease of notations.

\subsubsection{Global energy estimates and bounds at $z$}
We first use the global energy estimates to obtain bounds on $k$ at $y=z$. 
\begin{lemma}\label{Airy_kzz}
For all $z\in\T_\tp$, we have the following bounds.
\begin{enumerate}
    \item if $|b(y_*)-\lambda| \gtrsim1$,
    \begin{equation}\label{eq:diagonal:bounds:3}
        |k(z,z)|\lesssim \frac{\epsilon^{-\frac{2}{3}}}{\lb \ept\alpha,  \ept(b( z)-\lambda) \rb^{\frac{1}{2}} }.
    \end{equation}
    \item if $\epsilon^{\frac{1}{2}}\ll |b(y_*)-\lambda| \lesssim 1 $, 
    \begin{equation}\label{eq:diagonal:bounds:2}
        |k(z,z)|\lesssim \frac{\epsilon^{-\frac{2 + \beta}{3}}}{\lb \epbn\alpha,  \epbn(b( z)-\lambda) \rb^{\frac{1}{2}} }.
    \end{equation}
    \item if $|b(y_*)-\lambda| \lesssim \epsilon^{\frac{1}{2}}$, 
    \begin{equation}
        \label{eq:diagonal:bounds:1}
        |k(z,z)|\lesssim \frac{\epsilon^{-\frac{3}{4}}}{ \lb   \epsilon^{-\frac{1}{2}}\alpha ,  \epsilon^{-\frac{1}{2}}(b(z)-\lambda)\rb^{\frac{1}{2}} }.
    \end{equation}
\end{enumerate}
\end{lemma}

\begin{proof}
We first prove the degenerate case (2). 
Multiplying \eqref{eq:airy} with $\overline{k(y,z)}$ and $(\lambda-b(y))\overline{k(y,z)}$ respectively, and integrating over $\T_\tp$, we get that
\begin{equation}\label{airy*kbar:2}
    -\epsilon\int_{\T_\tp}|\partial_yk(y,z)|^2 dy-\alpha\int_{\T_\tp}|k(y,z)|^2dy+i\int_{\T_\tp}(\lambda-b(y))|k(y,z)|^2 dy = \overline{k(z,z)};
\end{equation}
\begin{equation}\label{airy*vkbar:2}
\begin{split}
     -\epsilon\int_{\T_\tp}(\lambda-b(y))|\partial_y k(y,z)|^2dy+\epsilon\int_{\T_\tp}b^\prime(y)\partial_yk(y,z)\overline{k(y,z)}dy-\alpha\int_{\T_\tp}(\lambda-b(y))|k(y,z)|^2\\+i\int_{\T_\tp}(\lambda-b(y))^2|k(y,z)|^2 dy=(\lambda-b(z))\overline{k(z,z)}.
\end{split}
\end{equation}
It follows from \eqref{airy*kbar:2} that 
\begin{equation}\label{dybound:2}
    \epsilon\int_{\T_\tp}|\partial_yk(y,z)|^2 dy+\alpha\int_{\T_\tp}|k(y,z)|^2 dy\leq|k(z,z)|.
\end{equation}
Similarly, it follows from \eqref{airy*vkbar:2} that
\begin{equation}\label{vbound1:2}
    \int_{\T_\tp}(\lambda-b(y))^2|k(y,z)|^2 dy\leq|\lambda-b(z)||k(z,z)|+\epsilon\int_{\T_\tp}|b^\prime(y)||\partial_yk(y,z)||k(y,z)|dy
\end{equation}
Using the spectral gap inequalities  
\begin{equation}\label{spec_gap:2}
\int_{\T_\tp}|b^\prime (y)|^2|k(y,z)|^2 dy\lesssim \epsilon^{-\frac{2-2\beta}{3}}\int_{\T_\tp}(\lambda-b(y))^2|k(y,z)|^2dy+\epsilon^{\frac{2+4\beta}{3}}\int_{\T_\tp}|\partial_yk(y,z)|^2 dy,
\end{equation}
and
\begin{equation}\label{spec_gap:2.1}
    \int_{\T_\tp}|k(y,z)|^2 dy\lesssim \epsilon^{-\frac{2+4\beta}{3}}\int_{\T_\tp}(\lambda-b(y))^2|k(y,z)|^2dy+\epsilon^{\frac{2-2\beta}{3}}\int_{\T_\tp}|\partial_yk(y,z)|^2 dy,
\end{equation}
we get that for any $\tau\in(0,1)$ and a corresponding $C(\tau)\in(1,\infty)$,
\begin{equation}\label{vbound2:2}
\begin{split}
    &\epsilon\int_{\T_\tp}|b^\prime(y)|\cdot|\partial_yk(y,z)|\cdot|k(y,z)|dy\\
    &\leq \epsilon^{\frac{2-2\beta}{3}}\tau\int_{\T_\tp}|b^\prime(y)|^2|k(y,z)|^2 dy+\epsilon^{\frac{4+2\beta}{3}}\frac{1}{\tau}\int_{\T_\tp}|\partial_yk(y,z)|^2 dy\\
    &\lesssim\tau\int_{\T_\tp}(b( y)-\lambda)^2|k(y,z)|^2dy+\epsilon^{\frac{4+2\beta}{3}} C(\tau)\int_{\T_\tp}|\partial_yk(y,z)|^2 dy,
\end{split}
\end{equation}
and that for sufficiently small $\sigma_0\in(0,1)$ and $\alpha\ge-\sigma_0\epsilon^{\frac{1+2\beta}{3}}$,
\begin{equation}\label{vbound2:2J}
     \epsilon\int_{\T_\tp}|\partial_yk(y,z)|^2 dy+|\alpha|\int_{\T_\tp}|k(y,z)|^2 dy\lesssim|k(z,z)|+\sigma_0\epsilon^{-\frac{1+2\beta}{3}}\int_{\T_\tp}(b(y) - \lambda)^2|k(y,z)|^2 dy.
\end{equation}
Thus combining \eqref{vbound1:2} and \eqref{vbound2:2}-\eqref{vbound2:2J}, we obtain that 
\begin{equation}\label{vbound3:2}
    \int_{\T_\tp}(b(y) - \lambda)^2|k(y,z)|^2 dy\lesssim|b(z)-\lambda||k(z,z)|+\epbp|k(z,z)|,
\end{equation}
which further implies that 
\begin{equation}\label{vbound4:2}
    \epbn\int_{\T_\tp}(b(y)-\lambda)^2|k(y,z)|^2 dy\lesssim\lb\epbn(b(z)-\lambda)\rb|k(z,z)|.
\end{equation}

Then, in view of \eqref{vbound2:2J} and \eqref{vbound3:2}, we can get the desired bounds \eqref{eq:diagonal:bounds:2} using the inequality
\begin{equation}\label{maxbound_form:2}
    \|f\|_{L^\infty(I)}\lesssim\|f\|_{L^2(I)}|I|^{-\frac{1}{2}}+\|f^\prime\|_{L^2(I)}|I|^{\frac{1}{2}}
\end{equation}
for an interval $I$, as follows.

If $\alpha\lesssim \epbp$, we choose 
$$
I=\left\{y\in\T_\tp: \epbv|y-z|\leq\lb\epbn(\lambda-b(z))\rb^{-\frac{1}{2}}\right\},\quad |I| \approx \epsilon^{\frac{1-\beta}{3}}\lb\epbn(\lambda-b(z))\rb^{-\frac{1}{2}}.
$$

If $\alpha\gg \epbp$, we choose $I = \left\{y\in\T_\tp: \epbv|y-z|\leq\lb\epbn\alpha\rb^{-\frac{1}{2}}\right\}$, $|I| \approx \epsilon^{\frac{1-\beta}{3}}\lb\epbn\alpha\rb^{-\frac{1}{2}}$. 
This concludes the proof of case (2) of Lemma \ref{Airy_kzz}. 

The cases (1) and (3) follow from similar arguments.
Indeed the non-degenerate case (1), we obtain \eqref{eq:diagonal:bounds:1} by just applying $\beta = 0$ into \eqref{eq:diagonal:bounds:2}, if additionally $|\lambda|\lesssim1$. If $|\lambda|\gg1$, then there is no need to use the spectral gap inequality \eqref{spec_gap:2}, and \eqref{vbound3:2} with $\beta=0$ follows directly from \eqref{vbound1:2}. 

For the viscous case, \eqref{eq:diagonal:bounds:3} follows from the same arguments as in the degenerate case  if we take $\beta = \frac14$. 
The proof of Lemma \ref{Airy_kzz} is now complete.
\end{proof}

\subsubsection{The entanglement inequality}
The following lemma will be used frequently in our proof of Proposition \ref{Airy_main}. 

\begin{lemma}\label{entangle}
For $\beta\in[0,1/4]$ and non-negative cutoff function $\varphi(y)\in C^1(\T_\tp)$ with $\varphi(z)=0$,  there exists $c_0\in(0,1)$ independent of $\beta$ and $\varphi$ such that
\begin{equation}\label{eq:entangle}
    \int_{\T_\tp}\left[|\varphi^\prime|^2-c_0^2\epsilon^{-\frac{2-2\beta}{3}}\lb\epbn\alpha,\epbn(b(y)-\lambda)\rb|\varphi|^2\right]|k(y,z)|^2 dy\geq0.
\end{equation}
\end{lemma}

\begin{remark}
    The statement of our entanglement inequality is the same as that in \cite{JiaUM}. However, our proof is much more flexible since our method here does not require the separation into regions where $\lambda-b(y)$ keeps its sign. In particular, we no longer need to use the local to global energy inequality, see section 3.2 of \cite{JiaUM}, which is essential in \cite{JiaUM} but is somewhat cumbersome.
\end{remark}

\begin{proof}
We provide the detailed proof only for the most complicated case when $\lambda\in b(\T_\tp)$. The proof of the case $\lambda\in\R\backslash b(\T_\tp)$ is similar and easier.

We first introduce a real-valued test function $\phi(y)$ which is a standard regularization of the jump function $\text{sgn}(b(y)-\lambda)$ on the scale of $\epsilon^{\frac{1-\beta}{3}}$, the natural length scale at $y\in\T_\tp$ with $\lambda=b(y)$, and which satisfies the following properties:
\begin{enumerate}
    \item $|\phi^\prime(y)|\lesssim \epsilon^{-\frac{1-\beta}{3}}$;
    \item $|\phi(y)-\text{sgn}(b(y)-\lambda)||b(y)-\lambda|\lesssim\epsilon^{\frac{1+2\beta}{3}}$.
\end{enumerate}

Multiplying $\varphi^2(y)\overline{k(y,z)}$ to \eqref{Airy_main} and integrating over $\T_\tp$, we have
\begin{equation}\label{airy*phi2kba:2}
\begin{split}
     -\epsilon\int_{\T_\tp}\varphi^2(y)|\partial_y k(y,z)|^2 dy-\epsilon\int_{\T_\tp}2\varphi^\prime(y)\varphi(y)\overline{k(y,z)}\partial_yk(y,z)dy-\alpha\int_{\T_\tp}\varphi^2(y)|k(y,z)|^2 dy\\
    +i\int_{\T_\tp}(b(y)-\lambda)\varphi^2(y)|k(y,z)|^2 dy = 0.
\end{split}
\end{equation}
By taking the real part and applying the H\"older inequality, we get that
\begin{equation}\label{phikboundw/a:2}
    \alpha\epsilon^{-1}\int_{\T_\tp}\varphi^2(y)|k(y,z)|^2 dy+\int_{\T_\tp}\varphi^2(y)|\partial_y k(y,z)|^2 dy\lesssim\int_{\T_\tp}|\varphi^\prime(y)|^2|k(y,z)|^2dy.
\end{equation}
We also multiply $\varphi^2(y)\phi(y)\overline{k(y,z)}$ to \eqref{Airy_main} and integrate over $\T_\tp$ to get that
\begin{equation}\label{airy*phi2sgnkbar:2}
\begin{split}
    -\epsilon\int_{\T_\tp}\varphi^2(y)\phi(y)|\partial_yk(y,z)|^2dy-\epsilon\int_{\T_\tp}\left(2\varphi(y)\varphi^\prime(y)\phi(y)+\varphi^2(y)\phi^\prime(y)\right)\overline{k(y,z)}\partial_yk(y,z)dy\\-\alpha\int_{\T_\tp}\varphi^2(y)\phi(y)|k(y,z)|^2 dy+i\int_{\T_\tp}(b(y)-\lambda)\varphi^2(y)\phi(y)|k(y,z)|^2 dy = 0.
\end{split}
\end{equation}
Taking the imaginary part, we get that
\begin{equation}
\begin{split}
   & \int_{\T_\tp}(b(y)-\lambda)\varphi^2(y)\phi(y)|k(y,z)|^2 dy\\
&\leq\epsilon\int_{\T_\tp}\Big[2|\varphi(y)||\varphi^\prime(y)||\phi(y)|+\varphi^2(y)|\phi^\prime(y)|\Big]\big|k(y,z)\big||\partial_y k(y,z)|dy
    \end{split}
\end{equation}
Thus by the construction of the function $\phi(y)$, we conclude that
\begin{equation}\label{phikboundw/v:2}
\begin{split}
   & \int_{\T_\tp}|b(y)-\lambda|\varphi^2(y) |k(y,z)|^2 dy \\
    &= \int_{\T_\tp}(b(y)-\lambda)\Big[\text{sgn}(b(y)-\lambda)-\phi(y)\Big]\varphi^2(y)|k(y,z)|^2dy+\int_{\T_\tp}(b(y)-\lambda)\phi(y)\varphi^2|k(y,z)|^2dy\\
&\leq\int_{\T_\tp}\epbp\varphi^2(y)|k(y,z)|^2dy+\epsilon\int_{\T_\tp}\Big[2|\varphi(y)||\varphi^\prime(y)||\phi(y)|+\varphi^2(y)|\phi^\prime(y)|\Big]|k(y,z)||\partial_y k(y,z)|dy.
\end{split}
\end{equation}
By the spectral gap inequality that for any $\tau\in(0,1)$, a corresponding $C(\tau)\in(1,\infty)$, and any  $h\in H^1(\T_\tp)$,
\begin{equation}\label{spec_gap_vphik:2}
\int_{\T_\tp}|h(y)|^2dy\leq\tau \epbn\int_{\T_\tp}|b(y)-\lambda||h(y)|^2dy+C(\tau)\epsilon^{\frac{2-2\beta}{3}}\int_{\T_\tp}|\partial_y h(y)|^2dy,
\end{equation}
 together with \eqref{phikboundw/v:2}, we get that 
\begin{equation}\label{spec_gap_vphik:2.000}
\begin{split}
    &\int_{\T_\tp}|b(y)-\lambda|\varphi^2(y) |k(y,z)|^2 dy\\
    &\lesssim \epsilon\int_{\T_\tp}|\varphi^\prime(y)|^2|k(y,z)|^2+\epsilon\int_{\T_\tp}\varphi^2(y)|\partial_yk(y,z)|^2dy.
\end{split}
\end{equation}
Using the spectral gap inequality \eqref{spec_gap_vphik:2}, the bounds \eqref{phikboundw/a:2} when $\alpha\ge-\sigma_0\epsilon^{\frac{1+2\beta}{3}}$ with a sufficiently small $\sigma_0\in(0,1)$, and \eqref{spec_gap_vphik:2.000}, we then obtain the desired inequality \eqref{eq:entangle}. 

\end{proof}

\subsubsection{Point-wise bounds for the kernel}
Finally, we prove our main conclusion Proposition \ref{Airy_main}.

We introduce the length scale $Q(x)$ for $x\in\T_\tp$ which will be used frequently in our argument. Notice that there is a slight abuse of notation and $x$ in this section is not to be confused with the $x$ variable for the non-square torus $\T\times\T_\tp$.

We first define for $x\in\T_\tp, \beta\in[0,1/4]$,
\begin{equation*}
    Q(x;\beta) :=\frac{\epsilon^{\frac{1-\beta}{3}}}{\lb\epbn\alpha,\epbn(b(x)-\lambda)\rb^{\frac{1}{2}}}.
\end{equation*}
Since $\beta$, determined by $\lambda\in b(\T_\tp)$, is fixed, we also use the following shortened notation $Q(x)$. In the non-degenerate case, $Q(x) = Q(x;0)$; in the degenerate case, $Q(x) = Q(x;\beta)$; and in the viscous case, $Q(x) = Q(x;\frac{1}{4})$.

When $|y-z|\lesssim Q(z)$, notice here $\{|y-z|\leq Q(z)\}\subset I$ where $I$ is taken from the proof of Lemma \ref{Airy_kzz}. Thus by Lemma \ref{Airy_kzz}, 
\begin{equation}
    |k(y,z)|\lesssim\epsilon^{-1}Q(z)\cdot e^{0},
\end{equation}
which satisfies \eqref{eq:kernel:bounds:3}-\eqref{eq:kernel:bounds:1} naturally. 

When $Q(z)\ll|z-y|$, we will apply the entanglement inequality to get the bound. Without loss of generality, we assume $y\in\T_\tp$ is in the positive direction to $z$, counter-clockwisely. 

Let $\Gamma_{x_1,x_2}$ be the arc on $\T_\tp$ from $x_1$ to $x_2$ in the counter-clockwise direction. We assume that 
\begin{equation}\label{kernelB1}
    \int_{\Gamma_{z+Q(z),y}}Q^{-1}(x)\,dx\ge  \int_{\Gamma_{y,z-Q(z)}}Q^{-1}(x)\,dx.
\end{equation}
Then we can find $y_\ast\in[z+Q(z),y]$ such that 
\begin{equation}\label{kernelB2}
     m_\ast:=c_0\int_{\Gamma_{z+Q(z),y_\ast}}Q^{-1}(x)\,dx=  c_0\int_{\Gamma_{y+Q(y),z-Q(z)}}Q^{-1}(x)\,dx. 
\end{equation}
Consider the points 
\begin{align}
\label{eq:entangle:endpoints1:2}
A_1 &= z - Q(z), \notag \\
A_2 &= z + Q(z), \notag \\
A_3 &=  y_\ast, \notag \\
A_4 &= y+Q(y).
\end{align}
 Using \eqref{eq:entangle:endpoints1:2} we choose $\varphi$ for the entanglement inequality \eqref{eq:entangle} satisfying the following conditions:
\begin{align}
\label{eq:entanglefunction1:2}
    & \varphi'(x) =   c_0 Q^{-1}(x) \varphi(x) \quad
    \text{ for }  x \in [A_2, A_3]; \notag\\
    & \varphi'(x) =   -c_0 Q^{-1}(x) \varphi(x) \quad
    \text{ for }  x \in [A_4, A_1], \quad \varphi(A_1)=\varphi(A_2) = 1; \notag \\
    &\varphi''(x) \equiv 0, \quad\text{ for } x \in [A_3, A_4] \cup [A_1,z)\cup(z, A_2],\quad \varphi(z)=0.
\end{align}

We have from  \eqref{eq:entanglefunction1:2} that there exists $c_1 \in (0,1)$, such that
\begin{equation}
\label{eq:A_3:value:2}
   \varphi(A_3)=e^{m_\ast} \gtrsim \exp\left[ c_1 \Big\langle\epbn\alpha,\epbn(b(y)-\lambda), \epbn(b(z)-\lambda)\Big\rangle^{\frac{1}{2}}|y-z|\epsilon^{-\frac{1-\beta}{3}}  \right]. 
    \end{equation}
Using the entanglement inequality \eqref{eq:entangle} we obtain that 
\begin{equation}
    \int_{y}^{y+Q(y)}e^{2m_\ast}Q^{-2}(x)|k(x,z)|^2dx\lesssim \int_{z-Q(z)}^{z+Q(z)}Q^{-2}(x)|k(x,z)|^2\,dx,
\end{equation}
which together with equation \eqref{eq:airy} implies the desired point-wise estimates. The proof of Proposition \ref{Airy_main} is now complete. 

\subsection{Applications of the kernel estimates}
In this section we prove several bounds concerning the operator 
\begin{equation}\label{eqAt}
    A_\Theta:=\frac{\nu}{k}\partial_y^2+i(\lambda-b(y))-\alpha.
\end{equation}
In the above, we use the same assumptions and notations as in proposition \ref{intP1} and recall that $k\in\Z\cap[1,\infty)$. We also used the composite symbol $\Theta:=(k,\alpha,\nu,\lambda)$ for the ease of notations. 

Assume that $j\in\{1,2\}, \gamma\in[7/4,2)$. 
 For $\lambda\in \Sigma_{j,\delta_0}, y\in \T_{\tp}$ and $\epsilon:=\nu/k\in(0,1/8)$, we denote $\Lambda:=(\lambda,\alpha,\epsilon)$ and recall from \eqref{Intd1} that
\begin{equation}\label{DeC1.92}
d_{j,k}(\Lambda):=\delta(\Lambda)\wedge \frac{1}{|k|},\quad \varrho_{j,k}(y;\Lambda):=\varrho_{j}(y;\Lambda)\wedge \frac{1}{|k|}.
\end{equation} 
Then we have the following bounds. 

\begin{lemma}\label{ake1}
For any $\sigma_1,\sigma_2,\sigma_3\in\R$, there exists $\epsilon_1=\epsilon_1(\sigma_1, \sigma_2,\sigma_3)>0$ sufficiently small, such that for $0<\nu<\epsilon_1, k\in\Z\cap[1,\infty)$, $\lambda\in\R$, and $1/4\ge\alpha\ge-\sigma_0|\nu/k|^{1/2}$ as in Proposition \ref{intP1}, we have the following bounds for 
\begin{equation}\label{ake1.1}
    w=\varrho_j^{\sigma_1}(y;\Lambda)\varrho_{j,k}^{\sigma_2}(y;\Lambda)\big(i(\lambda-b(y))+|\alpha|+|\nu/k|^{1/2}+|\nu/k|^{1/3}|\lambda-b(y_{j\ast})|^{1/3}\big)^{\sigma_3}
\end{equation}
and $p\in[1,\infty]$, $h\in L^p(\T_\tp)$.

(i) For $\lambda\in\Sigma_{j,\delta_0},  m\in\Z\cap[0,1]$, $j\in\{1,2\}$,
\begin{equation}\label{ake2}
\begin{split}
&\left\|w(y)\partial_y^m A_\Theta^{-1}h\right\|_{L^p(\T_\tp)}\\
&\lesssim_{\sigma_1,\sigma_2,\sigma_3} \left\|\frac{w(y)}{(L_{j}(y;\Lambda))^{m}\big(i(\lambda-b(y))+|\alpha|+|\nu/k|^{1/2}+|\nu/k|^{1/3}|\lambda-b(y_{j\ast})|^{1/3}\big)}h\right\|_{L^p(\T_\tp)}, 
\end{split}
\end{equation}
where for $y\in\T_\tp$ and $|\lambda-b(y_{j\ast})|\ge|\nu/k|^{1/2}$,
\begin{equation}\label{ake3}
L_{j}(y;\Lambda):=\frac{|\nu/k|^{1/3}|\lambda-b(y_{j\ast})|^{-1/6}}{\langle |\nu/k|^{-1/3}|\lambda-b(y_{j\ast})|^{-1/3}(\lambda-b(y)),  |\nu/k|^{-1/3}|\lambda-b(y_{j\ast})|^{-1/3}\alpha\rangle^{1/2}},\end{equation}
and for $y\in\T_\tp$, $|\lambda-b(y_{j\ast})|<|\nu/k|^{1/2}$,
\begin{equation}\label{ake3.1}
L_{j}(y;\Lambda):=|\nu/k|^{1/4}\langle |\nu/k|^{-1/2}(\lambda-b(y)), |\nu/k|^{-1/2}\alpha\rangle^{-1/2};
\end{equation}
In addition, we have for $m\in\{1,2\}$ the bounds 
\begin{equation}\label{ake3.15}
\begin{split}
&\left\|w(y)\partial_y^m A_\Theta^{-1}h\right\|_{L^p(\T_\tp)}\\
&\lesssim_{\sigma_1,\sigma_2,m,p} \left\|\frac{w(y)}{(L_{j}(y;\Lambda))^{m-1}\big(i(\lambda-b(y))+|\alpha|+|\nu/k|^{1/2}+|\nu/k|^{1/3}|\lambda-b(y_{j\ast})|^{1/3}\big)}\partial_yh\right\|_{L^p(\T_\tp)} \\
&\qquad\qquad+\left\|\frac{w(y)\varrho_{j}(y;\Lambda)}{(L_{j}(y;\Lambda))^{m-1}\big(i(\lambda-b(y))+|\alpha|+|\nu/k|^{1/2}+|\nu/k|^{1/3}|\lambda-b(y_{j\ast})|^{1/3}\big)^{2}}h\right\|_{L^p(\T_\tp)}.
\end{split}
\end{equation}
In particular, if $h\equiv 0$ for $|y-y_{j\ast}|\leq \delta(\lambda)$  then we have for $m\in\{1,2\}$,
\begin{equation}\label{ake3.2}
\begin{split}
&\left\|\varrho_j^{\sigma_1}(y;\Lambda)\varrho_{j,k}^{\sigma_2}(y;\Lambda)\partial_y^m A_\Theta^{-1}h\right\|_{L^p(\T_\tp)}\\
&\lesssim_{\sigma_1,\sigma_2,m,p} \left\|\frac{\varrho_j^{\sigma_1-2}(y;\Lambda)\varrho_{j,k}^{\sigma_2}(y;\Lambda)}{(L_{j}(y;\Lambda))^{m-1}}\partial_yh\right\|_{L^p(\T_\tp)}+\left\|\frac{\varrho_j^{\sigma_1-3}(y;\Lambda)\varrho_{j,k}^{\sigma_2}(y;\Lambda)}{(L_{j}(y;\Lambda))^{m-1}}h\right\|_{L^p(\T_\tp)}.
\end{split}
\end{equation}
(ii) For $\lambda\in\R\backslash\big(\cup_{j\in\{1,2\}}\Sigma_{j,\delta_0}\big)$, $m\in\{0,1\}$,
\begin{equation}\label{ake4}
\left\|\partial_y^m A_\Theta^{-1}h\right\|_{L^p(\T_\tp)}\lesssim \left\|(L_{j}(y;\Lambda))^{-m}\big(i(\lambda-b(y))+|\alpha|+|\nu/k|^{1/3}\big)^{-1}h\right\|_{L^p(\T_\tp)}, 
\end{equation}
where
\begin{equation}\label{ake5}
L(y;\Lambda):=|\nu/k|^{1/3}\langle |\nu/k|^{-1/3}(\lambda-b(y)), |\nu/k|^{-1/3}\alpha\rangle^{-1/2}.
\end{equation}
\end{lemma}

\begin{proof}

The main idea of the proof is to keep track of the correct length scales, and to show that the kernel for $A_\Theta^{-1}$ always defines a smaller length scale on which it decays when compared with the weights. Therefore we can treat $A_\Theta^{-1}$ as a local operator when proving these weighted estimates.

We focus only on the proof of (i), with (ii) being similar. Denote $N(y,z)$, $y,z\in\T_\tp$ as the kernel of $A_\Theta^{-1}$, i.e. for any $h\in L^2(\T_\tp)$ and $y\in\T_\tp$, 
\begin{equation}\label{ake5.01}
    A_\Theta^{-1}h(y)=\int_{\T_\tp} N(y,z)h(z)\,dz.
\end{equation}
It follows from Proposition \ref{Airy_main} that for some $c_0\in(0,1)$, $y,z\in\T_\tp$,
\begin{equation}\label{ake5.02}
    |N(y,z)|\lesssim |\nu/k|^{-1}L_{j}(y;\Lambda)e^{-c_0\big(L^{-1}_{j}(y;\Lambda)+L^{-1}_{j}(z;\Lambda)\big)|y-z|}.
\end{equation}
Denote for $y\in\T_\tp$,
\begin{equation}\label{ake5.03}
    w_\ast(y):=i(\lambda-b(y))+|\alpha|+|\nu/k|^{1/2}+|\nu/k|^{1/3}|\lambda-b(y_{j\ast})|^{1/3}.
\end{equation}

We note that for 
\begin{equation}\label{ake5.04}
    w\in\{\varrho_j^{\sigma_1}, \varrho_{j,k}^{\sigma_2}, L_{j}(\cdot;\Lambda),w_\ast\}
\end{equation}
    and any $\sigma>0$, $y,z\in\T_\tp$,
\begin{equation}\label{ake5.1}
w(y)\lesssim_\sigma w(z)\,e^{\sigma \big(L^{-1}_{j}(y;\Lambda)+L^{-1}_{j}(z;\Lambda)\big)|y-z|}.
\end{equation}

Using \eqref{ake5.1}, we can bound for any $w$ as in \eqref{ake5.04}, 
\begin{equation}\label{ake5.05}
    \begin{split}
&\Big\|w(y)A_\Theta^{-1}h(y)\Big\|_{L^p(\T_\tp)}\\
&\lesssim\Big\|w(y)\int_{\T_\tp}N(y,z)h(z)dz\Big\|_{L^p(\T_\tp)}\\
&\lesssim_{\sigma_1,\sigma_2} \Big\||\nu/k|^{-1}L_{j}(y;\Lambda)w_\ast(y)\int_{\T_\tp}e^{-c_1\big(L^{-1}_{j}(y;\Lambda)+L^{-1}_{j}(z;\Lambda)\big)|y-z|}w_\ast^{-1}(z)w(z)|h(z)|dz\Big\|_{L^p(\T_\tp)}.
    \end{split}
\end{equation}
To complete the proof of \eqref{ake2}, we divide $\T_\tp$ as the union of intervals $\mathcal{I}_i$, $i\in\Z\cap[1,N]$ for some integer $N\ge1$, where the intervals $\mathcal{I}_i$ only intersect at the endpoints with each other and $|\mathcal{I}_i|\approx L_{j}(y;\Lambda)\approx\ell_i$ for $y\in \mathcal{I}_i$. 

Then using \eqref{ake5.05} we can bound for suitable constants $c_2,c_3\in(0,1)$ and each $i\in\Z\cap[1,N]$,
\begin{equation}\label{ake5.07}
\begin{split}
&\Big\|w(y)A_\Theta^{-1}h(y)\Big\|_{L^p(\mathcal{I}_i)}\\
&\lesssim_{\sigma_1,\sigma_2} \Big\|\sum_{i'=1}^N\int_{\mathcal{I}_{i'}}\ell_i^{-1}e^{-c_2(\ell_{i}^{-1}+\ell_{i'}^{-1})|y-z|}w_\ast^{-1}(z)w(z)|h(z)|dz\Big\|_{L^p(\mathcal{I}_{i})}\\
&\lesssim_{\sigma_1,\sigma_2} \Big\|\sum_{i'=1}^N\int_{\mathcal{I}_{i'}}\ell_i^{-1}e^{-c_3|i-i'|}e^{-c_3\ell_i^{-1}|y-z|}w_\ast^{-1}(z)w(z)|h(z)|dz\Big\|_{L^p(\mathcal{I}_{i})}\\
&\lesssim_{\sigma_1,\sigma_2} \sum_{i'=1}^Ne^{-c_3|i-i'|}\big\|w_\ast^{-1}(z)w(z)h\big\|_{L^p(\mathcal{I}_{i'})}.
\end{split}
\end{equation}
The desired bound \eqref{ake2} follows from \eqref{ake5.07}.

The desired bounds \eqref{ake3.2} follow from the bounds \eqref{ake2} and the equation for $y\in\T_\tp$,
\begin{equation}\label{ake5.2}
\Big[\frac{\nu}{k}\partial_y^2+i(\lambda-b(y))-\alpha\Big]\partial_yu(y)=\partial_yh(y)+ib'(y)u(y),
\end{equation}
where
\begin{equation}\label{ake5.3}
u=A_\Theta^{-1}h.
\end{equation}
\end{proof}

The above lemmas are very useful for treating $A_\Theta^{-1}$ away from the singular region when $|\lambda-b(y)|\gtrsim \delta(\lambda)$. In the viscous layer when $|\lambda-b(y_{j\ast})|\lesssim |\nu/k|^{1/2}$, we can use the viscous term $(\nu/k)\partial_y^2$ to gain regularity without much loss. To deal with the singular behavior in the remaining case when $|\nu/k|^{1/2} \lesssim |\lambda-b(y)|\lesssim \delta(\lambda)$, we rely on the following lemma. 
\begin{lemma}\label{ake6}
For any $\sigma_1,\sigma_2\in\R$, there exists $\epsilon_1=\epsilon_1(\sigma_1, \sigma_2)>0$ sufficiently small, such that for $0<\nu<\epsilon_1, k\in\Z\cap[1,\infty)$, $\lambda\in\R$, $j\in\{1,2\}$, $-\sigma_0|\nu/k|^{1/2}\leq\alpha\ll |\nu/k|^{1/3}|\lambda-b(y_{j\ast})|^{1/3}$, where $\sigma_0$ is from Propostion \ref{Airy_main}, the following statement holds. 

(i) Assume that $\lambda\in\Sigma_{j,\delta_0}\cap b(\T_\tp)$ and $|\lambda-b(y_{j\ast})|\gg |\nu/k|^{1/2}$. In this case there exist $y^+_{\lambda}, y^-_{\lambda}\in S^j_{3\delta}$, arranged so that $y^+_\lambda$ is in the positive direction, counter-clockwise, to $y^-_\lambda$, such that 
\begin{equation}\label{ake6.0}
 \lambda=b(y^+_{\lambda})=b(y^-_{\lambda}).
 \end{equation}
 Choose a cut-off function $\varphi\in C_c^\infty(-1,1)$ with $\varphi\equiv 1$ on $[-1/2,1/2]$. Recall the definition \eqref{DeC1.92} for $d_{j,k}$. Assume that 
 $${\rm supp}\,h\subseteq S^j_{9\delta}.$$
 Denote
\begin{equation}\label{ake6.1}
\epsilon:=\nu/k, \,\,M:=\|h\|_{L^2(S^j_{9\delta})}+d_{j,k}(\Lambda)\|\partial_yh\|_{L^2(S^j_{9\delta})}.
\end{equation}
There exists a decomposition for $y\in\T_\tp$,
\begin{equation}\label{ake6.3}
A_\Theta^{-1}h(y):=w_1(y,\lambda)+w_2(y,\lambda),
\end{equation}
satisfying the property that for $y\in\T_\tp$,
\begin{equation}\label{ake6.4}
|w_1(y,\lambda)|\lesssim \sum_{\dagger\in\{+,-\}}\Big[d_{j,k}(\Lambda)|\lambda-b(y_{j\ast})|^{1/2}\Big]^{-1}|y-y^\dagger_{\lambda}|^{-1/2}M+\Big[\big(d_{j,k}(\Lambda)\big)^{1/2}|\lambda-b(y_{j\ast})|\Big]^{-1}M,
\end{equation}
and
\begin{equation}\label{ake6.5}
\begin{split}
&\bigg|w_2(y,\lambda)-\sum_{\dagger\in\{+,-\}}\frac{h(y^\dagger_\lambda)\varphi\big((y-y^\dagger_\lambda)/|\lambda-b(y_{j\ast})|^{1/2}\big)}{ib'(y^\dagger_\lambda)(y-y^\dagger_\lambda)+ |\epsilon|^{1/3}|\lambda-b(y_{j\ast})|^{1/3}}\bigg|\\
&\lesssim \sum_{\dagger\in\{+,-\}}|\lambda-b(y_{j\ast})|^{-1/2}\frac{|\epsilon|^{-1/3}|\lambda-b(y_{j\ast})|^{1/6}}{1+(|\epsilon|^{1/3}|\lambda-b(y_{j\ast})|^{-1/6})^{-2}|y-y^\dagger_{\lambda}|^2}\big(d_{j,k}(\Lambda)\big)^{-1/2}M.
\end{split}
\end{equation}

(ii) Assume that $\lambda\in b(\T_\tp)\backslash\Sigma_{j,\delta_0}$. 
In this case there exist $y^+_{\lambda}, y^-_{\lambda}\in \T_\tp \backslash S_{j,\delta_0}$, arranged so that $y^+_\lambda$ is in the positive direction, counter-clockwise, to $y^-_\lambda$, such that 
\begin{equation}\label{ake7.0}
 \lambda=b(y^+_{\lambda})=b(y^-_{\lambda}).
 \end{equation}
 Denote
\begin{equation}\label{ake7.1}
\epsilon:=\nu/k, \,\,M:=\|h\|_{L^2(\T_\tp)}+|k|^{-1}\|\partial_yh\|_{L^2(\T_\tp)}.
\end{equation}
There exists a decomposition for $y\in\T_\tp$
\begin{equation}\label{ake7.3}
A_\Theta^{-1}h(y):=w_1(y,\lambda)+w_2(y,\lambda),
\end{equation}
satisfying the property that for $y\in\T_\tp$,
\begin{equation}\label{ake7.4}
|w_1(y,\lambda)|\lesssim \sum_{\dagger\in\{+,-\}}|k||y-y^\dagger_{\lambda}|^{-1/2}M,
\end{equation}
and for $y\in\T_\tp$,
\begin{equation}\label{ake7.5}
\begin{split}
&\bigg|w_2(y,\lambda)-\sum_{\dagger\in\{+,-\}}\frac{h(y^\dagger_\lambda)\varphi(y-y^\dagger_\lambda)}{ib'(y^\dagger_\lambda)(y-y^\dagger_\lambda)+ |\epsilon|^{1/3}}\bigg|\lesssim \sum_{\dagger\in\{+,-\}}\frac{|\epsilon|^{-1/3}}{1+|\epsilon|^{-2/3}|y-y^\dagger_{\lambda}|^2}|k|^{1/2}M.
\end{split}
\end{equation}

\end{lemma}

\begin{proof}
Without loss of generality, we assume that $\epsilon>0$. We focus on the proof of the case (i), with (ii) being similar and simpler. Define for $y\in\T_\tp$, 
\begin{equation}\label{ake101}
    w(y,\lambda):=A_\Theta^{-1}h(y).
\end{equation}
It follows from Lemma \ref{ake1}, the bound \eqref{ake6} and Sobolev inequalities that for $y\in\T_\tp$ with $$\min\{|y-y_\lambda^+|,|y-y_\lambda^-|\}\gtrsim |\lambda-b(y_{j\ast})|^{1/2}$$ we have bound
\begin{equation}\label{ake102}
    |w(y,\lambda)|\lesssim |\lambda-b(y_{j\ast})|^{-1}\|h\|_{L^\infty(\T_\tp)}\lesssim \Big[\big(d_{j,k}(\Lambda)\big)^{1/2}|\lambda-b(y_{j\ast})|\Big]^{-1}M.
\end{equation}
Note that $w$ satisfies for $\dagger\in\{+,-\}$, $|y-y^\dagger_\lambda|\ll |\lambda-b(y_{j\ast})|^{1/2}$,
\begin{equation}\label{ake103}
\begin{split}
    &\Big[\epsilon\partial_y^2+ib'(y_\lambda^\dagger)(y_\lambda^\dagger-y)-\alpha\Big]w(y,\lambda)\\
    &=h(y)+i\Big[b'(y_\lambda^\dagger)(y_\lambda^\dagger-y)+b(y)-b(y_\lambda^\dagger)\Big]w(y,\lambda)\\
    &=h(y^\dagger_\lambda)+h(y)-h(y^\dagger_\lambda)+i\Big[b'(y_\lambda^\dagger)(y_\lambda^\dagger-y)+b(y)-b(y_\lambda^\dagger)\Big]w(y,\lambda).
    \end{split}
\end{equation}
To extract the most singular part near $y=y^\dagger_\lambda$, we define for $y\in\R$, $W(y,\Lambda)$  as the solution to 
\begin{equation}\label{ake104}
    \Big[\partial_y^2-iy-\epsilon^{-1/3}|b'(y_\lambda^\dagger)|^{-2/3}\alpha\Big]W(y,\Lambda)=1. 
\end{equation}
It follows from Proposition 3.1 in \cite{JiaUM} that $W(y, \Lambda)$ satisfies the following bounds for $y\in\R$,
\begin{equation}\label{ake105}
\Big|W(y,\Lambda)-\frac{1}{1-iy}\Big|\lesssim \frac{1}{1+y^2}.
\end{equation}
The desired bounds then follow from Lemma \ref{ake1} and the inequality \eqref{ake105} using equation \eqref{ake103}, by taking 
\begin{equation}\label{ake106}
    w_2(y,\lambda):=\sum_{\dagger\in\{+,-\}}
    \epsilon^{-1/3}|b'(y_\lambda^\dagger)|^{-2/3}h(y^\dagger_\lambda)W\Big((y-y^\dagger_\lambda)\epsilon^{-1/3}|b'(y_\lambda^\dagger)|^{1/3},\Lambda\Big)\varphi\big((y-y^\dagger_\lambda)/|\lambda-b(y_{j\ast})|^{1/2}\big).
\end{equation}

\end{proof}

\begin{remark}\label{RMb1}
Denote for $y\in\T_\tp, \dagger\in\{+,-\}$,
\begin{equation}\label{RMb0}
g(y):=\int_{\T_\tp}\mathcal{G}_k^j(y,z;\Lambda)A_\Theta^{-1}h(z)\,dz.
\end{equation}
It follows from \eqref{ake6.3} and \eqref{ake6.5}, using integration by parts and the identity for $\iota\in\R\backslash\{0\}$
\begin{equation}\label{RMb1.001}
    \frac{1}{y-y^\dagger+i\iota}=\partial_y\log\frac{y-y^\dagger+i\iota}{\delta},
\end{equation}
that 
\begin{equation}\label{RMb0.1}
\begin{split}
   & \left\|g\right\|_{L^2(S^j_{9\delta})}+d_{j,k}(\Lambda)\left\|\partial_yg\right\|_{L^2(S^j_{9\delta})} \lesssim \frac{d_{j,k}^{1/2}(\Lambda)}{\delta^{1/2}}M,
    \end{split}
\end{equation}
and for $y\in \T_\tp\backslash S^j_{9\delta}$, 
\begin{equation}\label{RMb0.2}
\begin{split}
    &|g(y)|+d_{j,k}(\Lambda)|\partial_yg(y)|\lesssim d_{j,k}^{-1/2}\delta^{-1}\Big[\delta^{-1}\int_{|z-y_{j\ast}|<10\delta }\mathcal{G}_{k}^j(y,z;\Lambda)\,dz\Big]\,M.
    \end{split}
\end{equation}

\end{remark}

\section{The limiting absorption principle and main bounds on the spectral density function}\label{sec:mbsd}

In this section we state the limiting absorption principle, and derive bounds on the spectral density functions using the limiting absorption principle and the equations \eqref{intP5}. 


\subsection{The limiting absorption principle for the Orr-Sommerfeld equation}\label{sec:lap}
For $k\in\Z\cap[1,\infty), \nu\in(0,1/8), \lambda\in\R, \alpha\ge -\sigma_0|\nu/k|^{1/2}$ with $\sigma_0 \in(0,1/2)$ sufficiently small from Proposition \ref{Airy_main}, denoting $\Theta:=(\lambda,\alpha,\nu,k,j)$,
 we define the operator $T_\Theta: H^1(\T_\tp)\to H^2(\T_\tp)$ as follows.

If $\lambda\in\R\backslash\big(\cup_{j=1}^2\Sigma_{j,\delta_0}\big)$ or $\alpha\ge\delta_0$, we define for any $h\in H^1(\T_\tp)$ and $y\in\T_\tp$,
\begin{equation}\label{lapj1}
    T_\Theta h(y):=\Delta_k^{-1}A_\Theta^{-1}\big(ib''h\big)(y),
\end{equation}
where we used the notation $\Delta_k:=\partial_y^2-k^2$.

If $\lambda\in\Sigma_{j,\delta_0}$ and $-\sigma_0|\nu/k|^{1/2}\leq\alpha\leq\delta_0$, for some $j\in\{1,2\}$, then we define for any $h\in H^1(\T_\tp)$ and $y\in\T_\tp$,
\begin{equation}\label{lapj2}
\begin{split}
T_\Theta h(y):=-\int_{\T_\tp}\mathcal{G}_k^j(y,z;\Lambda)&\Big\{A_\Theta^{-1}\big(ib''h\big)(z)\\
&\quad+\frac{b''(z)}{b(z)-\lambda -i\alpha}\Big[\varphi_0\big(\frac{z-y_{j\ast}}{\delta_0}\big)-\varphi_0\big(\frac{z-y_{j\ast}}{\delta(\Lambda)}\big)\Big]h(z)\Big\}dz.
\end{split}
\end{equation}

The main technical tool we need is the following limiting absorption principle. 
\begin{proposition}\label{lap_main}
Fix $\gamma\in[7/4,2)$. Under the assumptions in Proposition \ref{intP1}, there exist $\nu_0,\kappa,\sigma_\sharp\in(0,1/10)$ sufficiently small depending on $\gamma$, such that the following statement holds. Assume that 
\begin{equation}\label{eqRan1}
(\sigma_1,\sigma_2)\in\big\{(0,-\gamma), (1,-\gamma+1), (1-2(2-\gamma), -\gamma+2)\big\}. 
\end{equation}
For $k\in\Z\cap[1,\infty), \nu\in(0,\nu_0), \alpha\ge -\sigma_\sharp|\nu/k|^{1/2}$ and $\lambda\in\R$, we have the following bounds. 
    
   (i) If $\lambda\in\Sigma_{j,\delta_0}$ for some $j\in\{1,2\}$, then for any $h\in H^2(\T_\tp)$ and $(\sigma_1,\sigma_2)$ as in \eqref{eqRan1},
   \begin{equation}\label{lapj3}
        \big\|h+T_\Theta h\big\|_{X^{\sigma_1,\sigma_2}(\mathfrak{M})}\ge \kappa \|h\|_{X^{\sigma_1,\sigma_2}(\mathfrak{M})};
    \end{equation}
   (ii) If $\lambda\in\R\backslash\big(\cup_{j=1}^2\Sigma_{j,\delta_0}\big)$, then 
   \begin{equation}\label{lapj4}
       \big\|h+T_\Theta h\big\|_{H^1_k(\T_\tp)}\ge \kappa \|h\|_{H^1_k(\T_\tp)}.
   \end{equation}
\end{proposition}

\begin{remark}
We briefly explain the design of the exponents \eqref{eqRan1} and the weighted norm \eqref{Intd4}, which will be used in section \ref{sec:bspd} to control the spectral density functions and their derivatives in $\lambda$ up to two derivatives. 

The weights used in Proposition \ref{lap_main} are motivated by the decay of the modified Green's functions near the critical points $y_{1\ast}, y_{2\ast}$. Assuming $k=1$ and $\min\{|z-y_{1\ast}|,|z-y_{2\ast}|\}\gtrsim1$ for simplicity, the modified Green's function  $\mathcal{G}_k^j(y,z;\Lambda)\sim (|y-y_{j\ast}|+\delta(\Lambda))^2$
for $y$ close to $y_{j\ast}$. This asymptotic matches the weights used in \eqref{Intd4} for $(\sigma_1,\sigma_2)=(0,-\gamma)$. The use of $L^2$ and $L^\infty$ based spaces adapts to the singularity at $\lambda=b(y)$ that appears in \eqref{lapj2}. 

For $(\sigma_1,\sigma_2)=(0,-\gamma)$ and $(\sigma_1,\sigma_2)=(1,-\gamma+1)$, the corresponding norms reflect the heuristic that taking one derivative in $\lambda$ costs $\varrho_j^{-1}\varrho_{j,k}^{-1}$ (in contrast to taking one derivative in $y$ which costs only $\varrho_{j,k}^{-1}$). The extra adjustment for $(\sigma_1,\sigma_2)=(1-2(2-\gamma), -\gamma+2)$, where the weight $\varrho^{1-2(2-\gamma)}_j\varrho_{j,k}^{-\gamma+2}$ was used in favor of the seemingly more natural weight $\varrho^{2}_j\varrho_{j,k}^{-\gamma+2}$, is due to the requirement that the weights need to vary more slowly than the modified Green's function to avoid nonlocal effects. In applications below (see \eqref{mainprop2}), we need to incorporate an additional factor $\delta(\Lambda)^{1+2(2-\gamma)}$ to match the two derivatives taken. 

\end{remark}
   
\subsection{Main bounds on the spectral density function}\label{sec:bspd}
Assume that $j\in\{1,2\}, k\in\Z\cap[1,\infty)$, $|\alpha|\leq \sigma_{\sharp}|\nu/k|^{1/2}$ and $ \gamma\in[7/4,2)$. 

We define the ``modified spectral density function" $\psi_{k,\nu}^\ast(y,\lambda)$ for $y\in\T_{\tp}$ as 
\begin{equation}\label{mainprop0.3}
\psi_{k,\nu}^\ast(y,\Lambda):=\psi_{k,\nu}(y,\Lambda)-\sum_{j=1}^2\varphi_0\big((y-y_{j\ast})/\delta_0\big)\frac{\omega_{0k}(y)}{ib''(y)}.
\end{equation}
Then $\psi_{k,\nu}^\ast(y,\lambda)$ satisfies the equation for $y\in\T_{\tp}$,
\begin{equation}\label{mainprop0.31}
\left\{\begin{array}{rl}
\Big[\frac{\nu}{k}\partial_y^2+i(\lambda-b(y))-\alpha\Big]\omega^\ast_{k,\nu}(y,\Lambda)+ib''(y)\psi^\ast_{k,\nu}(y,\Lambda)=f_{0k}(y,\Lambda),&\\
(-k^2+\partial_y^2)\psi^\ast_{k,\nu}(y,\Lambda)=\omega^\ast_{k,\nu}(y,\Lambda),&
\end{array}\right.
\end{equation}
where 
\begin{equation}\label{mainprop0.32}
\begin{split}
f_{0k}(y,\Lambda):=&\,\,\omega_{0k}(y)-\sum_{j=1}^2\varphi_0\big((y-y_{j\ast})/\delta_0\big)\omega_{0k}(y)\\
&-\Big[\frac{\nu}{k}\partial_y^2+i(\lambda-b(y))-\alpha\Big](\partial_y^2-k^2)\Big\{\sum_{j=1}^2\varphi_0\big((y-y_{j\ast})/\delta_0\big)\frac{\omega_{0k}(y)}{ib''(y)}\Big\}.
\end{split}
\end{equation}

The proof of our main theorems is based on the following proposition on the regularity of the spectral density function. 

\begin{proposition}\label{mainprop}
Fix $\gamma\in[15/8,2)$ and $\nu_0, \sigma_\sharp\in(0,1)$ from Proposition \ref{lap_main}. Assume that $\nu\in(0,\nu_0)$, $k\in\Z\cap[1,\infty)$, $\epsilon:=\nu/k\in(0,1/8)$, $|\alpha|\leq\sigma_\sharp|\nu/k|^{1/2}$. Let $\omega_{0k}, \omega_{k,\nu}, \psi_{k,\nu}$ be defined as in Proposition \ref{intP1}, and let $ \omega^\ast_{k,\nu}, \psi^\ast_{k,\nu}$ be as in \eqref{mainprop0.31}-\eqref{mainprop0.32}.


Define the ``good derivative" $D_\lambda$ for $\lambda\in\R, y\in\T_{\tp}$ as follows.
If $\lambda\in b(\T_\tp)$ and $\min\{|\lambda-b(y_{1\ast})|, |\lambda-b(y_{2\ast})|\}\ge 10|\nu/k|^{1/2}$, then
\begin{equation}\label{mainprop0.2}
D_\lambda:=\partial_\lambda+\Big[1-\varphi_0\big(\frac{y-y_{1\ast}}{\delta_2(\lambda)}\big)\Big]\Big[1-\varphi_0\big(\frac{y-y_{2\ast}}{\delta_2(\lambda)}\big)\Big]\frac{1}{b'(y)}\partial_y.
\end{equation}
If $\lambda\not\in b(\T_\tp)$ or $\min\{|\lambda-b(y_{1\ast})|, |\lambda-b(y_{2\ast})|\}< 10|\nu/k|^{1/2}$, denote $D_\lambda:=\partial_\lambda$. 
 
 Then we have the following bounds. 

 (i) For $\lambda\in \R\backslash\big(\cup_{j\in\{1,2\}}\Sigma_{j,\delta_0}\big), \sigma\in\{0,1,2\}$,
\begin{equation}\label{mainprop3}
\big\|D_\lambda^\sigma\, \psi_{k,\nu}(\cdot,\Lambda)\big\|_{H^1_k(\T_{\tp})}\lesssim |k|^{-5/2+\sigma}\frac{1}{\langle\lambda\rangle^{\sigma+1}}\big\|\omega_{0k}\big\|_{H^3_k(\T_{\tp})};
\end{equation}
 
 (ii) For $j\in\{1,2\}, \lambda\in\Sigma_{j,\delta_0}$,
\begin{equation}\label{mainprop2}
\begin{split}
&\big\|\psi_{k,\nu}^\ast(\cdot,\Lambda)\big\|_{X^{0,-\gamma}(\mathfrak{M})}\lesssim |k|^{-1/2}\big\|\omega_{0k}\big\|_{H^3_k(\T_{\tp})},\\
&\big\|D_\lambda\, \psi_{k,\nu}^\ast(\cdot,\Lambda)\big\|_{X^{1,-\gamma+1}(\mathfrak{M})}\lesssim |k|^{-1/2}\big\|\omega_{0k}\big\|_{H^3_k(\T_{\tp})},\\
&\big\|D_\lambda^2\, \psi_{k,\nu}^\ast(\cdot,\Lambda)\big\|_{X^{1-2(2-\gamma),-\gamma+2}(\mathfrak{M})}\lesssim |k|^{-1/2}(\delta(\Lambda))^{-1-2(2-\gamma)}\big\|\omega_{0k}\big\|_{H^3_k(\T_{\tp})}.
\end{split}
\end{equation}

\end{proposition}

As a corollary of Lemma \ref{ake6} and Proposition \ref{mainprop}, we also have the following refined description of the singularities of $\partial_yD_\lambda \psi^\ast_{k,\nu}(y,\Lambda)$ when $y$ approaches the points $\{y\in\T_\tp: b(y)=\lambda\}$. 
\begin{corollary}\label{RDS1}
There exist $\nu_0, \sigma_\sharp>0$ sufficiently small, such that for $0<\nu<\nu_0, k\in\Z\cap[1,\infty)$, $\lambda\in\R$, $j\in\{1,2\}$, $|\alpha|\leq\sigma_\sharp|\nu/k|^{1/2}$, the following statement holds.  Denote
 \begin{equation}\label{RDS2.1}
     M_k:=\|\omega_{0k}\|_{H^3_k(\T_\tp)}.
 \end{equation}

(i) Assume that $\lambda\in\Sigma_{j,\delta_0}\cap b(\T_\tp)$ and $|\lambda-b(y_{j\ast})|\gg |\nu/k|^{1/2}$. 
Then we have the bounds for $|\lambda-b(y)|\ll|b(y)-b(y_{j\ast})|$,
\begin{equation}\label{RDS3}
    |\partial_yD_\lambda\psi^\ast_{k,\nu}(y,\Lambda)|\lesssim \big(|y-y_{j\ast}|\wedge\frac{1}{|k|}\big)^{\gamma-2}|y-y_{j\ast}|^{-1}\bigg(1+\bigg|\log{\bigg|\frac{\lambda-b(y)}{b(y)-b(y_{j\ast})}\bigg|}\bigg|\bigg)|k|^{-1/2}M_k,
\end{equation}
and
\begin{equation}\label{RDS4}
\begin{split}
   & \bigg|\omega_{k,\nu}(y,\Lambda)+\frac{ib''(y)\psi^\ast_{k,\nu}(y,\Lambda)}{i(\lambda-b(y))+|\nu/k|^{1/3}|y-y_{j\ast}|^{2/3}}\bigg|\\
   &\lesssim \varrho_{j,k}^{\gamma-1/2}(y,\Lambda)|y-y_{j\ast}|^{-1/2}\frac{|\nu/k|^{-1/3}|y-y_{j\ast}|^{1/3}}{1+(|\nu/k|^{1/3}|y-y_{j\ast}|^{2/3})^{-2}|\lambda-b(y)|^2}|k|^{-1/2}M_k.
    \end{split}
\end{equation}

 (ii) If $\lambda\in b(\T_\tp)\backslash\big(\cup_{j=1}^2 \Sigma_{j,\delta_0}\big)$, we have the bounds for $y\in\T_\tp$ with $|\lambda-b(y)|\ll1$, 
 \begin{equation}\label{RDS5}
      |\partial_yD_\lambda\psi_{k,\nu}(y,\Lambda)|\lesssim \big(1+|\log{|\lambda-b(y)|}|\big)|k|^{-1/2}M_k,
 \end{equation}
 and
 \begin{equation}\label{RDS6}
     \begin{split}
   & \bigg|\omega_{k,\nu}(y,\Lambda)-\frac{\omega_{0k}(y)-ib''(y)\psi_{k,\nu}(y,\Lambda)}{i(\lambda-b(y))+|\nu/k|^{1/3}}\bigg|\\
   &\lesssim \frac{|\nu/k|^{-1/3}}{1+|\nu/k|^{-2/3}|\lambda-b(y)|^2}|k|^{-5/2}M_k.
    \end{split}
 \end{equation}
\end{corollary}

The proof of Corollary \ref{RDS1} on the refined properties of the functions $\partial_y\psi^\ast_{k,\nu}$,  $\partial_yD_{\lambda}\psi^\ast_{k,\nu}$ and $\omega_{k,\nu}^\ast$ follows from equations \eqref{mainprop0.31}, \eqref{pmpr1.5}, and \eqref{refb5}  using Proposition \ref{mainprop} and Lemmas \ref{ake1} and \ref{mGk50}. These refined bounds are useful below in the proof of the main theorem. 
The rest of the section is devoted to the proofs of Proposition \ref{mainprop}. We organize the proofs into different sections as the details of the argument are somewhat lengthy. 

\subsection{Proof for statement (i) of Proposition \ref{mainprop}}
We divide the proof into the cases $\sigma\in\{0,1,2\}$, organized in subsections.
\subsubsection{The case $\sigma=0$ for statement (i) of Proposition \ref{mainprop}} To establish \eqref{mainprop} for $\sigma=0$, by the limiting absorption principle \eqref{lapj4}, it suffices to show that 
\begin{equation}\label{pmpr1}
    \big\|(k^2-\partial_y^2)^{-1}A_\Theta^{-1}\omega_{0k}\big\|_{H^1_k(\T_\tp)}\lesssim |k|^{-5/2}\langle\lambda\rangle^{-1} M_k.
\end{equation}
If $\lambda\not\in b(\T_\tp)$, the desired bounds \eqref{pmpr1} follow from Lemma \ref{ake1} and standard elliptic estimates, since in this case $A_\Theta^{-1}$ does not involve any singularities. 

We now consider the case $\lambda\in b(\T_\tp)$. 
We note first the inequality that
\begin{equation}\label{pmpr1.2}
    \Big\|\omega_{0k}\Big\|_{H^1_k(\T_\tp)}\lesssim |k|^{-2}M_k.
\end{equation}
To prove \eqref{pmpr1}, we use \eqref{pmpr1.2}, Lemma \ref{ake1} and Lemma \ref{ake6}, by considering the regions (I) $|\lambda-b(y)|\gtrsim1$ and (II) $|\lambda-b(y)|\ll1$ respectively. In the region (I), the desired bound \eqref{pmpr1} follows from \eqref{pmpr1.2} and Lemma \ref{ake1}. In region (II), we use Lemma \ref{ake6}. After peeling off the easier bounds for $\dagger\in\{+,-\}$ (using the notation as in Lemma \ref{ake6})
\begin{equation}\label{pmpr1.3}
    \bigg\|\int_{\T_\tp}G_k(y,z)\frac{|\nu/k|^{-1/3}}{1+|\nu/k|^{-2/3}|z-y_{\lambda}^\dagger|^2}dz\bigg\|_{H^1_k(\T_\tp)}\lesssim |k|^{-1/2},
\end{equation}
it suffices to notice the inequality that for $\dagger\in\{+,-\}$ (with the notation that $\epsilon':=|\nu/k|^{1/3}|\lambda-b(y_{j\ast})|^{1/3}/b'(y_\lambda^\dagger)$),
\begin{equation}\label{pmpr1.4}
   \bigg\||k|^{-1}\int_{\R}e^{-|k||y-z|}\frac{\varphi(z-y_\lambda^\dagger)}{i(z-y_\lambda^\dagger)+\epsilon'}dz\bigg\|_{H^1_k(\R)}\lesssim |k|^{-1/2}. 
\end{equation}
This completes the proof of \eqref{mainprop3} for $\sigma=0$. 

\subsubsection{The case $\sigma=1$ for statement (i) of Proposition \ref{mainprop}}\label{casesigma1} We now turn to the case $\sigma=1$. 
Set for $y\in\T_\tp, \lambda\in \R$, 
\begin{equation}\label{refb2}
a(y,\lambda):=\Big[1-\varphi_0\big(\frac{y-y_{1\ast}}{\delta_2(\lambda)}\big)\Big]\Big[1-\varphi_0\big(\frac{y-y_{2\ast}}{\delta_2(\lambda)}\big)\Big]\frac{1}{b'(y)}.
\end{equation}
Taking the derivative $D_\lambda$, by a straightforward but somewhat lengthy calculation, we obtain that for $y\in\T_\tp$,
\begin{equation}\label{pmpr1.5}
    \begin{split}
&(\partial_y^2-k^2)D_\lambda \psi_{k,\nu}(y,\Lambda)+A_\Theta^{-1} \big(ib''(\cdot)D_\lambda\psi_{k,\nu}(\cdot,\Lambda)\big)(y)\\
&=A_\Theta^{-1}\big(g^\ast_{0k}(\cdot,\Lambda)\big)(y)+\Big[\partial_y^2a(y,\lambda)\partial_y\psi_{k,\nu}+2\partial_ya(y,\lambda)\partial^2_y\psi_{k,\nu}\Big](y,\Lambda),
\end{split}
\end{equation}
where for $y\in\T_\tp$,
\begin{equation}\label{pmpr1.6}
\begin{split}
g^\ast_{0k}(y,\Lambda):=&\,\,D_\lambda \omega_{0k}(y)+\frac{\nu}{k}\partial_y^2a(y,\lambda)\partial_y\omega_{k,\nu}(y,\Lambda)\\
&+2\frac{\nu}{k}\partial_ya(y,\lambda)\partial_y^2\omega_{k,\nu}(y,\Lambda)-ib'''(y)a(y,\lambda)\psi_{k,\nu}(y,\Lambda)\\
&+i\big[b'(y)a(y,\lambda)-1\big]\omega_{k,\nu}(y,\Lambda).
\end{split}
\end{equation}
The calculations to get \eqref{pmpr1.5}-\eqref{pmpr1.6} are similar to those in subsection \ref{casesigma1'} which are slightly more complicated. For the sake of conciseness, we omit the repetitive details here. By Proposition \ref{lap_main}, it suffices to show for 
\begin{equation}\label{pmpr1.7}
\begin{split}
    h\in A_\Theta^{-1}\Big\{&D_\lambda \omega_{0k}, \frac{\nu}{k}\partial_y^2a(y,\lambda)\partial_y\omega_{k,\nu}(y,\Lambda),\\
&2\frac{\nu}{k}\partial_ya(y,\lambda)\partial_y^2\omega_{k,\nu}(y,\Lambda),b'''(y)a(y,\lambda)\psi_{k,\nu}(y,\Lambda),\\
&\big[b'(y)a(y,\lambda)-1\big]\omega_{k,\nu}(y,\Lambda)\Big\},
\end{split}
\end{equation}
and 
\begin{equation}\label{pmpr1.8}
    \begin{split}
h\in\Big\{\partial_y^2a(y,\lambda)\partial_y\psi_{k,\nu}, \partial_ya(y,\lambda)\partial^2_y\psi_{k,\nu}\Big\},
    \end{split}
\end{equation}
the following bounds hold
\begin{equation}\label{pmpr1.9}
    \big\|(k^2-\partial_y^2)^{-1}h\big\|_{H^1_k(\T_\tp)}\lesssim \langle\lambda\rangle^{-2}|k|^{-3/2}M_k. 
\end{equation}
If $\lambda\in\R\backslash \Big[b(\T_\tp)\cup\big(\cup_{j=1}^2\Sigma_{j,\delta_0}\big)\Big]$, then the desired bounds \eqref{pmpr1.9} follow from \eqref{mainprop3} for the case $\sigma=0$ easily since in this case $A_\Theta^{-1}$ does not involve strong singularities. Below we assume that $\lambda\in b(\T_\tp)\backslash\big(\cup_{j=1}^2\Sigma_{j,\delta_0}\big)$. The desired bounds for 
\begin{equation}\label{pmpr1.10}
   \begin{split}
    h\in \Big\{&A_\Theta^{-1}D_\lambda \omega_{0k}, A_\Theta^{-1}\big(b'''a(y,\lambda)\psi_{k,\nu}),\\
   & A_\Theta^{-1}\big[(b'(y)a(y,\lambda)-1)\omega_{k,\nu}\big], \partial_y^2a(y,\lambda)\partial_y\psi_{k,\nu}\Big\},
\end{split}
\end{equation}
follow from the case $\sigma=0$, using the inequalities \eqref{pmpr1.3}-\eqref{pmpr1.4} and Lemma \ref{ake6}. We note that $b'a-1\equiv0$ in a neighborhood of $\{y\in \T_\tp: b(y)-\lambda=0\}$. Hence $(b'(y)a(y,\lambda)-1)\omega_{k,\nu}$ is not singular. 

To handle the case $h=\partial_ya(y,\lambda)\partial^2_y\psi_{k,\nu}$, we use the identity for $y\in\T_\tp$,
\begin{equation}\label{pmpr1.11}
    \partial^2_y\psi_{k,\nu}(y,\Lambda)=k^2\psi_{k,\nu}(y,\Lambda)+\omega_{k,\nu}(y,\Lambda).
\end{equation}
The desired bounds \eqref{pmpr1.9} for $h=\partial_ya(y,\lambda)\partial^2_y\psi_{k,\nu}$ then follow from analogous argument as in the proof of \eqref{pmpr1}. It remains to prove \eqref{pmpr1.9} for 
\begin{equation}\label{pmpr1.12}
    \begin{split}
    h\in A_\Theta^{-1}\Big\{\frac{\nu}{k}\partial_y^2a(y,\lambda)\partial_y\omega_{k,\nu}(y,\Lambda),\frac{\nu}{k}\partial_ya(y,\lambda)\partial_y^2\omega_{k,\nu}(y,\Lambda)\Big\},
\end{split}
\end{equation}
We consider only the harder case when $h=A_\Theta^{-1}\Big[\frac{\nu}{k}\partial_ya(y,\lambda)\partial_y^2\omega_{k,\nu}\Big]$ which is more singular. 

It follows from equation \eqref{intP5} that for $y\in\T_\tp$,
\begin{equation}\label{pmpr1.13}
    \Big[\frac{\nu}{k}\partial_y^2+i(\lambda-b(y))-\alpha\Big]\omega_{k,\nu}(y,\Lambda)=\omega_{0k}(y)-ib''(y)\psi_{k,\nu}(y,\Lambda).
\end{equation}
Thus for $y\in\T_\tp$, 
\begin{equation}\label{pmpr1.14}
\begin{split}
    &\Big[\frac{\nu}{k}\partial_y^2+i(\lambda-b(y))-\iota\alpha\Big]\partial_y\omega_{k,\nu}(y,\Lambda)\\
    &=\partial_y\big[\omega_{0k}(y)-ib''(y)\psi_{k,\nu}(y,\Lambda)\big]+ib'(y)\omega_{k,\nu}(y,\Lambda).
    \end{split}
\end{equation}
Using the bounds 
\begin{equation}\label{pmpr1.15}
    \big\|\partial_y\big[\omega_{0k}(y)-ib''(y)\psi_{k,\nu}(y,\Lambda)\big]\big\|_{L^2(\T_\tp)}\lesssim |k|^{-2}M_k,
\end{equation}
we obtain from Lemma \ref{ake1} that 
\begin{equation}\label{pmpr1.16}
\begin{split}
    &\Big\|\frac{\nu}{k}A_\Theta^{-1}\Big\{\partial_ya\,\partial_yA_\Theta^{-1}\partial_y\big[\omega_{0k}(y)-ib''(y)\psi_{k,\nu}(y,\Lambda)\big]\Big\}\Big\|_{L^2(\T_\tp)}\\
    &\lesssim\Big\|(i(\lambda-b(y))+|\nu/k|^{1/3})^{-2}\frac{\nu}{k}L^{-1}(y;\Lambda)\partial_y\big[\omega_{0k}(y)-ib''(y)\psi_{k,\nu}(y,\Lambda)\big]\Big\|_{L^2(\T_\tp)}\\
    &\lesssim |k|^{-2}M_k, 
\end{split}
\end{equation}
and for $y\in\T_\tp$,
\begin{equation}\label{pmpr1.17}
\begin{split}
    &\Big|\frac{\nu}{k}A_\Theta^{-1}\Big\{\partial_ya\,(y,\lambda)\partial_yA_\Theta^{-1}\big(b'\omega_{k,\nu}\big)(y,\Lambda)\Big\}\Big|\\
   &\lesssim \frac{1}{(|\lambda-b(y)|+|\nu/k|^{1/3})^3} |\nu/k|\frac{|k|^{-2}M_k}{|\nu/k|^{1/2}(|\nu/k|^{1/3}+|\lambda-b(y)|)^{-1/2}}\\
   &\lesssim \frac{|\nu/k|^{-1/3}}{(1+|\nu/k|^{-1/3}|\lambda-b(y)|)^{5/2}}|k|^{-2}M_k.
\end{split}
\end{equation}
The desired bounds \eqref{pmpr1.9} for $h=A_\Theta^{-1}\Big[\frac{\nu}{k}\partial_ya(y,\lambda)\partial_y^2\omega_{k,\nu}\Big]$ follow from \eqref{pmpr1.14}, and \eqref{pmpr1.16}-\eqref{pmpr1.17}. This completes the proof of \eqref{pmpr1.9} and therefore also the bound \eqref{mainprop2} for $\sigma=1$.

\subsubsection{The case $\sigma=2$ for statement (i) of Proposition \ref{mainprop}}
Similar calculations to \eqref{refb7}-\eqref{refb10} below (which we again omit to avoid repetitive details) show that
\begin{equation}\label{pmpr1.18}
\begin{split}
&(\partial_y^2-k^2) D^2_\lambda\psi_{k,\nu}(y,\Lambda)+A_\Theta^{-1}\big(ib''(\cdot)D^2_\lambda\psi_{k,\nu}(\cdot,\Lambda)\big)(y)\\
&= A_\Theta^{-1}g_{0k}^{\ast\ast}(y,\Lambda)+\Big[\partial_y^2a(y,\lambda)\partial_yD_\lambda\psi_{k,\nu}+2\partial_ya(y,\lambda)\partial^2_yD_\lambda\psi_{k,\nu}\Big](y,\Lambda),
\end{split}
\end{equation}
where we have defined 
\begin{equation}\label{pmpr1.19}
\begin{split}
g_{0k}^{\ast\ast}(y,\Lambda):=&D_\lambda\Big\{g_{0k}^{\ast}+A_\Theta\Big[\partial_y^2a(y,\lambda)\partial_y\psi_{k,\nu}+2\partial_ya(y,\lambda)\partial^2_y\psi_{k,\nu}\Big](y,\Lambda)\Big\}\\
& +\Big[\frac{\nu}{k}\partial_y^2a(y,\lambda)(\partial_y^2-k^2)\partial_y+2\frac{\nu}{k}\partial_ya(y,\lambda)(\partial_y^2-k^2)\partial_y^2\Big]D_\lambda\psi_{k,\nu}(y,\Lambda)\\
& +i\Big[a(y,\lambda)b'(y)-1\Big](\partial_y^2-k^2)D_\lambda \psi_{k,\nu}(y,\Lambda)-ia(y,\lambda)b'''(y)D_\lambda \psi_{k,\nu}(y,\Lambda).
\end{split}
\end{equation}
As in the case of $\sigma=1$, it suffices to show that for 
\begin{equation}\label{pmpr1.20}
\begin{split}
&h\in A_\Theta^{-1}\Big\{D_\lambda g_{0k}^{\ast},D_\lambda A_\Theta\Big[\partial_y^2a(y,\lambda)\partial_y\psi_{k,\nu}+2\partial_ya(y,\lambda)\partial^2_y\psi_{k,\nu}\Big],\\
&\frac{\nu}{k}\partial_y^2a(y,\lambda)(\partial_y^2-k^2)\partial_yD_\lambda\psi_{k,\nu}(y,\Lambda),\frac{\nu}{k}\partial_ya(y,\lambda)(\partial_y^2-k^2)\partial_y^2D_\lambda\psi_{k,\nu}(y,\Lambda),\\
&\Big[a(y,\lambda)b'(y)-1\Big](\partial_y^2-k^2)D_\lambda \psi_{k,\nu}(y,\Lambda),a(y,\lambda)b'''(y)D_\lambda \psi_{k,\nu}(y,\Lambda)\Big\},
\end{split}
\end{equation}
and for
\begin{equation}\label{pmpr1.21}
h\in\Big\{\partial_y^2a(y,\lambda)\partial_yD_\lambda\psi_{k,\nu},\partial_ya(y,\lambda)\partial^2_yD_\lambda\psi_{k,\nu}\Big\},
\end{equation}
we have the bound 
\begin{equation}\label{pmpr1.22}
    \big\|(k^2-\partial_y^2)^{-1}h\big\|_{H^1_k(\T_\tp)}\lesssim \langle \lambda\rangle^{-3}|k|^{-1/2}M_k.
\end{equation}
To prove the bounds \eqref{pmpr1.22} we shall use the following commutator relations for any $q(x,\lambda)$ with sufficient regularity,
\begin{equation}\label{pmpr1.23}
\begin{split}
   \big[A_\Theta, D_\lambda\big]&=i\big(a(y,\lambda)b'(y)-1\big)+\frac{\nu}{k}\partial_y^2a(y,\lambda)\partial_y+2\frac{\nu}{k}\partial_ya(y,\lambda)\partial_y^2,\\
  \big[D_\lambda, q(x,\lambda)\big]&=D_\lambda q(x,\lambda),\,\,\big[D_\lambda, \partial_y\big]=-\partial_ya(y,\lambda)\partial_y.
    \end{split}
\end{equation}
We observe from \eqref{pmpr1.23} that the commutator terms behave better than the uncommutated terms, by either gaining a factor of $\nu/k$ or one derivative. 

Using equation \eqref{pmpr1.5}, the commutator relations \eqref{pmpr1.23} and the identity
\begin{equation}\label{pmpr1.231}
    \partial^2_yD_\lambda\psi_{k,\nu}(y,\Lambda)=\big( \partial^2_y-k^2\big)D_\lambda\psi_{k,\nu}(y,\Lambda)+k^2D_\lambda\psi_{k,\nu}(y,\Lambda),
\end{equation}
the desired bounds for 
\begin{equation}\label{pmpr1.24}
    h\in\Big\{\partial_y^2a(y,\lambda)\partial_yD_\lambda\psi_{k,\nu},\partial_ya(y,\lambda)\partial^2_yD_\lambda\psi_{k,\nu}\Big\}
\end{equation}
and 
\begin{equation}\label{pmpr1.25}
   \begin{split}
&h\in A_\Theta^{-1}\Big\{D_\lambda A_\Theta\Big[\partial_y^2a(y,\lambda)\partial_y\psi_{k,\nu}+2\partial_ya(y,\lambda)\partial^2_y\psi_{k,\nu}\Big],\\
&\frac{\nu}{k}\partial_y^2a(y,\lambda)(\partial_y^2-k^2)\partial_yD_\lambda\psi_{k,\nu}(y,\Lambda),\frac{\nu}{k}\partial_ya(y,\lambda)(\partial_y^2-k^2)\partial_y^2D_\lambda\psi_{k,\nu}(y,\Lambda),\\
&\Big[a(y,\lambda)b'(y)-1\Big](\partial_y^2-k^2)D_\lambda \psi_{k,\nu}(y,\Lambda),a(y,\lambda)b'''(y)D_\lambda \psi_{k,\nu}(y,\Lambda)\Big\},
\end{split}
\end{equation}
follow from the bounds \eqref{mainprop3} for $\sigma\in\{0,1\}$ and the arguments as in proof of the cases $\sigma\in\{0,1\}$, after peeling off the easier commutator terms. 

It remains to prove \eqref{pmpr1.22} for $h=A_\Theta^{-1}D_\lambda g_{0k}^{\ast}$. Using the definition \eqref{pmpr1.6}, it suffices to prove \eqref{pmpr1.22} for
\begin{equation}\label{pmpr1.26}
 \begin{split}
 h\in A_\Theta^{-1}D_\lambda\Big\{&\,D_\lambda \omega_{0k}(y),\,\, \frac{\nu}{k}\partial_y^2a(y,\lambda)\partial_y\omega_{k,\nu}(y,\Lambda),\\
&\frac{\nu}{k}\partial_ya(y,\lambda)\partial_y^2\omega_{k,\nu}(y,\Lambda), \,\,b'''(y)a(y,\lambda)\psi_{k,\nu}(y,\Lambda),\\
&\big[b'(y)a(y,\lambda)-1\big]\omega_{k,\nu}(y,\Lambda)\Big\}.
\end{split}  
\end{equation}
Following similar arguments as in the proof for $h$ as in \eqref{pmpr1.24}-\eqref{pmpr1.25} and considering only the more singular terms, it suffices to prove \eqref{pmpr1.22} for 
\begin{equation}\label{pmpr1.27}
 \begin{split}
 h\in A_\Theta^{-1}D_\lambda\Big\{&\,\frac{\nu}{k}\partial_ya(y,\lambda)\partial_y^2\omega_{k,\nu}(y,\Lambda), \big[b'(y)a(y,\lambda)-1\big]\omega_{k,\nu}(y,\Lambda)\Big\}.
\end{split}  \end{equation}
After peeling off the easier commutator terms using \eqref{pmpr1.23}, we reduce to prove \eqref{pmpr1.22} for 
\begin{equation}\label{pmpr1.28}
 \begin{split}
 h\in A_\Theta^{-1}\Big\{&\,\frac{\nu}{k}\partial_ya(y,\lambda)\partial_y^2(\partial_y^2-k^2)D_\lambda\psi_{k,\nu}(y,\Lambda), \\
 &\big[b'(y)a(y,\lambda)-1\big](\partial_y^2-k^2)D_\lambda\psi_{k,\nu}(y,\Lambda)\Big\}.
\end{split}  \end{equation}
The desired bounds \eqref{pmpr1.22} for $h=A_\Theta^{-1}\big\{
 \big[b'(y)a(y,\lambda)-1\big](\partial_y^2-k^2)D_\lambda\psi_{k,\nu}\big\}(y,\Lambda)$ follows from Lemma \ref{ake6} noting that $b'a-1$ vanishes in a neighborhood of $\{y\in\T_\tp:\,b(y)=\lambda\}$. The proof of \eqref{pmpr1.22} for $A_\Theta^{-1}\big\{\frac{\nu}{k}\partial_ya(y,\lambda)\partial_y^2(\partial_y^2-k^2)D_\lambda\psi_{k,\nu}\big\}(y,\Lambda)$ is similar to \eqref{pmpr1.13}-\eqref{pmpr1.17}.

\subsection{Proof of statement (ii) in Proposition \ref{mainprop}}
We choose $\gamma_0,\gamma_1,\gamma_2\in(1,2)$ with $\gamma<\gamma_2<\gamma_1<\gamma_0<2$, and $\nu_0$ sufficiently small such that Proposition \ref{mainprop} holds for $\gamma_0,\gamma_1,\gamma_2$. 

Fix $j\in\{1,2\}$. We provide the detailed proof only for $\lambda\in b(\T_\tp)\cap\Sigma_{j,\delta_0}$ with $|\lambda-b(y_{j\ast})|\gtrsim |\nu/k|^{1/2}$. The case when $|\lambda-b(y_{j\ast})|\ll |\nu/k|^{1/2}$ or $\lambda\in \Sigma_{j,\delta_0}\backslash b(\T_\tp)$ is easier since the singular behavior of $A_{\Theta}^{-1}$ for $|y-y_{j\ast}|\lesssim |\nu/k|^{1/4}$ is compensated by the decay property of the modified Green's function and we can use Lemma \ref{ake1} and Lemma \ref{ake6} to get effective bounds.

The following weighted bounds are very useful to treat various nonlocal contributions below. 

\begin{lemma}\label{LAMweight}
We have the following pointwise inequality for $\sigma_1\in(-1,1),\sigma_2\in\R$ with $\sigma_1+\sigma_2\in(-1,2)$, any $h\in L^\infty(\T_\tp)$ and $y\in\T_\tp$, 
\begin{equation}\label{pmpr2.03}
\int_{\T_\tp}\Big|\mathcal{G}_k^j(y,z;\Lambda)\varrho_j^{\sigma_1-1}\varrho_{j,k}^{\sigma_2-1}(z;\Lambda)h(z)\Big|dz\lesssim_{\sigma_1,\sigma_2} \varrho_j^{\sigma_1-1}\varrho_{j,k}^{\sigma_2+1}(y;\Lambda)\|h\|_{L^\infty(\T_\tp)}.
\end{equation}
\end{lemma}

The proof follows from direct calculations using the bounds on the modified Green's function in Lemma \ref{mGk50}. 

The inequality \eqref{pmpr2.03} is useful when $|\lambda- b(y)|\gtrsim |\lambda-b(y_{j\ast})|$ so that $A^{-1}_\Theta$ is not singular. To handle the case when $|\lambda- b(y)|\ll |\lambda-b(y_{j\ast})|$, in view of Lemma \ref{ake6} and adopting the notation there, we use the following bounds. Denote $\epsilon:=\nu/k$ and define for $y\in\T_\tp$, 
\begin{equation}\label{pmpr2.031}
    \mathcal{I}_{k,\nu}(y,\lambda):=\int_{\T_\tp}\mathcal{G}_k^j(y,z;\Lambda)\sum_{\dagger\in\{+,-\}}\frac{\varphi_0\big((y-y^\dagger_\lambda)/|\lambda-b(y_{j\ast})|^{1/2}\big)}{ib'(y^\dagger_\lambda)(y-y^\dagger_\lambda)+ |\epsilon|^{1/3}|\lambda-b(y_{j\ast})|^{1/3}}dz.
\end{equation}
Then we have,
\begin{equation}\label{pmpr2.04}
\begin{split}
&d^{-1/2}_{j,k}\big(\delta(\Lambda)\big)^{-1}\sum_{\beta\in\{0,1\}}\big\|d^\beta_{j,k}(\Lambda)\partial^\beta_y\mathcal{I}_{k,\nu}(y,\lambda)\big\|_{L^2(S^j_{8\delta})}\\
&+d^{-2}_{j,k}(\Lambda)\sum_{\beta\in\{0,1\}}\big\|\varrho_{j,k}^{1+\beta}(y,\Lambda)\partial^\beta_y\mathcal{I}_{k,\nu}(y,\lambda)\big\|_{L^\infty(\T_\tp\backslash S^j_{8\delta})}\lesssim |\lambda-b(y_{j\ast})|^{-1/2}.
\end{split}
\end{equation}

We prove the bounds for $D^\sigma_\lambda\, \psi_{k,\nu}^\ast, \sigma\in\{0,1,2\}$, respectively, organized as subsections.
\subsubsection{Proof of the bounds for $D^\sigma_\lambda\, \psi_{k,\nu}^\ast, \sigma=0$ in statement (ii) in Proposition \ref{mainprop}}
Using the bounds \eqref{lapj3} of the limiting absorption principle and the equation \eqref{mainprop0.31}-\eqref{mainprop0.32}, it suffices to show that 
\begin{equation}\label{pmpr2.01}
    \big\|\mathcal{G}_k^j(\Lambda)A_{k,\nu}^{-1}(\Lambda)f_{0k}(\cdot,\Lambda)\big\|_{X^{0,-\gamma_0}(\mathfrak{M})}\lesssim |k|^{-1/2}M_k.
\end{equation}
In the above, with a slight abuse of notation, we denote for $k\in\Z\cap[1,\infty), \nu\in(0,1/18), \lambda\in\cup_{j=1}^2\Sigma_{j,\delta_0}$, $\mathcal{G}_k^j(\Lambda): H^1_k(\T_\tp)\to H^1_k(\T_\tp)$ as follows. For any $h\in H^1_k(\T_\tp)$,
\begin{equation}\label{pmpr2.02}
\mathcal{G}_k^j(\Lambda)h(y):=\int_{\T_\tp}\mathcal{G}_k^j(y,z;\Lambda)h(z)dz. 
\end{equation}
Notice that for $y\in\T_\tp$, 
\begin{equation}\label{pmpr2.0201}
\begin{split}
    A_{\Theta}^{-1}f_{0k}(\cdot,\Lambda)=&\,A_{\Theta}^{-1}\Big[\omega_{0k}(y)-\sum_{j=1}^2\varphi_0\big((y-y_{j\ast})/\delta_0\big)\omega_{0k}(y)\Big]\\
&-(\partial_y^2-k^2)\Big[\sum_{j=1}^2\varphi_0\big((y-y_{j\ast})/\delta_0\big)\frac{\omega_{0k}(y)}{ib''(y)}\Big].
\end{split}
\end{equation}

The desired bounds \eqref{pmpr2.01} follow from Lemma \ref{ake1} and Lemma \ref{ake6} and the bounds on the modified Green's function. 

\subsubsection{Proof of the bounds for $D^\sigma_\lambda\, \psi_{k,\nu}^\ast, \sigma=1$ in statement (ii) in Proposition \ref{mainprop}}\label{casesigma1'}
We note that $\partial_\lambda\psi^\ast_{k,\nu}$ satisfies the equation
\begin{equation}\label{refb1}
\begin{split}
A_{\Theta}\partial_\lambda\omega^\ast_{k,\nu}(y,\Lambda)+ib''(y)\partial_\lambda\psi^\ast_{k,\nu}(y,\Lambda)&=-i\omega^\ast_{k,\nu}(y,\Lambda)\\
&\quad+\partial_\lambda f_{0k}(y,\Lambda),\\
(-k^2+\partial_y^2)\partial_\lambda\psi^\ast_{k,\nu}(y,\Lambda)&=\partial_\lambda\omega^\ast_{k,\nu}(y,\Lambda).
\end{split}\end{equation}
Direct computation shows that $a(y,\lambda)\partial_y\psi^\ast_{k,\nu}$ satisfies the equation for $y\in\T_\tp$,
\begin{equation}\label{refb3}
\begin{split}
&A_{\Theta}(\partial_y^2-k^2)\Big[a(y,\lambda)\partial_y\psi^\ast_{k,\nu}(y,\Lambda)\Big]+ib''(y)a(y,\lambda)\partial_y\psi^\ast_{k,\nu}(y,\Lambda)\\
&=a(y,\lambda)\partial_yf_{0k}(y,\Lambda)+A_{\Theta}\Big[\partial_y^2a(y,\lambda)\partial_y\psi^\ast_{k,\nu}+2\partial_ya(y,\lambda)\partial^2_y\psi^\ast_{k,\nu}\Big](y,\Lambda)\\
&\, +\frac{\nu}{k}\partial_y^2a(y,\lambda)\partial_y\omega^\ast_{k,\nu}(y,\Lambda)+2\frac{\nu}{k}\partial_ya(y,\lambda)\partial_y^2\omega^\ast_{k,\nu}(y,\Lambda)\\
&\,-ib'''(y)a(y,\lambda)\psi^\ast_{k,\nu}(y,\Lambda)+ib'(y)a(y,\lambda)\omega^\ast_{k,\nu}(y,\Lambda).
\end{split}
\end{equation}
It follows from equations \eqref{refb1} and \eqref{refb3} that $y\in\T_\tp$,
\begin{equation}\label{refb5}
\begin{split}
&(\partial_y^2-k^2)D_\lambda \psi^\ast_{k,\nu}(y,\Lambda)+A_\Theta^{-1} \big(ib''(\cdot)D_\lambda\psi^\ast_{k,\nu}(\cdot,\Lambda)\big)(y)\\
&=A_\Theta^{-1}\big(f_{0k}^{\ast}(\cdot,\Lambda)\big)(y)+\Big[\partial_y^2a(y,\lambda)\partial_y\psi^\ast_{k,\nu}+2\partial_ya(y,\lambda)\partial^2_y\psi^\ast_{k,\nu}\Big](y,\Lambda),
\end{split}
\end{equation}
where for $y\in\T_\tp$,
\begin{equation}\label{refb6}
\begin{split}
f_{0k}^{\ast}(y,\Lambda):=&\,\,D_\lambda f_{0k}(y,\Lambda)+\frac{\nu}{k}\partial_y^2a(y,\lambda)\partial_y\omega^\ast_{k,\nu}(y,\Lambda)\\
&+2\frac{\nu}{k}\partial_ya(y,\lambda)\partial_y^2\omega^\ast_{k,\nu}(y,\Lambda)-ib'''(y)a(y,\lambda)\psi^\ast_{k,\nu}(y,\Lambda)\\
&+i\big[b'(y)a(y,\lambda)-1\big]\omega^\ast_{k,\nu}(y,\Lambda).
\end{split}
\end{equation}
Using the bounds \eqref{lapj3} of the limiting absorption principle and the equation \eqref{refb5}-\eqref{refb6}, it suffices to show that 
\begin{equation}\label{refb6.01}
    \big\|\mathcal{G}_k^j(\Lambda)h\big\|_{X^{1,-\gamma+1}(\mathfrak{M})}\lesssim |k|^{-1/2}M_k,
\end{equation}
where
\begin{equation}\label{refb6.02}
 h\in\Big\{A_\Theta^{-1}\big(f_{0k}^{\ast}(\cdot,\Lambda)\big)(y), \Big[\partial_y^2a(y,\lambda)\partial_y\psi^\ast_{k,\nu}+2\partial_ya(y,\lambda)\partial^2_y\psi^\ast_{k,\nu}\Big](y,\Lambda)\Big\}.  
\end{equation}
To prove \eqref{refb6.01}, we divide into the singular region for $y\in \T_\tp$ with $|\lambda-b(y)|\ll |\lambda-b(y_{j\ast})|$ and the non-singular region for $y\in\T_\tp$ with $|\lambda-b(y)|\gtrsim |\lambda-b(y_{j\ast})|$. More precisely, using the notations in Lemma \ref{ake6}, it suffices to prove that for $h$ as in \eqref{refb6.02},
\begin{equation}\label{refb6.03}
      \Big\|\mathcal{G}_k^j(\Lambda) \Big[1-\sum_{\dagger\in\{+,-\}}\varphi_0\big((y-y^\dagger_\lambda)/|\lambda-b(y_{j\ast})|^{1/2}\big)\Big]h\Big\|_{X^{1,-\gamma_1+1}(\mathfrak{M})}\lesssim |k|^{-1/2}M_k,  
\end{equation}
and for $\dagger\in\{+,-\}$
\begin{equation}\label{refb6.04}
        \big\|\mathcal{G}_k^j(\Lambda) \varphi_0\big((y-y^\dagger_\lambda)/|\lambda-b(y_{j\ast})|^{1/2}\big)h\big\|_{X^{1,-\gamma_1+1}(\mathfrak{M})}\lesssim |k|^{-1/2}M_k. 
\end{equation}
To establish \eqref{refb6.03}, we notice the bounds that for $y\in\T_\tp, \dagger\in\{+,-\}$ and $h$ as in \eqref{refb6.02},
\begin{equation}\label{refb6.05}
    \Big|\Big[1-\sum_{\dagger\in\{+,-\}}\varphi_0\big((y-y^\dagger_\lambda)/|\lambda-b(y_{j\ast})|^{1/2}\big)\Big]h(y)\Big|\lesssim \varrho_{j,k}^{\gamma_0-2}\varrho_j^{-2}(y;\Lambda)\lesssim \varrho_{j,k}^{\gamma_1-2}\varrho_j^{\gamma_0-\gamma_1-2}(y;\Lambda). 
\end{equation}
The desired bounds \eqref{refb6.03} then follow from the estimates \eqref{pmpr2.02}. 

To establish \eqref{refb6.04}, we use similar arguments as in section \ref{casesigma1}, since the weights are constant over the support of $\varphi\big((y-y^\dagger_\lambda)/|\lambda-b(y_{j\ast})|^{1/2}\big)$ and can be taken out of the estimates when appropriately accounted for. We consider the terms 
\begin{equation}\label{refbs1}
   h\in A_\Theta^{-1}\Big\{ D_\lambda f_{0k},\frac{\nu}{k}\partial_y^2a\partial_y\omega^\ast_{k,\nu},
\frac{\nu}{k}\partial_ya\partial_y^2\omega^\ast_{k,\nu},b'''(y)a(y,\lambda)\psi^\ast_{k,\nu},\big[b'(y)a(y,\lambda)-1\big]\omega^\ast_{k,\nu}\Big\}
\end{equation}
and 
\begin{equation}\label{refbs2}
    h\in\Big\{\partial_y^2a(y,\lambda)\partial_y\psi^\ast_{k,\nu}(y,\Lambda), \partial_ya(y,\lambda)\partial^2_y\psi^\ast_{k,\nu}(y,\Lambda)\Big\}.
\end{equation}
Using \eqref{mainprop0.32} and setting for $y\in\T_\tp$,
$$p(y)=(\partial_y^2-k^2)\Big[\sum_{j=1}^2\varphi\big((y-y_{j\ast})/\delta_0\big)\omega_{0k}(y)\Big],$$
by direct calculations we obtain that for $y\in\T_\tp$,
\begin{equation}\label{refbs2.01}
\begin{split}
    D_\lambda f_{0k}=&a\partial_y\Big[\omega_{0k}(y)-\sum_{j=1}^2\varphi\big((y-y_{j\ast})/\delta_0\big)\omega_{0k}(y)\Big]+i(ab'-1)p\\
&-A_{\Theta}\Big[a\,\partial_yp\Big]+\frac{\nu}{k}\Big[2\partial_y(\partial_y\partial_yp)-\partial_y^2a\partial_yp\Big],
    \end{split}
\end{equation}
To prove the desired bound \eqref{refb6.04} for $h=A_{\Theta}^{-1} D_\lambda f_{0k}$, by \eqref{refbs2.01}, noting the support property of the first and second terms on the right hand side of \eqref{refbs2.01} and using integration by parts in the third term after cancelling $A_{\Theta}$, it suffices to consider the term 
\begin{equation}\label{refbs2.02}
    h=A_{\Theta}^{-1}\frac{\nu}{k}\Big[2\partial_y(\partial_y\partial_yp)-\partial_y^2a\partial_yp\Big],
\end{equation}
which follows from Lemma \ref{ake1}. 

The desired bounds \eqref{refb6.04} for \begin{equation}\label{refbs6.06}
   h\in A_{\Theta}^{-1}\Big\{\frac{\nu}{k}\partial_y^2a\partial_y\omega^\ast_{k,\nu},
\frac{\nu}{k}\partial_ya\partial_y^2\omega^\ast_{k,\nu},\big[b'(y)a(y,\lambda)-1\big]\omega^\ast_{k,\nu}\Big\}
\end{equation}
follow from Lemma \ref{ake1} and equation \eqref{mainprop0.31} noting the support property of $b'a-1$. The desired bounds for 
\begin{equation}\label{refbs6.07}
    h\in\Big\{b'''(y)a(y,\lambda)\psi^\ast_{k,\nu},\partial_y^2a(y,\lambda)\partial_y\psi^\ast_{k,\nu}(y,\Lambda), \partial_ya(y,\lambda)\partial^2_y\psi^\ast_{k,\nu}(y,\Lambda)\Big\}
\end{equation}
follow from equation \eqref{mainprop0.31} and Lemma \eqref{ake6}, as in \eqref{pmpr1.11}. We omit the lengthy but straightforward calculations and refer to \eqref{pmpr1.12}-\eqref{pmpr1.17} for details. 

\subsubsection{Proof of the bounds for $D^\sigma_\lambda\, \psi_{k,\nu}^\ast, \sigma=2$ in statement (ii) in Proposition \ref{mainprop}}
Similar calculations show that for $y\in\T_\tp$,
\begin{equation}\label{refb7}
\begin{split}
&A_{\Theta}(\partial_y^2-k^2)\partial_\lambda D_\lambda\psi^\ast_{k,\nu}(y,\Lambda)+ib''(y)\partial_\lambda D_\lambda\psi^\ast_{k,\nu}(y,\Lambda)\\
&=\partial_\lambda f_{0k}^{\ast}(y,\Lambda)+\partial_\lambda \Big\{A_{\Theta}\Big[\partial_y^2a(y,\lambda)\partial_y\psi^\ast_{k,\nu}+2\partial_ya(y,\lambda)\partial^2_y\psi^\ast_{k,\nu}\Big](y,\Lambda)\Big\}\\
&\quad -i(\partial_y^2-k^2) D_\lambda\psi^\ast_{k,\nu}(y,\Lambda),
\end{split}
\end{equation}
and 
\begin{equation}\label{refb8}
\begin{split}
&A_{\Theta}(\partial_y^2-k^2)\Big(a(y,\lambda)\partial_y D_\lambda\psi^\ast_{k,\nu}(y,\Lambda)\Big)+ib''(y)a(y,\lambda)\partial_y D_\lambda\psi^\ast_{k,\nu}(y,\Lambda)\\
&= a(y,\lambda)\partial_y\Big\{f_{0k}^{\ast}+A_{\Theta}\Big[\partial_y^2a(y,\lambda)\partial_y\psi^\ast_{k,\nu}+2\partial_ya(y,\lambda)\partial^2_y\psi^\ast_{k,\nu}\Big](y,\Lambda)\Big\}\\
&\quad +ia(y,\lambda)b'(y)(\partial_y^2-k^2)D_\lambda \psi^\ast_{k,\nu}(y,\Lambda)-ia(y,\lambda)b'''(y)D_\lambda \psi^\ast_{k,\nu}(y,\Lambda)\\
&\quad +\Big[\frac{\nu}{k}\partial_y^2a(y,\lambda)(\partial_y^2-k^2)\partial_y+2\frac{\nu}{k}\partial_ya(y,\lambda)(\partial_y^2-k^2)\partial_y^2\Big]D_\lambda\psi^\ast_{k,\nu}(y,\Lambda)\\
&\quad +A_{\Theta}\Big[\partial_y^2a(y,\lambda)\partial_yD_\lambda\psi^\ast_{k,\nu}+2\partial_ya(y,\lambda)\partial^2_yD_\lambda\psi^\ast_{k,\nu}\Big](y,\Lambda).
\end{split}
\end{equation}
Therefore, $D_\lambda^2\psi^\ast_{k,\nu}$ satisfies the equation for $y\in\T_\tp$,
\begin{equation}\label{refb9}
\begin{split}
&(\partial_y^2-k^2) D^2_\lambda\psi^\ast_{k,\nu}(y,\Lambda)+A_{\Theta}^{-1}\big(ib''(\cdot)D^2_\lambda\psi^\ast_{k,\nu}(\cdot,\Lambda)\big)(y)\\
&= A_{\Theta}^{-1}f_{0k}^{\ast\ast}(y,\Lambda)+\Big[\partial_y^2a(y,\lambda)\partial_yD_\lambda\psi^\ast_{k,\nu}+2\partial_ya(y,\lambda)\partial^2_yD_\lambda\psi^\ast_{k,\nu}\Big](y,\Lambda),
\end{split}
\end{equation}
where we have defined for $y\in\T_\tp$,
\begin{equation}\label{refb10}
\begin{split}
f_{0k}^{\ast\ast}(y,\Lambda):=&D_\lambda\Big\{f_{0k}^{\ast}+A_\Theta\Big[\partial_y^2a(y,\lambda)\partial_y\psi^\ast_{k,\nu}+2\partial_ya(y,\lambda)\partial^2_y\psi^\ast_{k,\nu}\Big](y,\Lambda)\Big\}\\
& +\Big[\frac{\nu}{k}\partial_y^2a(y,\lambda)(\partial_y^2-k^2)\partial_y+2\frac{\nu}{k}\partial_ya(y,\lambda)(\partial_y^2-k^2)\partial_y^2\Big]D_\lambda\psi^\ast_{k,\nu}(y,\Lambda)\\
& +i\Big[a(y,\lambda)b'(y)-1\Big](\partial_y^2-k^2)D_\lambda \psi^\ast_{k,\nu}(y,\Lambda)-ia(y,\lambda)b'''(y)D_\lambda \psi^\ast_{k,\nu}(y,\Lambda).
\end{split}
\end{equation}

Using the bounds \eqref{lapj3} of the limiting absorption principle and the equation \eqref{refb5}-\eqref{refb6}, it suffices to show that 
\begin{equation}\label{refb7.01}
\big\|\mathcal{G}_k^j(\Lambda)h\big\|_{X^{1-2(2-\gamma_2),-\gamma_2+2}(\mathfrak{M})}\lesssim |k|^{-1/2}(\delta(\Lambda))^{-1-2(2-\gamma_2)}M_k,
\end{equation}
where
\begin{equation}\label{refb7.02}
 h\in\Big\{A_\Theta^{-1}\big(f_{0k}^{\ast\ast}(\cdot,\Lambda)\big)(y), \partial_y^2a(y,\lambda)\partial_yD_\lambda\psi^\ast_{k,\nu}, \partial_ya(y,\lambda)\partial^2_yD_\lambda\psi^\ast_{k,\nu}\Big\}.  
\end{equation}
To prove the bounds \eqref{refb7.01}-\eqref{refb7.02}, as in section \ref{casesigma1'}, we divide we divide into the singular regions for $y\in \T_\tp$ with $|\lambda-b(y)|\ll |\lambda-b(y_{j\ast})|$ and non-singular regions for $y\in\T_\tp$ with $|\lambda-b(y)|\gtrsim |\lambda-b(y_{j\ast})|$. More precisely, using the notation as in Lemma \ref{ake6}, it suffices to prove that for $h$ as in \eqref{refb7.02},
\begin{equation}\label{refb7.03}
\begin{split}
      &\Big\|\mathcal{G}_k^j(\Lambda) \Big[1-\sum_{\dagger\in\{+,-\}}\varphi\big((y-y^\dagger_\lambda)/|\lambda-b(y_{j\ast})|^{1/2}\big)\Big]h\Big\|_{X^{1-2(2-\gamma_2),-\gamma_2+2}(\mathfrak{M})}\\
      &\lesssim |k|^{-1/2}(\delta(\Lambda))^{-1-2(2-\gamma_2)}M_k,   \end{split}
\end{equation}
and for $\dagger\in\{+,-\}$,
\begin{equation}\label{refb7.04}
\begin{split}
        &\big\|\mathcal{G}_k^j(\Lambda) \varphi\big((y-y^\dagger_\lambda)/|\lambda-b(y_{j\ast})|^{1/2}\big)h\big\|_{X^{1-2(2-\gamma_2),-\gamma_2+2}(\mathfrak{M})}\\
        &\lesssim |k|^{-1/2}(\delta(\Lambda))^{-1-2(2-\gamma_2)}M_k. 
        \end{split}
\end{equation}
To establish \eqref{refb7.03}, using Lemma \eqref{ake1}, \eqref{mainprop2} with $\sigma\in\{0,1\}$ and equation \eqref{mainprop0.31} and \eqref{refb5}-\eqref{refb6}, we obtain the bounds that for $y\in\T_\tp, \dagger\in\{+,-\}$ and $h$ as in \eqref{refb7.02},
\begin{equation}\label{refb7.05}
\begin{split}
   & \Big|\Big[1-\sum_{\dagger\in\{+,-\}}\varphi\big((y-y^\dagger_\lambda)/|\lambda-b(y_{j\ast})|^{1/2}\big)\Big]h(y)\Big|\\
   &\lesssim \varrho_{j,k}^{\gamma_1-3}\varrho_j^{-3}(y;\Lambda)\lesssim \delta^{-1-2(2-\gamma_2)}\varrho_j^{-2+2(2-\gamma_2)}\varrho_{j,k}^{\gamma_2-3}(y;\Lambda).
    \end{split}
\end{equation}
The desired bounds then follow from the estimates \eqref{pmpr2.02}. 

We now turn to the proof of \eqref{refb7.04} for $h$ as in \eqref{refb7.02}. 
The bounds \eqref{refb7.04} for 
\begin{equation}\label{refb7.020}
 h\in\Big\{\partial_y^2a(y,\lambda)\partial_yD_\lambda\psi^\ast_{k,\nu}, \partial_ya(y,\lambda)\partial^2_yD_\lambda\psi^\ast_{k,\nu}\Big\}
\end{equation}
follows from the bounds \eqref{mainprop2} with $\sigma=1$ and integration by parts using the bounds on the modified Green's function. To establish the bound \eqref{refb7.03} for $h=A_\Theta^{-1}\big(f_{0k}^{\ast\ast}(\cdot,\lambda,\nu,\alpha)\big)(y)$, using the commutator relations \eqref{pmpr1.23} and noting that the terms associated with the commutators are easier to handle, and in view of the support property of $ab'-1$, it suffices to prove \eqref{refb7.04} for
\begin{equation}
   \begin{split}
&h\in A_{\Theta}^{-1}\Big\{D_\lambda f_{0k}^{\ast}+A_{\Theta}\Big[\partial_y^2a(y,\lambda)\partial_y\psi^\ast_{k,\nu}+2\partial_ya(y,\lambda)\partial^2_y\psi^\ast_{k,\nu}\Big](y,\Lambda)\\
&\quad +\Big[\frac{\nu}{k}\partial_y^2a(y,\lambda)(\partial_y^2-k^2)\partial_y+2\frac{\nu}{k}\partial_ya(y,\lambda)(\partial_y^2-k^2)\partial_y^2\Big]D_\lambda\psi^\ast_{k,\nu}(y,\Lambda)\Big\}.
\end{split}
\end{equation}
The bound \eqref{refb7.04} follows for 
\begin{equation}
h=\Big[\partial_y^2a(y,\lambda)\partial_y\psi^\ast_{k,\nu}+2\partial_ya(y,\lambda)\partial^2_y\psi^\ast_{k,\nu}\Big](y,\Lambda)
\end{equation}
follows from the bounds on the modified Green's function, integration by parts and the bounds \eqref{lapj3} for $\sigma=0$. The bounds \eqref{refb7.04} for 
\begin{equation}\label{refb7.07} 
    h=A_\Theta^{-1}\Big[\frac{\nu}{k}\partial_y^2a(y,\lambda)(\partial_y^2-k^2)\partial_y+2\frac{\nu}{k}\partial_ya(y,\lambda)(\partial_y^2-k^2)\partial_y^2\Big]D_\lambda\psi^\ast_{k,\nu}(y,\Lambda)
\end{equation}
follows from equation \eqref{refb5} using argument similar to \eqref{pmpr1.13}-\eqref{pmpr1.17}. It remains to consider the term 
\begin{equation}\label{refb7.08}
    h=A_{\Theta}^{-1}D_\lambda f_{0k}^{\ast}.
\end{equation}
Recalling the identity \eqref{refb6} and skipping terms that can be bounded that are similar as before, we focus only on proving \eqref{refb7.04} for
\begin{equation}\label{refb7.09}
    h=A_{\Theta}^{-1}D^2_\lambda f_{0k}. 
\end{equation}
Recalling the formula \eqref{mainprop0.32}, using the commutator relations \eqref{pmpr1.23} and focusing on the most singular terms, it suffices to consider the term 
\begin{equation}\label{refb7.10}
    h=(\partial_y^2-k^2)D_\lambda^2\Big[\sum_{j=1}^2\varphi\big((y-y_{j\ast})/\delta_0\big)\omega_{0k}(y)\Big]. 
\end{equation}
The desired bound \eqref{refb7.04} for $h$ in \eqref{refb7.10} follow from simple calculations using the bounds on the modified Green's function.

The proof of Proposition \ref{mainprop} is now complete.


\section{Proof of the main theorem}\label{sec:pmt}
In this section we give the proof of the main theorem \ref{thm}, using the limiting absorption principle, see Proposition \ref{mainprop} and Corollary \ref{RDS1}, with $\gamma_\ast=\gamma+\frac{2-\gamma}{2}$ for sufficiently small $\nu>0$ depending on $\gamma$, and the representation formula \eqref{intP6}. 

As in \cite{JiaUM}, the representation formulae \eqref{intP6} for $\omega_0\in C^\infty(\T_\tp)$ follow from the theory of semigroups of sectorial operators, see section 4, Chapter II of \cite{Engel}. We omit the standard details and refer to \cite{Engel}. The main task is to establish the following bounds. 
 \begin{proposition}\label{PROrep}
    Assume that $k\in\mathbb{Z}\cap[1,\infty)$ and $\omega_{0k}\in C^{\infty}(\T_\tp)$. Define 
    \begin{equation}
        \aleph :=\{\mu\in\mathbb{C}: \Re \mu\geq-\sigma_\sharp|\nu/k|^{1/2}\}.
    \end{equation}
    For any $\mu=\alpha-i\lambda\in\aleph$ and $\nu\in(0,\nu_0)$, \eqref{intP5} is solvable for a unique solution $\omega_{k,\nu}\in H^2(\T_{\tp})$ which is holomorphic in $\mu$ in the interior of $\aleph$ and continuous on $\aleph$. In addition, we have the bound
    \begin{equation}\label{resest1}
        \|\omega_{k,\nu}(\cdot,\Lambda)\|_{L^2(\T_\tp)}\lesssim_{k,\nu,\omega_{0k}}\frac{1}{\lb\lambda,\alpha\rb},
    \end{equation}
    and
    \begin{equation}\label{resest2}
        \|\partial_{\lambda}\omega_{k,\nu}(\cdot,\Lambda)\|_{L^2(\T_\tp)}\lesssim_{k,\nu,\omega_{0k}}\frac{1}{\lb\lambda,\alpha\rb^2},
    \end{equation}
\end{proposition}
Proposition \ref{PROrep} follows from the limiting absorption principle, see Proposition \ref{lap_main}, which implies that there is no eigenvalue of $L_{k,\nu}$ inside $\aleph$. The absence of eigenvalues inside $\aleph$ is sufficient to deduce \eqref{resest1}-\eqref{resest2} since we allow the bounds to depend on $\nu$ and consequently the problem is now elliptic.

We divide the proof of Theorem \ref{thm} into three different regions (i) $\min\{|y-y_{1\ast}|, |y-y_{2\ast}|\}\gtrsim 1$, (ii) $|\nu/k|^{1/4}\ll\min\{|y-y_{1\ast}|, |y-y_{2\ast}|\}\ll 1$, and (iii) $\min\{|y-y_{1\ast}|, |y-y_{2\ast}|\}\ll |\nu/k|^{1/4}$. We assume that $t\ge1$ as the case of $0<t<1$ follows from the local wellposedness of the Navier-Stokes equations. 

Denote
\begin{equation}\label{pmt0.1}
    M_k:=\big\|\omega_{0k}\big\|_{H^3_k(\T_\tp)},
\end{equation}
and recall the definition \eqref{refb2} for $a(y,\lambda)$ with $y\in\T_\tp, \lambda\in\R$, 
\begin{equation}\label{pmt0.11}
a(y,\lambda):=\Big[1-\varphi\big(\frac{y-y_{1\ast}}{\delta_2(\lambda)}\big)\Big]\Big[1-\varphi\big(\frac{y-y_{2\ast}}{\delta_2(\lambda)}\big)\Big]\frac{1}{b'(y)}.
\end{equation}
We shall use the following formulae (recall also \eqref{mainprop0.2} for the derivative $D_\lambda$), 
\begin{equation}\label{pmt0.12}
\begin{split}
\partial_\lambda&=D_\lambda-a(y,\lambda)\partial_y,\\
\partial_y\partial_\lambda&=\partial_yD_\lambda-\partial_ya(y,\lambda)\partial_y-a(y,\lambda)\partial_y^2,\\
\partial_\lambda^2&=D_\lambda^2-2a(y,\lambda)\partial_yD_\lambda-\partial_\lambda a(y,\lambda)\partial_y+a(y,\lambda)\partial_ya(y,\lambda)\partial_y+a^2(y,\lambda)\partial_y^2.
\end{split}
\end{equation}
The main tools are Proposition \ref{mainprop}, Corollary \ref{RDS1}, and the representation formula \eqref{intP6}. More specifically we use the identities, with $\alpha=-\sigma_\sharp|\nu/k|^{1/2}$, that for $y\in\T_\tp$ and $t\ge 1$,
\begin{equation}\label{pmt0.13}
    \psi_k(t,y)=\frac{1}{2\pi k^2t^2}e^{\alpha |k|t} e^{-\nu k^2 t}\int_\R e^{-ik\lambda t} \partial_\lambda^2\psi_{k,\nu}^\ast(y,\Lambda)\,d\lambda,
\end{equation}
\begin{equation}\label{pmt0.135}
    \partial_y\psi_k(t,y)=-\frac{1}{2i\pi kt}e^{\alpha |k|t} e^{-\nu k^2 t}\int_\R e^{-ik\lambda t} \partial_y\partial_\lambda\psi_{k,\nu}^\ast(y,\Lambda)\,d\lambda,
\end{equation}
and 
\begin{equation}\label{pmt0.14}
    \omega_k(t,y)=-\frac{1}{2\pi }e^{\alpha |k|t} e^{-\nu k^2 t}\int_\R e^{-ik\lambda t} \omega_{k,\nu}(y,\Lambda)\,d\lambda.
\end{equation}

To prove Theorem \ref{thm}, it suffices to show that for $k\in\Z\cap[1,\infty)$ and $t\ge1$, 
\begin{equation}\label{pmt0.15}
    |\psi_k(t,y)|\lesssim \frac{|y-y_{1\ast}|^{\gamma-2}+|y-y_{2\ast}|^{\gamma-2}}{\langle t\rangle^2}e^{-\sigma_\sharp\nu^{1/2}t}|k|^{-7/4}M_k,
\end{equation}
\begin{equation}\label{pmt0.16}
    |\partial_y\psi_k(t,y)|\lesssim e^{-\sigma_\sharp\nu^{1/2}t}\langle t\rangle^{-1}|k|^{-3/4}M_k,
\end{equation}
and
\begin{equation}\label{pmt0.17}
\big|\omega_k(t,y)\big|\lesssim \Big[\big(|y-y_{j\ast}|+|\nu|^{1/4}\big)^{\gamma}+t^{-\gamma/2}\Big]e^{-\sigma_\sharp\nu^{1/2}t}|k|^{-1/2-(2-\gamma)/10}M_k.
\end{equation}
We shall give the detailed proofs for \eqref{pmt0.15} and \eqref{pmt0.17} only, as the proof of \eqref{pmt0.16} is identical to that of \eqref{pmt0.13}. We only remark that $\partial_\lambda\partial_y\psi_{k,\nu}^\ast(y,\Lambda)$ is {\it less singular} than $\partial_\lambda^2\psi_{k,\nu}^\ast(y,\Lambda)$ in the region $\lambda\in\Sigma_{j,\delta_0}, j\in\{1,2\}$. Roughly speaking, when $\lambda\in\Sigma_{j,\delta_0}$, taking one derivative in $y$ costs a factor $\varrho_{j,k}^{-1}$ while taking a derivative in $\lambda$ costs a larger factor of $\varrho_{j,k}^{-1}\varrho_j^{-1}$.


Since the desired bounds \eqref{pmt0.15} and \eqref{pmt0.17} are pointwise, in view of the nature of the singularities of the integrand in \eqref{pmt0.13} and \eqref{pmt0.14}, it is convenient to divide the integrals \eqref{pmt0.13} and \eqref{pmt0.14} into the sum of a local integral and a nonlocal integral, as follows. For $y\in\T_\tp$, we define 
\begin{equation}\label{pmt0.18}\begin{split}
     &\psi_{k{\rm loc}}(t,y)\\
     &:=\frac{1}{2\pi k^2t^2}e^{\alpha |k|t} e^{-\nu k^2 t}\int_{\{\lambda \in\R:\,|\lambda-b(y)|\ll \,\min_{j\in\{1,2\}}|b(y)-b(y_{j\ast})|+|\nu/k|^{1/2}\}} e^{-ik\lambda t} \partial_\lambda^2\psi_{k,\nu}^\ast(y,\Lambda)\,d\lambda,
     \end{split}
\end{equation}
\begin{equation}\label{pmt0.19}
\begin{split}
     &\psi_{k{\rm nloc}}(t,y)\\
     &:=\frac{1}{2\pi k^2t^2}e^{\alpha |k|t} e^{-\nu k^2 t}\int_{\{\lambda \in\R:\,|\lambda-b(y)|\gtrsim \,\min_{j\in\{1,2\}}|b(y)-b(y_{j\ast})|+|\nu/k|^{1/2}\}} e^{-ik\lambda t} \partial_\lambda^2\psi_{k,\nu}^\ast(y,\Lambda)\,d\lambda.
     \end{split}
\end{equation}
Similarly, we define, for $y\in\T_\tp$, 
\begin{equation}\label{pmt0.195}
\begin{split}
        &\omega_{k{\rm loc}}(t,y)\\
        &:=-\frac{1}{2\pi }e^{\alpha |k|t} e^{-\nu k^2 t}\int_{\{\lambda \in\R:\,|\lambda-b(y)|\ll\, \min_{j\in\{1,2\}}|b(y)-b(y_{j\ast})|+|\nu/k|^{1/2}\}} e^{-ik\lambda t} \omega_{k,\nu}(y,\Lambda)\,d\lambda,
        \end{split}
\end{equation} 
\begin{equation}\label{pmt0.196}
\begin{split}
        &\omega_{k{\rm nloc}}(t,y)\\
        &:=-\frac{1}{2\pi }e^{\alpha |k|t} e^{-\nu k^2 t}\int_{\{\lambda \in\R:\,|\lambda-b(y)|\gtrsim\, \min_{j\in\{1,2\}}|b(y)-b(y_{j\ast})|+|\nu/k|^{1/2}\}} e^{-ik\lambda t} \omega_{k,\nu}(y,\Lambda)\,d\lambda.
        \end{split}
\end{equation}
To establish \eqref{pmt0.15} and \eqref{pmt0.17}, it suffices to prove that for $h\in\{\psi_{k{\rm loc}}(t,y),\psi_{k{\rm nloc}}(t,y)\}$ and $y\in\T_\tp, t\ge1$, we have the bounds 
\begin{equation}\label{pmt0.197}
    |h(t,y)|\lesssim \frac{|y-y_{1\ast}|^{\gamma-2}+|y-y_{2\ast}|^{\gamma-2}}{\langle t\rangle^2}e^{-\sigma_\sharp\nu^{1/2}t}|k|^{-7/4}M_k,
\end{equation}
and that $h\in\{\omega_{k{\rm loc}}(t,y),\omega_{k{\rm nloc}}(t,y)\}$ and $y\in\T_\tp, t\ge1$, we have the bounds for $j\in\{1,2\}$,
\begin{equation}\label{pmt0.198}
    \big|h(t,y)\big|\lesssim \Big[\big(|y-y_{j\ast}|+|\nu|^{1/4}\big)^\gamma+t^{-\gamma/2}\Big]e^{-\sigma_\sharp\nu^{1/2}t}|k|^{-1/2-(2-\gamma)/10}M_k.
\end{equation}
We divide the proof of \eqref{pmt0.197} and \eqref{pmt0.198} into several cases, depending on the location of $y\in\T_\tp$ relative to the critical points $y_{1\ast},y_{2\ast}$.

\subsection{The region $\min\{|y-y_{1\ast}|, |y-y_{2\ast}|\}\gtrsim 1$.}

 We firstly summarize the following bounds for $\psi^\ast_{k,\nu}(y,\Lambda)$, which follow from Proposition \ref{mainprop} and Corollary \ref{RDS1}. For $\min\{|y-y_{1\ast}|, |y-y_{2\ast}|\}\gtrsim 1$ and $\lambda\in \R\backslash \cup_{j=1}^2 \Sigma_{j,\delta_0}$ and $\sigma\in\{0,1,2\}$, we have the bounds 
\begin{equation}\label{pmt1}
    \|D^\sigma_\lambda\psi^\ast_{k,\nu}(\cdot,\Lambda)\|_{H^1_k(\T_\tp)}\lesssim |k|^{-5/2+\sigma}\langle\lambda\rangle^{-1-\sigma}M_k.
\end{equation}
Additionally, for $\lambda\in\Sigma_{j,\delta_0}$ with $j\in\{1,2\}$, $\beta\in\{0,1\}$ and $\min\{|y-y_{1\ast}|, |y-y_{2\ast}|\}\gtrsim 1$, we have the bounds 
\begin{equation}\label{pmt2}
\begin{array}{lr}
|k|^{-\beta}|\partial_y^\beta D_\lambda^\sigma \psi^\ast_{k,\nu}(y,\Lambda)|\lesssim |k|^{-\gamma_\ast-1/2+\sigma}M_k,&\quad{\rm if}\,\,\sigma\in\{0,1\},\\
|\partial_y^\beta D_\lambda^2 \psi^\ast_{k,\nu}(y,\Lambda)|\lesssim |k|^{-\gamma_\ast+3/2+\beta}\big(\delta(\lambda)\big)^{-1-2(2-\gamma_\ast)}M_k, &\quad{\rm if}\,\,\sigma=2.
\end{array}
\end{equation}

The bounds \eqref{pmt1}-\eqref{pmt2}, together with \eqref{RDS5}-\eqref{RDS6} and the identities \eqref{pmt0.12} are sufficient for proving the desired estimates bounds \eqref{pmt0.197} for $h=\psi_{k{\rm nloc}}$ and \eqref{pmt2} for $\omega_{k{\rm nloc}}$.

To prove the bounds \eqref{pmt0.197} for $h=\psi_{k{\rm loc}}$, we use \eqref{RDS6}, the identity \eqref{pmt0.12}, and with a slight abuse of notation denoting $\psi_{k,\nu}^\ast(y,\lambda)$ as $\psi_{k,\nu}^\ast(y,\Lambda)$ to emphasize the dependence on $\lambda$, to obtain that
\begin{equation}\label{pmt2.001}
   \psi^\ast_{k,\nu}(y,\lambda)-\psi^\ast_{k,\nu}(y,b(y)) =\int_{b(y)}^\lambda \Big[D_\tau \psi^\ast_{k,\nu}(y,\tau)-a(y,\tau)\partial_y\psi^\ast_{k,\nu}(y,\tau)\Big]\,d\tau
\end{equation}
and for $\lambda\in\R$ with $|\lambda-b(y)|\ll1$ that
\begin{equation}\label{pmt2.1}
     \begin{split}
   & \bigg|\omega_{k,\nu}(y,\Lambda)-\frac{\omega_{0k}(y)-ib''(y)\psi^\ast_{k,\nu}(y,b(y))}{i(\lambda-b(y))+|\nu/k|^{1/3}}\bigg|\\
   &\lesssim \frac{|\nu/k|^{-1/3}}{1+|\nu/k|^{-2/3}|\lambda-b(y)|^2}|k|^{-5/2}M_k+\frac{1}{|\lambda-b(y)|^{1/2}}|k|^{-3/2}M_k+|k|^{-1}M_k.
    \end{split}
 \end{equation}
The desired bounds then follow from \eqref{pmt2.1} and the identity
\begin{equation}\label{pmt2.2}
(\partial_y^2-k^2)\psi_{k,\nu}^\ast(y,\Lambda)=\omega^\ast_{k,\nu}(y,\Lambda).
\end{equation}
The bounds \eqref{pmt2} for $\omega_{k{\rm loc}}$ follow from \eqref{pmt2.1}.

\subsection{The region $|\nu/k|^{1/4}\ll\min\{|y-y_{1\ast}|, |y-y_{2\ast}|\}\ll 1$.}\label{sec:intM} For the sake of concreteness we assume that $|y-y_{1\ast}|\in[m/2,m]$ and $ |\nu/k|^{1/4}\ll m\ll 1$. It follows from proposition \ref{mainprop} that the following inequalities hold. For $\sigma\in\{0,1\}, \beta\in\{0,1\}$,

\begin{itemize}

\item  If $|\lambda-b(y_{1\ast})|^{1/2}\ll m$, then
\begin{equation}\label{pmt3.1}   
|\partial_y^\beta D^\sigma_\lambda \psi^\ast_{k,\nu}(y,\Lambda)| \lesssim |y-y_{1\ast}|^{-\sigma} \big(|y-y_{1\ast}|\wedge \frac{1}{|k|}\big)^{\gamma_\ast-\sigma-\beta}|k|^{-1/2}M_k; 
\end{equation}

\item If $\lambda\in \R\backslash\Sigma_{2,\delta_0}$ and $|\lambda-b(y_{1\ast})|^{1/2}\gtrsim m$, then
\begin{equation}\label{pmt3.3}
    \|\partial_y^\beta D^\sigma_\lambda \psi^\ast_{k,\nu}(y,\Lambda)\|_{L^2(m/4,4m)}
    \lesssim \delta(\lambda)^{1/2-\sigma} \big(\delta(\lambda)\wedge \frac{1}{|k|}\big)^{\gamma_\ast-\sigma-\beta}\frac{1}{\langle\lambda\rangle^{1+\sigma}}|k|^{-1/2}M_k;
\end{equation}    

\item If $\lambda\in \Sigma_{2,\delta_0}$, then
\begin{equation}\label{pmt3.35}
   |\partial_y^\beta D^\sigma_\lambda \psi^\ast_{k,\nu}(y,\Lambda)|\lesssim |k|^{-\gamma_\ast-1/2+\sigma+\beta}M_k.     
\end{equation}

\end{itemize} 

For $\sigma=2, \beta\in\{0,1\}$, we have the following bounds. 

\begin{itemize}

\item If $|\lambda-b(y_{1\ast})|^{1/2}\ll m$, then
    \begin{equation}\label{pmt3.4}
    \begin{split}
    &|\partial_y^\beta D^2_\lambda \psi^\ast_{k,\nu}(y,\Lambda)|\\
    &\lesssim (\delta(\lambda))^{-1-2(2-\gamma_\ast)}|y-y_{1\ast}|^{-1+2(2-\gamma_\ast)} \big(|y-y_{1\ast}|\wedge \frac{1}{|k|}\big)^{\gamma_\ast-2-\beta}|k|^{-1/2}M_k;
    \end{split}
    \end{equation}

\item If $\lambda\in \R\backslash\Sigma_{2,\delta_0}, |\lambda-b(y_{1\ast})|^{1/2}\gtrsim m$, then
   \begin{equation}\label{pmt3.6}
    \|\partial_y^\beta D^2_\lambda \psi^\ast_{k,\nu}(y,\Lambda)\|_{L^2(m/4,4m)}
    \lesssim \delta(\lambda)^{-3/2} \big(\delta(\lambda)\wedge \frac{1}{|k|}\big)^{\gamma_\ast-2-\beta}\frac{1}{\langle\lambda\rangle^{3}}|k|^{-1/2}M_k;   
    \end{equation}

\item If $\lambda\in \Sigma_{2,\delta_0}$, then
    \begin{equation}\label{pmt3.7}
   |\partial_y^\beta D^2_\lambda \psi^\ast_{k,\nu}(y,\Lambda)|\lesssim (\delta(\lambda))^{-1-2(2-\gamma_\ast)}|k|^{-\gamma_\ast+3/2+\beta}M_k.            
   \end{equation}

\end{itemize}

The desired bounds \eqref{pmt0.197} for $h=\psi_{k{\rm nloc}}(t,y)$ follow from the bounds \eqref{pmt3.1}-\eqref{pmt3.7}, by straightforward calculations, using the formula \eqref{pmt0.19}. 

To prove \eqref{pmt0.197} for $h=\psi_{k{\rm loc}}(t,y)$ and \eqref{pmt0.198}, we use the identity \eqref{pmt2.2}, the bound \eqref{RDS4} which we rewrite as follows. Denoting $\psi^\ast_{k,\nu}(y,\Lambda)$ as $\psi^\ast_{k,\nu}(y,\lambda)$ with a slight abuse of notation to emphasize the dependence on $\lambda$, for $\lambda\in\Sigma_{1,\delta_0}$ with $|\lambda-b(y)|\ll\, |b(y)-b(y_{1\ast})|$, we have the bounds
\begin{equation}\label{pmt3.701}
\begin{split}
   & \bigg|\omega_{k,\nu}(y,\Lambda)+\frac{ib''(y)\psi^\ast_{k,\nu}(y,b(y))}{i(\lambda-b(y))+|\nu/k|^{1/3}|y-y_{1\ast}|^{2/3}}\bigg|\\
   &\lesssim \varrho_{1,k}^{\gamma_\ast-1/2}(y ; \Lambda)|y-y_{1\ast}|^{-1/2}\frac{|\nu/k|^{-1/3}|y-y_{1\ast}|^{1/3}}{1+(|\nu/k|^{1/3}|y-y_{1\ast}|^{2/3})^{-2}|\lambda-b(y)|^2}|k|^{-1/2}M_k\\
   &\quad+\varrho_{1,k}^{\gamma_\ast-3/2}(y ; \Lambda)\varrho_1^{1/2}(y ; \Lambda)|\lambda-b(y)|^{-1/2}|k|^{-1/2}M_k.
    \end{split}
\end{equation}

The bounds \eqref{pmt0.198} for $h=\omega_{k{\rm loc}}(t,y)$ follow from \eqref{pmt3.701}. 

To prove \eqref{pmt0.198} with $h=\omega_{k{\rm nloc}}(t,y)$, we use also the following bounds. In view of \eqref{refb1} and Proposition \ref{mainprop}, we have the pointwise estimates for $\lambda\in \Sigma_{1,\delta_0}$ with $|\lambda-b(y)|\gtrsim |b(y)-b(y_{1\ast})|\}$,
\begin{equation}\label{pmt3.702}
\begin{split}
    \big|\omega_{k,\nu}(y,\Lambda)\big|&=\Big|\omega_{k,\nu}^\ast(y,\Lambda)+(\partial_y^2-k^2)\Big\{\sum_{j=1}^2\varphi\big((y-y_{j\ast})/\delta_0\big)\frac{\omega_{0k}(y)}{ib''(y)}\Big\}\Big|\\
    &\lesssim \varrho_{j,k}^{\gamma_\ast}(y ; \Lambda)\varrho_j^{-2}(y ; \Lambda)|k|^{-1/2}M_k,
    \end{split}
\end{equation}
\begin{equation}\label{pmt3.72}
    |\partial_\lambda \omega_{k,\nu}^\ast(y,\Lambda)|\lesssim \varrho_{j,k}^{\gamma_\ast-3/2}(y ; \Lambda)\varrho_j^{-5/2}(y ; \Lambda)|k|^{-1/2}M_k.
\end{equation}
Using the identity 
\begin{equation}\label{pmt3.721}
    \partial_\lambda \omega_{k,\nu}^\ast(y,\Lambda)=\partial_\lambda \omega_{k,\nu}(y,\Lambda)
\end{equation}
and equation \eqref{intP5} we get that
for $\lambda\in\R\backslash\big(\cup_{j = 1}^2 \Sigma_{j, \delta_0} \big)$,
\begin{equation}\label{pmt3.72s}
     |\partial_\lambda \omega_{k,\nu}^\ast(y,\Lambda)|\lesssim \frac{1}{\langle\lambda\rangle^{2}}|k|^{-3/2}M_k.
\end{equation}
For $\lambda\in\Sigma_{2,\delta_0}$, using  \eqref{refb1} and Proposition \ref{mainprop} we get that
\begin{equation}\label{pmt3.72s1}
     |\partial_\lambda \omega_{k,\nu}^\ast(y,\Lambda)|\lesssim |k|^{-1/2}M_k.
\end{equation}
Fix $\tau\in(1,\infty)$ as a parameter to be chosen below. We divide the integral \eqref{pmt0.196} into three regions for $\lambda\in\R$: (I) $\lambda\in\Sigma_{1,\delta_0}$ with $|b(y)-b(y_{1\ast})|\lesssim|\lambda-b(y)|<\tau^2|b(y)-b(y_{1\ast})|$; (II) $\lambda\in\Sigma_{1,\delta_0}$ with $|\lambda-b(y)|\ge \tau^2|b(y)-b(y_{1\ast})|$; (III) $\lambda\in\R\backslash\big[\Sigma_{1,\delta_0}\cup\Sigma_{2,\delta_0}\big]$; (IV) $\lambda\in\Sigma_{2,\delta_0}$. 

Using the bounds \eqref{pmt3.702}-\eqref{pmt3.72s1}, and integrating by parts in $\lambda$ in the regions (II)-(IV), we obtain the bounds for $t\ge1$,
\begin{equation}\label{pmt3.8}
\begin{split}
    &|\omega_{k{\rm nloc}}(t,y)|\lesssim\bigg[|y-y_{1\ast}|^\gamma +\tau^\gamma|y-y_{1\ast}|^\gamma+\frac{1}{|k|t}\tau^{\gamma-2}|y-y_{1\ast}|^{\gamma-2} \bigg]|k|^{-1/2-(2-\gamma)/10}M_k.
\end{split}
\end{equation}
Optimizing in $\tau\in(1,\infty)$, the desired bounds for \eqref{pmt0.198} with $h=\omega_{k{\rm nloc}}(t,y)$ follows.

\subsection{The region $\min\{|y-y_{1\ast}|, |y-y_{2\ast}|\}\lesssim |\nu/k|^{1/4}$.} For the sake of concreteness, we assume that $|y-y_{1\ast}|\lesssim|\nu/k|^{1/4}$. In this case we shall use the following bounds, which are consequences of Proposition \ref{mainprop}. 
\begin{itemize}
    \item If $\lambda\in\Sigma_{1,\delta_0}$, $\sigma\in\{0,1,2\}, \beta\in\{0,1\}$ and $\lambda\in\Sigma_{1,\delta_0}$, 
    \begin{equation}\label{pmt5.1}
    \begin{split}
       & \|\partial_y^\beta D^\sigma_\lambda \psi^\ast_{k,\nu}(y,\Lambda)\|_{L^2(|y-y_{1\ast}|\leq 3\delta(\Lambda)}\lesssim (\delta(\Lambda))^{1/2-\sigma}\Big[\delta(\Lambda)\wedge \frac{1}{|k|}\Big]^{\gamma_\ast-\sigma-\beta}|k|^{-1/2}M_k;
        \end{split}
    \end{equation}

    \item If $\lambda\in\R\backslash\cup_{j\in\{1,2\}}\Sigma_{j,\delta_0}$, $\sigma\in\{0,1,2\}$ and $\beta\in\{0,1\}$, 
    \begin{equation}\label{pmt5.2}
        |\partial_y^\beta D^\sigma_\lambda \psi^\ast_{k,\nu}(y,\Lambda)|\lesssim \langle\lambda\rangle^{-1-\sigma}|k|^{-5/2+\sigma+\beta}M_k;
    \end{equation}

     \item If $\lambda\in\Sigma_{2,\delta_0}$, $\sigma\in\{0,1\}$ and $\beta\in\{0,1\}$,
     \begin{equation}\label{pmt5.3}
   |\partial_y^\beta D^\sigma_\lambda \psi^\ast_{k,\nu}(y,\Lambda)|\lesssim  |k|^{-\gamma_\ast+\sigma+\beta}|k|^{-1/2}M_k;
     \end{equation}

     \item If $\lambda\in\Sigma_{2,\delta_0}$, $\sigma=2$ and $\beta\in\{0,1\}$,
     \begin{equation}\label{pmt5.35}
   |\partial_y^\beta D^2_\lambda \psi^\ast_{k,\nu}(y,\Lambda)|\lesssim \big(\delta(\Lambda)\big)^{-1-2(2-\gamma_\ast)} |k|^{-\gamma_\ast+2+\beta}.
     \end{equation}
    
\end{itemize}
This case is somewhat simpler in fact, since, roughly speaking, we can use the ellipticity of the Orr-Sommerfeld equation to bound derivatives without much loss, using Lemma \ref{ake1}. Indeed, the desired bounds \eqref{pmt0.197}-\eqref{pmt0.198} follow from an analogous argument as in section \ref{sec:intM}, using \eqref{pmt5.1}-\eqref{pmt5.35}, but somewhat simpler since there is no need to separately treat the region $|\lambda-b(y)|\ll|\nu/k|^{1/2}$ which is no longer singular. We omit the details. 
Theorem \ref{thm} is now proved. \\

\section{Proof of the limiting absorption principle (I): non-degenerate region}\label{sec:nondeg2}
In this section we give the proof of the limiting absorption principle in the non-degenerate region when $\lambda\in\R\backslash\big(\Sigma_{1,\delta_0}\cup \Sigma_{2,\delta_0}\big)$ or $\alpha\ge1$. We begin with some technical bounds on various operators that are useful below, and then prove the limiting absorption principle using a compactness argument. 

Denote $\epsilon:=\nu/k$. Recall that $\epsilon \in (0, 1/8]$, $k \in \Z \cap [1, \infty)$, and 
$$(\alpha, \lambda) \in \Big[ (-\sigma_0\epsilon^{1/2}, \infty) \times  \big(\R \setminus (\Sigma_{1, \delta_0}   \cup \Sigma_{2, \delta_0}) \big)\Big]  \cup \Big[ (1, \infty) \times \big(\Sigma_{1, \delta_0} \cup \Sigma_{2, \delta_0} \big)\Big] .$$ 
We also recall the definition \eqref{lapj1} for $T_\Theta$ and set for simplicity of notations $\Delta_k:=\partial_y^2-k^2$.

\subsection{Preliminary bounds}
We define two regions in the  $(\alpha, \lambda )$ plane
\begin{align}
    R_{nd} &:=  \big\{ (\alpha, \lambda) \, | \, \alpha\ge-\sigma_0|\epsilon|^{1/2},  \lambda \in  \R \setminus (\Sigma_{1, \delta_0}  \cup \Sigma_{2, \delta_0}) \big\}, \\
    R_{\alpha} &:= \big\{ (\alpha, \lambda) \, | \, \alpha \geq 1,   \lambda \in \Sigma_{1, \delta_0} \cup \Sigma_{2, \delta_0}     \big\}. 
\end{align}

In $R_{nd}$, which we call the ``non-degenerate region", $\lambda$ is separated from the critical points; in $R_{\alpha}$, which we call the ``$\alpha$ dominated region", $\alpha\geq 1$. Therefore, there is no singularity around the critical points $y_{1\ast}, y_{2\ast}$ in either case when we apply the operator $A_\Theta^{-1}$. 

Throughout the section we assume that $(\alpha,\lambda)\in R_{nd}\cup R_{\alpha}$.
We begin with the following pointwise bounds. 
\begin{lemma}
\label{lem:T:regular}
Suppose that $$m:=\inf_{y\in\T_\tp}(|b(y) - \lambda| + |\alpha| + \epsilon^{\frac{1}{3}}) \gtrsim1.$$
Then for any bounded function $f$ on $\T_\tp$,
\begin{align}
    \sup_{y\in\T_\tp}\Big|\frac{1}{k^\beta}\p_y^{\beta}  T_\Theta f(y) \Big| \lesssim \frac{\norm{f}_{L^{\infty}(\T_\tp)}}{k^2m}, \quad{\rm with}\,\, \beta \in \{0,1 \}.
\end{align}
\end{lemma}
\begin{proof}
   The proof is a consequence \eqref{Helm1.0},  Proposition \ref{Airy_main} in the non-degenerate regime, and the fact that $k e^{-k|y|}$ is an approximation of the identity.   
\end{proof}
The bounds from Lemma \ref{lem:T:regular} are sufficient for our purposes when $(\alpha, \lambda) \in R_{\alpha}$. 
For $y \in \T_\tp$  such that $$|b(y) - \lambda| + |\alpha| + \epsilon^{\frac{1}{3}} \ll 1,$$ we need to carry out a finer analysis of the behavior near the critical layer.
To begin with, we decompose $T_\Theta$ as follows. Let $\varphi_0\in C_c^\infty(-1,1)$ with $\varphi_0\equiv 1$ on $[-1/2,1/2]$.
Then for any number $L>0$, define for $y\in\R$,
\begin{equation}
\label{cutoff:scaled}
    \varphi_L(y) := \varphi_0 \left( \frac{y}{L} \right).
\end{equation}
We now decompose for any $f\in H^1(\T_\tp)$,
\begin{align}
\label{Tnd:decomp}
    T_{\Theta }f &= \sum_{\ell=0}^2 T_{\Theta, \ell} f 
    \end{align}
where 
\begin{align}
    T_{\Theta,  0} f(y) := \int_{\T_\tp} G_k(y,z) \ai(f)(z) \Big[1  - \sum_{\ell=1}^2 \varphi_{\theta }(z - y_\ell)  \Big] \, dz  
\end{align}
and
 \begin{align}
         T_{\Theta,  \ell }f(y) := \int_{\T_\tp} G_k(y,z)  \varphi_{\theta }(z - y_\ell) \ai(f)(z) \, dz 
    \end{align}
where $b(y_\ell) = \lambda, \ell \in \{ 1, 2 \}$ and $\theta = \theta(\delta_0) > 0$ is a sufficiently small constant. 
We note that for $\lambda$ in the interior of $b(\T_\tp)$, the set $b^{-1}(\lambda)$ has cardinality 2. 

Since the argument of $T_{\Theta, 0}$ is supported away from the critical layer we have the following pointwise bounds.
\begin{lemma}
\label{lem:to}
  There exists $\theta_0 \in (0,1)$, depending only on $b(y)$, such that for $\theta \in (0, \theta_0]$ and any $f \in L^{\infty}(\T_\tp)$ we have the following bounds:
    \begin{equation}
        \frac{1}{k^\beta}\big|\p_y^\beta T_{\Theta, 0} f(y)\big| \lesssim \frac{\norm{f}_{L^{\infty}(\T_\tp )}}{k^2},  \quad \beta \in \{0, 1, 2 \}.
    \end{equation}
\end{lemma}
\begin{proof}
Proposition \ref{Airy_main} implies that 
\begin{equation}
   \Big| \ai(f)(z) \big(1  - \sum_{\ell=1}^2 \varphi_{\theta }^2(z - y_\ell) \big)\Big| \lesssim_{\theta_0} \norm{f}_{L^{\infty}(\T_\tp)}.
\end{equation}
The proof is then an immediate consequence of Lemma \ref{mGk50} and the fact that $k e^{-k|y|}$ is an approximation of identity.   
\end{proof}

We now turn our attention to $T_{\Theta, \ell }, \ell \in \{ 1,2 \}$. The following result is useful in characterizing the singular behavior near the critical layer. 

\begin{lemma}[]
\label{lem:Airybounds}
 There exists $\theta_0 \in (0,1)$,  such that for $\theta \in (0, \theta_0], \,\ell\in\{1,2\}$, the following statement holds.
Suppose that $f\in H^1_k(\T_\tp)$ with
\begin{equation}
\label{eq:f:H1k:regularity}
  M:= \norm{f}_{H_k^1(\T_\tp)}. 
\end{equation}
We view the function $\varphi_{ \theta }( y - y_\ell) \ai( f)(y)$
 as a function on $\R$.
Then
\begin{equation}
\label{eq:Airy:freq:pw:non}
\sup_{\xi \in \R}\big| \mathcal{F}\left(\varphi_{ \theta }( \cdot - y_\ell) \ai( f) \right)(\xi)\big|  \lesssim  M,
\end{equation}
where $\mathcal{F}(h)$ denotes the Fourier transform of $h$ and $y_\ell$ satisfies $b(y_\ell) = \lambda$.
\end{lemma}

\begin{proof}
 With a slight abuse of notation, denote $\varphi_\ell(y) := \varphi\big( \frac{2(y- y_\ell)}{\theta } \big)$, for $\ell \in \{ 1,2\}, y\in\R$. Define also $w_{c,\ell}$ as the solution to 
\begin{equation}\label{Haoadd1}
    \epsilon \p_y^2 w_{c,\ell}(y) -\alpha w_{c,\ell}(y)   + ib'(y_{\ell})(y_{\ell} - y)w_{c , \ell}(y)  = f(y_{\ell}), \quad y \in \R.
\end{equation}
We can solve $w_{c,\ell}$ explicitly as
\begin{equation}
\label{eq:wc:formula}
    w_{c,\ell}(y,y_\ell) = \frac{\ f(y_\ell)  }{b'(y_\ell)} \left(\frac{b'(y_\ell)}{\epsilon} \right)^{\frac{1}{3}}
    W_\ell \left(\left( \frac{ b'(y_\ell)}{\epsilon} \right)^{\frac{1}{3}}   (y-y_\ell )\right),
\end{equation}
where 
\begin{equation}
    \widehat{W}_\ell(\xi) = -\sqrt{2\pi} \exp \left( \frac{\xi^3}{3} + \frac{\ept \alpha \xi }{(b'(y_\ell))^{\frac{2}{3}}}  \right) \mathbbm{1}_{(-\infty, 0]}(\xi),  \quad \xi \in \R.
\end{equation}
Define $w_R: \T_\tp \to \C$ as 
\begin{align}
\label{eq:Airy:decomp}
    w_R := \ai(f) - \sum_{\ell \in \{ 1,2\} }\varphi_\ell w_{c,\ell}, 
\end{align}
where we view $\varphi_{\ell} w_{c,\ell}$ as functions on $\T_\tp$. 
It follows from \eqref{Haoadd1} and \eqref{eq:Airy:decomp} that $w_R$ satisfies the equation for $y\in\T_\tp$,
\begin{align}
    &\epsilon\p_y^2 w_R(y) - \alpha w_R(y) + i(\lambda -b(y))w_R(y)\notag\\
    &= f(y) -  \sum_{\ell \in \{1,2 \}} \varphi_{\ell}(y)f(y_\ell) - \epsilon\sum_{\ell \in \{1,2 \}} \Big[ \p_y \varphi_{\ell}(y) \p_y w_{c,\ell}(y) + \p_y^2\varphi_{\ell}(y) w_{c,\ell}(y) \Big]      \notag\\
    &\quad+i\sum_{\ell \in \{1,2 \}} \varphi_{\ell}(y)w_{c,\ell}(y)\Big[ b(y_\ell) + b'(y_\ell)( y - y_\ell) 
  - b(y) 
    \Big].
\end{align}
In view of Proposition \ref{Airy_main}, we conclude that for
$y\in\T_\tp$,
\begin{equation}
\label{eq:wR:bdd}
    |w_R(y)| \lesssim   \frac{M\epsilon^{-\frac{1}{6}} }{\lb \ept|\alpha|,  \ept (b(y)-\lambda) \rb^{\frac{1}{2}} }.
\end{equation}
Therefore,
\begin{equation}
   \big|\mathcal{F}(\varphi_{ \theta } (\cdot - y_\ell) \ai(f))  (\xi)\big|\lesssim  M,  \quad \ell \in \{ 1, 2\},
\end{equation}
which completes the proof.
\end{proof}
\begin{remark}
    Note that $\varphi_{ \theta } (y - y_\ell) \ai(f)(y)$ can also be viewed as a periodic function on $\T_\tp$. Then the Fourier coefficients $\big[\varphi_\theta(\cdot - y_\ell) \ai(f)\big]_l$ are bounded up to an implicit constant by $M$, 
    \begin{equation}
        \sup_{l \in  \frac{2\pi\mathbb{Z}}{\tp}} \Big|\big[\varphi_\theta(\cdot - y_\ell) \ai(f)\big]_l\Big| \lesssim M.
    \end{equation}
\end{remark}
Recalling the definition of $T_{ \Theta, \ell }f, \ell \in \{1 , 2 \} $ as $\Delta_k^{-1}\varphi_{\theta }(\cdot - y_\ell) \ai(i b''f)$ and the Fourier multiplier representation of $\Delta_k^{-1}$, we have as a consequence of Lemma \ref{lem:Airybounds} the following bounds for Fourier coefficients of $T_{\Theta, \ell }f$:
\begin{corollary}
\label{coro:t:freq}
Let $\norm{f}_{H_k^1} = M$. Then we have  
    \begin{equation}
    \big|[T_{ \Theta, \ell } f]_l\big| \lesssim \frac{M}{l^2 + k^2}, \quad k \in \Z ,\, l \in  \frac{2\pi\mathbb{Z}}{\tp}.
\end{equation}
\end{corollary}
Finally, combining  Lemma \ref{lem:T:regular}, Lemma \ref{lem:to} and Corollary \ref{coro:t:freq}, we can conclude the following bounds for $T_\Theta$  on $H_k^1(\T_\tp)$.
\begin{lemma}
\label{lem:T:bdd}
 $T_\Theta: H_k^1(\T_\tp) \to H_k^1(\T_\tp)$ satisfies the bounds
    \begin{equation}
\label{eq:T:k:bdd}
\norm{T_\Theta}_{ H_k^1(\T_\tp)\to H^1_k(\T_\tp)} \lesssim k^{-1/2}.
\end{equation}

 Moreover, the family of operators $T_\Theta$ has the following ``uniform compactness" property. More precisely, assume that $ \ell, k_\ell \in  \Z \cap [1, \infty) $ and that $f_\ell$ is a sequence of functions satisfying $\norm{f_\ell}_{H_{k_\ell}^1(\T_\tp)} = 1$. Let  $$\Theta_\ell:=(\epsilon_\ell, \alpha_\ell, k_\ell, \lambda_\ell) \in \big(0, 1/8 \big)\times (-\sigma_0 \epsilon_\ell^{1/2}, 1) \times \big (\Z \cap [1, \infty)  \big ) \times (\Sigma_0\setminus \cup_{j \in \{ 1,2\}   }   \Sigma_{j, \delta_0}).$$ Then for any $k_0\in\Z\cap[1,\infty)$, the sequence  
\begin{equation}
    T_{\Theta_\ell} (f_\ell) 
\end{equation}
has a convergent subsequnece in $H_{k_0}^1(\T_\tp)$.
\end{lemma}
For applications below, we need to study the singular behavior of $\ai(f)$ as $\epsilon \to 0$, which is useful to connect the spectral assumption on the invisicd Rayleigh operator to the viscous Orr-Sommerfed operator.
\begin{lemma}
\label{lem:limiting:behavior}
Let $f \in H_k^1(\T_\tp)$. Then $\ai(f)$ has the following limiting behavior for $y\in\T_\tp$ in the sense of distributions.
\begin{itemize}
    \item If $\alpha_0 \neq 0$ or $\lambda_0 \notin b(\T_\tp)$,
    \begin{equation}
    \label{eq:alpha0}
         \lim_{\epsilon \to 0, \, \alpha \to \alpha_0, \, \lambda \to \lambda_0 } \ai(f)(y) = \frac{ -f(y)}{ \alpha_0 +   i(b(y) - \lambda_0)};
    \end{equation}

    \item If $\lambda_0 \in b(\T_\tp) \setminus \cup_{j \in \{1 ,2 \}}  \Sigma_{j, \delta_0}$,
    \begin{equation}
    \label{eq:alpha+}
    \lim_{\epsilon \to 0,\,  \alpha \to  0, \, \lambda \to \lambda_0} \ai(f)(y) = {\rm P.V.} \frac{ -f(y)}{ i(b(y) - \lambda_0)} -  \pi \sum_{b(z) = \lambda_0} \frac{f(z)}{|b'(z)|} \delta(y-z).
\end{equation}

\end{itemize}

\end{lemma}

\begin{proof}
 We only give the detailed proof of \eqref{eq:alpha+}, since the other case is easier, and indicate the necessary changes for the proof of \eqref{eq:alpha0} at the end of the proof. Fix an even cutoff function $\Phi \in C_0^{\infty}(-2,2)$ with $\Phi \equiv 1$ on $[-1,1]$. Let $\phi \in C_0^{\infty}(\T_\tp)$ be a test function. We have
    \begin{align}
    \label{lim1}
        &\lim_{\epsilon \to 0, \, \alpha \to 0, \, \lambda \to \lambda_0} \int_{\T_\tp} \ai(f)(y) \phi(y) \, dy \notag\\
        &=  \lim_{\epsilon \to 0, \, \alpha \to 0, \, \lambda \to \lambda_0}  \int_{\T_\tp} \sum_{b(z) = \lambda} \Phi \left(\frac{y-z}{(|\alpha| + \epsilon)^{\frac{1}{4}}} \right) \ai(f)(y) \phi(y) \, dy  \notag\\
        &\quad+ \int_{\T_\tp} \sum_{b(z) = \lambda} \left(1 - \Phi \left(\frac{y-z}{(|\alpha| + \epsilon)^{\frac{1}{4}}} \right) \right) \ai(f)(y) \phi(y) \, dy .
    \end{align}
For $y$ in the support of $1 - \Phi$, we divide \eqref{eqAt} by $i(b(y) -\lambda) $  to get that
\begin{equation}
\label{lim2}
    \ai(f)(y) = -\frac{f(y)}{i(b(y) - \lambda)} - \frac{\alpha \ai(f)(y)}{i(b(y) - \lambda)} + \frac{\epsilon\p_y^2 \ai(f)(y) }{i(b(y) - \lambda)}.
\end{equation}
Proposition \ref{Airy_main} then implies that 
\begin{align}
\label{lim3}
   \lim_{\epsilon \to 0,\, \alpha \to 0, \,  \lambda \to \lambda_0} \int_{\T_\tp} \sum_{b(z) = \lambda}\left(1 - \Phi \left(\frac{y-z}{(|\alpha| + \epsilon)^{\frac{1}{4}}} \right) \right) \ai(f)(y) \phi(y) \, dy  = -\text{P.V.} \int_{\T_\tp} \frac{f(y) \phi(y)}{i(b(y) - \lambda_0)} \, dy.
\end{align}
On the other hand, recalling the decomposition  \eqref{eq:Airy:decomp} in Lemma \ref{lem:Airybounds}, we obtain that 
\begin{align}
\label{lim4}
    & \int_{\T_\tp} \sum_{b(z) = \lambda}\Phi \left(\frac{y-z}{(|\alpha| + \epsilon)^{\frac{1}{4}}} \right)  
   \ai(f)(y)\phi(y) \, dy   \notag \\
    &=  \int_{\T_\tp} \sum_{\ell \in \{1, 2\}} \Phi \left(\frac{y-z_\ell}{(|\alpha| + \epsilon)^{\frac{1}{4}}} \right) w_{c,\ell}(y) \phi(y) \, dy  + \int_{\T_\tp}  \sum_{\ell \in \{1, 2\}}\Phi \left(\frac{y-z_\ell}{(|\alpha| + \epsilon)^{\frac{1}{4}}} \right) w_R(y) \phi(y) \, dy .
\end{align}
From \eqref{eq:wR:bdd}, it follows that 
\begin{equation}
\label{lim5}
     \lim_{\epsilon \to 0,\, \alpha \to 0, \,  \lambda \to \lambda_0} \int_{\T_\tp}  \Phi \left(\frac{y-z_\ell}{(|\alpha| + \epsilon)^{\frac{1}{4}}} \right) w_R(y) \phi(y) \, dy  = 0.
\end{equation}
We now focus our attention on $w_{c, 1}$ with the treatment of $w_{c,2}$ being identical. Since $w_{c,1}$ is defined on $\R$, we have
\begin{align}
\label{lim6}
    &\int_{\mathbb{R}} \Phi \left(\frac{y-z_1}{(|\alpha| + \epsilon)^{\frac{1}{4}}} \right) w_{c,1}(y) \phi(y) \, dy \notag\\
    &=  \phi(z_{1})\int_{\mathbb{R}} \Phi \left(\frac{y-z_1}{(|\alpha| + \epsilon)^{\frac{1}{4}}} \right) w_{c,1}(y)  \, dy + \int_{\mathbb{R}} \Phi \left(\frac{y-z_1}{(|\alpha| + \epsilon)^{\frac{1}{4}}} \right) w_{c,1}(y)(\phi(y) -\phi(z_{1}) )  \, dy. 
\end{align}
Since $ w_{c,1}(y)(\phi(y) -\phi(z_{1}) )$ is bounded, the second term converges to 0 as $\alpha, \epsilon \to 0$. Therefore, it suffices to consider the first term. In view of \eqref{eq:wc:formula}, we get that
\begin{equation}
\label{lim7}
     \int_{\mathbb{R}} \Phi \left(\frac{y-z_1}{(|\alpha| + \epsilon)^{\frac{1}{4}}} \right) w_{c,1}(y)  \, dy = \int_{\mathbb{R}} \Phi \left(\frac{y-z_1}{(|\alpha| + \epsilon)^{\frac{1}{4}}} \right) \frac{\ept f(z_1)  }{(b'(z_1))^{\frac{2}{3}}} W(\ept b'(z_1)(y - z_1)) \, dy.
\end{equation}
By rescaling  and Parseval's identity, we obtain that 
\begin{align}
\label{lim7.5}
     &\lim_{\epsilon \to 0, \, \alpha \to  0, \,  \lambda \to \lambda_0}\int_{\mathbb{R}} \Phi \left(\frac{y-z_1}{(|\alpha| + \epsilon)^{\frac{1}{4}}} \right) \frac{\ept f(z_1)  }{(b'(z_1))^{\frac{2}{3}}} W(\ept b'(z_1)(y - z_1)) \, dy \notag\\
    & = - \sqrt{2 \pi} \varphi(z_1)\frac{f(z_1)}{|b'(z_1)|} \int_{-\infty}^0\widehat{\Phi}(\xi) d \xi  
     = -\pi \varphi(z_1) \frac{f(z_1)}{|b'(z_1)|}.
\end{align}
We now turn to the case where $\lambda_0 \notin b(\T_\tp)$ and $\alpha_0 = 0$. For $\lambda$ sufficiently close to $\lambda_0$, $b(y) - \lambda$ does not vanish on $\T_\tp$. Hence, assuming that $\lambda \notin b(\T_\tp)$, we can write that 
\begin{equation}
\label{lim8}
    \ai(f)(y) = -\frac{f(y)}{ \alpha + i(b(y) - \lambda)} + \frac{\epsilon\p_y^2 \ai(f)(y) }{ \alpha + i(b(y) - \lambda)}. 
\end{equation}
Proposition \ref{Airy_main} implies 
\begin{align}
\label{lim9}
   \lim_{\epsilon \to 0,\, \alpha \to 0, \,  \lambda \to \lambda_0} \int_{\T_\tp}   \ai(f)(y) \phi(y) \, dy  =  \int_{\T_\tp} \frac{f(y) \phi(y)}{i(b(y) - \lambda_0)} \, dy.
\end{align}
Finally, assume that $\alpha_0 \neq 0$ and $\lambda \in b(\T_\tp)$. In this case we again write $\ai(f)$ as in \eqref{lim8} by taking $\alpha$ sufficiently close to $\alpha_0$ so that $\alpha \neq 0$. It then follows from Proposition \ref{Airy_main} that
\begin{align}
\label{lim10}
   \lim_{\epsilon \to 0,\, \alpha \to \alpha_0, \,  \lambda \to \lambda_0} \int_{\T_\tp}   \ai(f)(y) \phi(y) \, dy  =  \int_{\T_\tp} \frac{f(y) \phi(y)}{ \alpha_0 + i(b(y) - \lambda_0)} \, dy,
\end{align}
which completes the proof of Lemma \ref{lem:limiting:behavior}. 
\end{proof}

\subsection{The compactness argument}\label{sec:nondeg}
We are now ready to prove the limiting absorption in the non-degenerate case. 
\begin{proposition}
\label{prop:lap:non}
There exist $\kappa  >0$ and $\epsilon_0 \in (0, 1/8)$ such that for all $\epsilon \in (0, \epsilon_0)$, $k \in \Z \cap [1, \infty)$, and $(\alpha ,\lambda) \in R_{nd} \cup R_{\alpha}$, 
we have the bounds 
\begin{equation}
\label{eq:lapnondeg}
 \norm{ \psi + T_{\Theta} (i b'' \psi )}_{H_k^1(\T_\tp)} \geq \kappa \norm{\psi}_{H_k^1(\T_\tp)},
\end{equation}
for any $\psi \in H_k^1(\T_\tp)$.
\end{proposition}

\begin{proof}
Suppose \eqref{eq:lapnondeg} does not hold. Then there exists a sequence $(\psi_\ell, \kappa_\ell, \Theta_\ell)$ for $\ell\in\Z\cap[1,\infty)$,
with $(\kappa_\ell, \epsilon_\ell) \to 0+$, $\norm{\psi_\ell }_{H_{k_{\ell} }^1(\T_\tp) } = 1$, such that 
    \begin{equation}
    \label{eq:lapnon0}
        \norm{\psi_\ell + T_{\Theta_\ell} \psi_\ell}_{H_{k_\ell}^1(\T_\tp) } \leq \kappa_{\ell}.
    \end{equation}
  We first assume that $(\alpha_\ell, \lambda_\ell) \in R_{nd}$. 
  By Lemma  \ref{lem:T:regular}  and Lemma \ref{lem:T:bdd} we see that $(k_\ell, \alpha_\ell, \lambda_\ell)$ are bounded. Therefore, passing to a  subsequence if necessary, we can assume that 
  $$ (\alpha_\ell, k_\ell , \lambda_\ell) \to ( \alpha_0, k_0, \lambda_0) \in  [0, \infty) \times  \Z \cap [1, \infty) \times \overline{\R \setminus (\cup_{j \in \{ 1,2\}}\Sigma_{j, \delta_0}} ) .$$  It also follows from Lemma \ref{lem:T:bdd} that $T_{\Theta_\ell}(i b''\psi_\ell)$ is relatively compact in $H_{k_0}^1(\T_\tp)$. Writing $\psi_\ell$ as  
    \begin{equation}
        \psi_\ell = \psi_\ell + T_{\Theta_\ell}(ib''\psi_\ell) -T_{\Theta_\ell}(ib'' \psi_\ell) 
    \end{equation}
    and passing to a subsequence if necessary, we can infer that $\psi_\ell$ converges \emph{strongly} in $H_{k_0}^1(\T_\tp)$ to $\Psi\in H_{k_0}^1(\T_\tp)$ with $\|\Psi\|_{H_{k_0}^1(\T_\tp)}=1$.
    
    As a consequence of \eqref{eq:lapnon0}, we conclude that $\Psi$ satisfies the equation that
    \begin{equation}
     \Psi + \lim_{\ell \to \infty} T_{\Theta_\ell}(ib'' \Psi) = 0.   
    \end{equation}
Since
\begin{equation}
    \lim_{\ell \to \infty} T_{\Theta_{\ell}} (ib'' \Psi) =  \lim_{\ell \to \infty}\Delta_{k_0}^{-1} A_{\Theta_\ell} (ib''\Psi),
\end{equation}
applying  $-\Delta_{k_0}$ yields
    \begin{equation}
        -\Delta_{k_0} \Psi - \lim_{\ell \to \infty} A_{\epsilon_\ell}^{-1} (ib''\Psi) = 0,
    \end{equation}
    in the sense of distributions.
    Assuming that $\alpha_0 = 0$, and using Lemma \ref{lem:limiting:behavior}, we obtain that for $y\in\T_\tp$,
    \begin{equation}
        -\Delta_{k_0} \Psi(y) + \text{P.V.}\frac{b''(y)\Psi(y)}{b(y) - \lambda_0} + i\pi \sum_{b(z) = \lambda_0} \frac{b''(z) \Psi(z)}{|b'(z)|} \delta(y - z) = 0,
    \end{equation}
    which contradicts the main spectral assumption \eqref{MaAs}. 
    
    In the case when $\alpha_0 \neq 0$, we have 
    \begin{equation}
    \label{not:singular}
        -\Delta_{k_0} \Psi(y) + \frac{b''(y)\Psi(y)}{b(y) - \lambda_0- i \alpha_0}  = 0,
    \end{equation}
   which would imply the existence of a discrete eigenvalue for the Rayleigh operator which also contradicts the main spectral Assumption \ref{MaAs}. 
   
   Finally, assume that  $(\alpha_\ell, \lambda_\ell) \in R_\alpha$ and therefore $(\alpha_0, \lambda_0) \in [1, \infty) \times \overline{ (\cup_{j \in \{ 1,2\}}\Sigma_{j, \delta_0} )} $. Since $\alpha_0 \neq 0$, we can use Lemma \ref{lem:limiting:behavior} to once again obtain \eqref{not:singular}, which leads to another contradiction. The  proof of Proposition \ref{prop:lap:non} is now complete.
\end{proof}

\section{Proof of the limiting absorption principle (II): the degenerate region}\label{sec:deg2}
In this section, we consider the case where $\lambda \in \cup_{j \in \{ 1,2\}}\Sigma_{j, \delta_0} $ and $\alpha$ is small. More precisely, 
we define a region $R_{dj}$ in the $(\alpha, \lambda )$ plane as 
\begin{equation}\label{defdjJ}
  \ R_{dj} : = \big\{ (\alpha, \lambda) \in \R ^2\, |\, \alpha \in (-\sigma_0\epsilon^{1/2}, 1) , \,  \lambda \in \Sigma_{j, \delta_0}  \big\}, \quad j \in \{1, 2\},
\end{equation}
where we recall  that $\sigma_0$ is chosen in Proposition \ref{Airy_main}. In contrast to the case when $(\alpha, \lambda) \in R_{nd} \cup R_{\alpha}$, $\ai $ is now ``quadratically" singular near the critical layers, $b(y) = \lambda$,  using the heuristic in \eqref{oop4}. 
Therefore  $\ai$ is of the same  order as $\Delta_k$ and we can no longer treat $T_{\Theta} = \Delta_k^{-1} \ai$ perturbatively. To address this issue, we need to incorporate part of the nonlocal term in the main terms and prove various weighted estimates, before we can carry out an analogous compactness argument as in the non-degenerate case to prove the limiting absorption argument. 
\subsection{Bounds on the operator $T_\Theta$.}
While the definition of $T_{\Theta}$ given in \eqref{lapj2} is concise, it obscures the  important cancellation between $A_{\Theta}^{-1} (i b'' \psi )$ and the  potential $b''(y) (b(y) - \lambda - i \alpha)^{-1}$. Therefore, we expand $A_{\Theta}^{-1} (i b'' \psi )$ further which leads to the following decomposition for $T_{\Theta} $.
\begin{lemma}
Assume that $\epsilon \in (0, 1/8)$, $(\alpha, \lambda) \in (-\sigma_0\epsilon^{1/2}, 1) \times \Sigma_{j , \delta_0}, \, j \in \{1 ,2  \}$, and $k \in \Z \cap [1, \infty)$. For  $f \in H^1(\T_\tp)$  we can decompose for  $y \in \T_\tp$, 
    \begin{equation}
        T_{\Theta}f(y) = -i\tc{I}{1}f(y) - i\tc{I}{2}f(y) - i\tc{v}{1}f(y) - \tc{v}{2}f(y),
    \end{equation}
    where 
    \begin{subequations}
         \begin{align}
         \label{eq:ti1}
        &\tc{I}{1}f(y) :=   \int_{\T_\tp} \mathcal{G}_k^j(y,z;\Lambda) \ai \left[ \Big(1- \varphi_0\Big(  \frac{\cdot- y_{j*}}{\delta_0}\Big)
 \Big)  b'' f   \right](z) \, dz ,  \\
 \label{eq:ti2}
    &\tc{I}{2}f(y) :=  \int_{\T_\tp} \mathcal{G}_k^j(y,z;\Lambda) \ai \Big( \varphi_0\Big( \frac{\cdot - y_{j*}}{\delta(\Lambda)} \Big) b'' f \Big)(z) \, dz,  \\
    \label{eq:tv1}
        &\tc{v}{1}f(y) \notag\\
        &:= \int_{\T_\tp} \mathcal{G}_k^j(y,z;\Lambda)  
  \varphi_0\left(  \frac{(z- y_{j*}) }{\delta(\Lambda)/3}\right)
   \ai\left[ \Big(\varphi_0\Big( \frac{\cdot - y_{j*}}{\delta_0} \Big) - \varphi_0 \Big(  \frac{\cdot - y_{j*}}{\delta(\Lambda)}\Big) \Big) b'' f \right](z)\, dz, \\
        \label{eq:tv2}
        &\tc{v}{2}f(y) \notag \\
        &:= \int_{\T_\tp} \mathcal{G}_k^j(y,z;\Lambda) \left(1 - \varphi_0\left(  \frac{(z- y_{j*}) }{\delta(\Lambda)/3}\right)   \right)\frac{\epsilon \p_z^2\ai\left[ \left(\varphi_0\left( \frac{\cdot - y_{j*}}{\delta_0} \right) - \varphi_0 \left(  \frac{\cdot - y_{j*}}{\delta(\Lambda)}\right) \right)  b'' f \right](z)}{b(z) - \lambda - i  \alpha }\, dz .
    \end{align}
    \end{subequations}
   \end{lemma}
where $\varphi_0$ is the same as in the modified Green's function, see \eqref{mGk1}. 
   
   \begin{remark}
      The motivation for the subscripts $I$ and $v$ are ``inviscid" and ``viscous" respectively. Both $T_{I1}$ and $T_{I2}$ have analogs that appear in the inviscid case (see \cite{Iyer}) but $T_{v1}$ and $T_{v2}$ only appear when viscosity is present.   
   \end{remark}
\begin{proof}
 For $y \in \T_\tp$ we write
\begin{align}
    \ai(ib'' f) = \ai\big[(1 - \eta_{j,\delta}) ib''g \big] + \ai\big[\eta_{j,\delta} ib'' f \big],
\end{align}
where for $y\in\T_\tp$,
\begin{equation}
    \eta_{j, \delta}(y) := \varphi\Big(\frac{y-y_{j\ast}}{\delta_0}\Big)-\varphi\Big(\frac{y-y_{j\ast}}{\delta(\Lambda)}\Big).
\end{equation}
It follows from \eqref{eqAt} that for $y$ away from the critical layer we have 
\begin{equation}
\label{decomp1}
   \ai(\eta_{j, \delta} ib''f )(y) = -\frac{ \eta_{j, \delta}(y) b''(y) f(y)  }{b(y) - \lambda - i \alpha} + \frac{\epsilon \p_y^2 \ai(\eta_{j, \delta} b'' f )(y)}{b(y) - \lambda - i\alpha },
\end{equation}
therefore, 
\begin{align}
\label{decomp2}
&\Big(1 - \varphi\Big(\frac{ y -y_{j*} }{\delta(\Lambda) /3 } \Big) \Big) \left[\ai \big( \eta_{j, \delta} ib''f \big)(z)
+\frac{b''(y)}{b(y)-\lambda -i\alpha}\Big[\varphi\Big(\frac{z-y_{j\ast}}{\delta_0}\Big)-\varphi\Big(\frac{z-y_{j\ast}}{\delta(\Lambda)}\Big)\Big]f(z) \right] \notag\\
&= \left(1 - \varphi\Big(\frac{ y -y_{j*} }{\delta(\Lambda)/3} \Big) \right)\frac{\epsilon \p_y^2 \ai(\eta_{j, \delta} b'' f )(y)}{b(y) - \lambda - i\alpha },
\end{align}
and 
\begin{align}
\label{decomp3}
& \varphi\Big(\frac{ y -y_{j*} }{\delta(\Lambda)/3} \Big)  \left[\ai \big( \eta_{j, \delta} ib''f \big)(y)
+\frac{b''(y)}{b(y)-\lambda -i\alpha}\Big[\varphi\left(\frac{y-y_{j\ast}}{\delta_0}\right)-\varphi\left(\frac{y-y_{j\ast}}{\delta(\Lambda)}\right)\Big]f(y) \right]  \notag\\
&=\varphi\Big(\frac{ y -y_{j*} }{\delta(\Lambda)/3} \Big)\ai \big( \eta_{j, \delta} ib''f\big)(y).
\end{align}
Finally, we can write 
\begin{equation}
\label{decomp4}
    \ai\big[(1 - \eta_{j, \delta}) ib''f \big] = \ai\left[ 1- \varphi \Big ( \frac{ \cdot - y_{j*}}{\delta_0} \Big ) i b'' f   \right] + \ai\left[ \varphi\Big( \frac{ \cdot - y_{j*}}{\delta(\Lambda)} \Big) i b'' f  \right].
\end{equation}
Combining \eqref{decomp2}, \eqref{decomp3}, and \eqref{decomp4} concludes the proof.
\end{proof}
In order to prove the limiting absorption principle for $(\alpha, \lambda) \in R_{dj}$ we will need to prove  bounds for $T_{\Theta}$ on $X^{\sigma_1, \sigma_2}(\mathfrak M)$ for $(\sigma_1, \sigma_2)$ given by \eqref{eqRan1} as well as their properties in higher regularity spaces. We summarize all the required results in the following propositions.

\begin{proposition}
\label{prop:k:bounds}
Assume that  $ (\alpha, \lambda)  \in (-\sigma_0\epsilon^{1/2}, 1) \times \Sigma_{j , \delta_0} ,\, j \in \{ 1, 2\}$,  and $k \in \Z \cap [1, \infty)$. Then for all $\epsilon \in (0, 1/8)$ and $\sigma_1 \in [0, 1]$,  $\sigma_2 \in [-2, 1/2]$ satisfying $\sigma_1 + \sigma_2 \in (-2, 1)$   we have the following bounds on $T_{\Theta}$ for $f \in X^{\sigma_1, \sigma_2}(\mathfrak M)$, 
\begin{align}
    \label{X1}
    \norm{T_{\Theta} f}_{ X^{\sigma_1, \sigma_2}(\mathfrak M) }  \lesssim_{\sigma_1, \sigma_2} 
    \frac{\norm{f}_{X^{\sigma_1, \sigma_2}(\mathfrak M)}}{\lb k\delta(\Lambda) \rb^{\frac{1}{4}}},
\end{align} 
and for $k\delta\lesssim1$ and $|y-y_{j\ast}|\gtrsim \delta_0$,
\begin{equation}\label{X1.5}
    \big|T_{\Theta}f(y)\big| \lesssim_{\sigma_1, \sigma_2}\varrho_j^{-\sigma_1-1/4}\varrho_{j,k}^{-\sigma_2+1/4}(y;\Lambda)\,\norm{f}_{X^{\sigma_1, \sigma_2}(\mathfrak M)}.
\end{equation}
Furthermore, 
\begin{align}
\label{X2}
     \norm{\tc{v}{1}f }_{X^{\sigma_1, \sigma_2}(\mathfrak M)} +
       \norm{\tc{v}{2} f}_{X^{\sigma_1, \sigma_2}(\mathfrak M)} \lesssim_{\sigma_1, \sigma_2} \frac{\epsilon^{1/2}}{\delta(\Lambda)^2} \norm{f}_{X^{\sigma_1, \sigma_2}(\mathfrak M)}.
\end{align}
In particular, in the definition \eqref{Intd1},  $C^{\dagger}$ is chosen  so that we have the bound
\begin{align}
\label{X3}
     \norm{\tc{v}{1} f}_{X^{\sigma_1, \sigma_2}(\mathfrak M)} +
       \norm{\tc{v}{2} f}_{X^{\sigma_1, \sigma_2}(\mathfrak M)} \leq \frac{1}{2} \norm{f}_{X^{\sigma_1, \sigma_2}(\mathfrak M)}.
\end{align}
\end{proposition}

In order to prove the requisite compactness needed for the proof of the limiting absorption principle, we will need a more detailed understanding of the regularity of $T_{\Theta}$ than those implied by Proposition \ref{prop:k:bounds}. The type of regularity we can prove for $T_{\Theta }$ will depend on whether $\lambda$ is in the ``intermediate regime": $ \epsilon^{\frac{1}{2}}\ll |b(y_{j*}) - \lambda| \ll 1$, or the  ``viscous regime":  $|\lambda - b(y_{j*})| \lesssim  \epsilon^{\frac{1}{2}}$. 

We firstly consider the intermediate region.
 \begin{proposition}
 \label{prop:higher:regularity1}
Assume that   $ (\alpha, \lambda)  \in (-\sigma_0\epsilon^{1/2}, 1) \times \Sigma_{j , \delta_0} ,\, j \in \{ 1, 2\}$, $k \in \Z \cap [1, \infty)$, and  $k \delta  \leq N$, for some $N\geq 1$.  Then for all $\epsilon \in (0, 1/8)$, $ \epsilon^{\frac{1}{2}} \ll |\lambda - b(y_{j*})| \ll 1$, $\sigma_1 \in [0, 1]$,  $\sigma_2 \in [-2, 1/2]$ satisfying $\sigma_1+ \sigma_2 \in (-2, 1)$  we have the following bounds for $y \in \T_\tp$. 

\begin{itemize}
    \item (bounds on $\tc{I}{1}$)
    \begin{align}
\big|\varrho_{j,k }^{\beta}(y;\Lambda)\p_y^\beta \tc{I}{1} f(y) \big| \lesssim_N \norm{f}_{X^{\sigma_1, \sigma_2}(\mathfrak M) } \varrho_{j,k}^2(y;\Lambda) k^{\sigma_2},   \, \beta \in \{0, 1, 2 \};
\end{align}
\item (bounds on $\tc{I}{2}$) 
We can write $\p_y\tc{I}{2}f$ as 
\begin{align}
  \p_y \tc{I}{2} f &=    \Upsilon_{I2} +  R_{I2},
\end{align}
where $\Upsilon_{I2}: \T_\tp \to \C$ is compactly supported in $\{y:\,|y-y_{j\ast}|\leq \delta(\Lambda)\}$, $R_{I2}: \T_\tp \to \C$, and satisfy
\begin{align}
& \sup_{\xi \in \R}\Big[(|\xi|+1)|\mathcal{F}(\Upsilon_{I2})(\xi)|\Big] \lesssim_N \frac{\norm{f}_{X^{\sigma_1, \sigma_2}(\mathfrak M)}}{\delta(\Lambda)^{ 1 + \sigma_1 + \sigma_2 }},\\
   &  \big|\tc{I}{2}f(y) \big| + \varrho_{j,k}(y;\Lambda) \big |R_{I2}(y) \big| + \varrho_{j,k}^2(y;\Lambda) \big |\p_y R_{I2}(y) + d_{j,k}^{-1}\Upsilon_{I2}  \big|  \notag\\
   &\lesssim_N \frac{\norm{f}_{X^{\sigma_1, \sigma_2}(\mathfrak M)} }{\delta(\Lambda)^{\sigma_1 + \sigma_2}}   \frac{ \delta(\Lambda) }{\varrho_j(y;\Lambda)};  
\end{align}
\item (bounds on $\tc{v}{1}$)
 \begin{align}
\varrho_{j, k}^\beta(y;\Lambda)|\p_y^{\beta}\tc{v}{1}f(y) | \lesssim_N  \frac{\exp(-c_1 \delta(\Lambda)^2 \epsilon^{-\frac{1}{2}})\delta(\Lambda) \norm{f}_{X^{\sigma_1, \sigma_2}(\mathfrak M) } }{\varrho_{j}(y;\Lambda) \delta(\Lambda)^{\sigma_1 + \sigma_2}} ,\quad \beta \in \{ 0,1 ,2\}, 
    \end{align}
where $c_1>0$ is sufficiently small (and more precisely is given by Lemma \ref{lem:exp:decay} below);

\item  (bounds on $\tc{v}{2}$)

\begin{align}
      |\varrho_{j,k}^\beta(y;\Lambda) \p_y^\beta\tc{v}{2}f(y)| &\lesssim_{N,\sigma_1, \sigma_2} 
 \frac{\epsilon \delta(\Lambda) \norm{f}_{X^{\sigma_1, \sigma_2} 
 (\mathfrak M)}}{\delta(\Lambda)^4\varrho_j(y;\Lambda) \delta(\Lambda)^{\sigma_1 + \sigma_2}}, \quad \beta \in \{0,1 \}.
\end{align}

\end{itemize}
\end{proposition}

We now turn to the viscous region.
\begin{proposition}
\label{prop:higher:reg:2}
Assume that  $ (\alpha, \lambda)  \in (-\sigma_0\epsilon^{1/2}, 1) \times \Sigma_{j , \delta_0} ,\, j \in \{ 1, 2\}$,  $k \in \Z \cap [1, \infty)$, and $k \delta  \leq N$. Then for all $\epsilon \in (0, 1/8)$,   $|\lambda - b(y_{j*})| \lesssim \epsilon^{\frac{1}{2}}$, and $f \in X^{\sigma_1, \sigma_2}(\mathfrak M)$ satisfying $\sigma_1 \in [0, 1]$, $\sigma_1 + \sigma_2 \in (-2, 1)$ the following bounds hold for $y \in \T_\tp$,

\begin{itemize} 
\item (bounds on $\tc{I}{1}$)
    \begin{align}
\big|\varrho_{j,k }^{\beta}(y;\Lambda)\p_y^\beta \tc{I}{1} f(y) \big| \lesssim_N \norm{f}_{X^{\sigma_1, \sigma_2}(\mathfrak M) } \varrho_{j,k}^2(y;\Lambda) k^{\sigma_2},   \, \beta \in \{0, 1, 2 \};
\end{align}

\item (bounds on $\tc{I}{2}$)
\begin{align}
\big|\varrho_{j,k}^\beta(y;\Lambda)\p_y^{\beta}T_{I2} f(y) \big| + \big|\varrho_{j,k}^\beta(y;\Lambda)\p_y^{\beta}T_{v1} f(y) \big|  &\lesssim_{N, \sigma_1, \sigma_2} \frac{\norm{f}_{X^{\sigma_1, \sigma_2}}}{\delta(\Lambda)^{\sigma_1 + \sigma_2} }  \frac{\delta(\Lambda)}{\varrho_j(y;\Lambda)}, \, \beta \in \{0, 1, 2 \};
\end{align}
\item (bounds on $\tc{v}{1}$ and $\tc{v}{2}$)
\begin{align}
      \big|\varrho_{j,k}^\beta(y;\Lambda) \p_y^\beta\tc{v}{2}f(y) \big| &\lesssim_{N, \sigma_1, \sigma_2} 
 \frac{\epsilon \delta(\Lambda) \norm{f}_{X^{\sigma_1, \sigma_2} 
 (\mathfrak M)}}{\delta(\Lambda)^4\varrho_j(y;\Lambda) \delta(\Lambda)^{\sigma_1 + \sigma_2}}, \quad \beta \in \{0,1 \},\\
  \big|\varrho_{j,k}^2(y;\Lambda) \p_y^2\tc{v}{2}f(y) \big| &\lesssim_{N, \sigma_1, \sigma_2} \frac{\norm{f}_{X^{\sigma_1, \sigma_2} (\mathfrak M) }}{ \varrho_j^{\sigma_1}(y;\Lambda)\varrho_{j,k}^{\sigma_2}(y;\Lambda)    }.
\end{align}

\end{itemize}

    
\end{proposition}
\begin{remark}
\label{rem:kdelt}
    The implicit constants appearing in Propositions \ref{prop:k:bounds}, \ref{prop:higher:regularity1}  and \ref{prop:higher:reg:2} may grow with $N$. However, in the application of these propositions in the proof of the limiting absorption principle, we will only need to consider the case where $N$ is fixed.  
    We will only prove Propositions \ref{prop:higher:regularity1} and \ref{prop:higher:reg:2} under the assumption that $N =1$, but an inspection of the proofs shows  the general case follows with perhaps a larger constant. 
\end{remark}

\begin{remark}
    The role of higher regularity  is to capture the ``uniform compactness" of $T_{\Theta}$ (in the sense given in Lemma \ref{lem:T:bdd}). Some care is needed here as the function space $X^{\sigma_1, \sigma_2}(\mathfrak M)$ itself has non-trivial $\lambda$ dependence. The essential feature is that, upon rescaling by $\delta(\Lambda)$, the higher order regularity estimates encode compactness on the fixed space $H_{loc}^1(\R)$. This is made precise in section \ref{sec:concentrated}.  
\end{remark}
The proofs of Propositions \ref{prop:k:bounds}, \ref{prop:higher:regularity1}, and \ref{prop:higher:reg:2} are very important but are somewhat technical. The full proofs will be given in Section \ref{sec:T}.

\subsection{The compactness argument}\label{sec:deg}
Using Propositions \ref{prop:k:bounds}, \ref{prop:higher:regularity1}, and \ref{prop:higher:reg:2} we are now ready to prove the limiting absorption in the degenerate case. For the ease of notations, we assume that $y_{j\ast}=0$ and also drop the dependence of various quantities on $j\in\{1,2\}$ which is fixed, when there is no possibility of confusion. 

We firstly consider the intermediate region where the spectral parameter is $|\nu/k|^{1/2}$ away from the critical value. 
\begin{proposition}
    \label{prop:lap:intermediate} 
     Fix $\gamma \in [7/4, 2)$. 
     Assume that 
     \begin{equation}
         (\sigma_1, \sigma_2) \in \{ (0, -\gamma), (1, -\gamma + 1), (1 - 2(2 - \gamma), -\gamma + 2)    \},
     \end{equation}
     and let $\sigma_0$ be as in Proposition \ref{Airy_main}. 
    There exists $\epsilon_{0,I} \in (0, 1/8), \kappa_I > 0, M>0$, depending on $\gamma$, such that for all $\epsilon \in (0, \epsilon_{0, I})$, $k \in \Z \cap [1, \infty)$, $\alpha \in [-\sigma_{0}\epsilon^{1/2} , 1]$ , and $\lambda \in \Sigma_{j,\delta_0} \setminus [-M \epsilon^{1/2}, M \epsilon^{1/2}   ]$ with $j\in\{1,2\}$, we have 
\begin{equation}\label{eq:lap:intermediate}
    \norm{f + T_{\Theta}f}_{X^{\sigma_1, \sigma_2}(\mathfrak{M}) } \geq \kappa_I \norm{f}_{X^{\sigma_1, \sigma_2}(\mathfrak{M})  } 
\end{equation}
for all $f \in X^{\sigma_1, \sigma_2}(\mathfrak M) $.
\end{proposition}

We now turn to the viscous region where $|\lambda|\lesssim |\nu/k|^{1/2}$. 
\begin{proposition}
    \label{prop:lap:viscous}
    Let $\gamma, (\sigma_1, \sigma_2), \text{ and } M$ be as in Proposition \ref{prop:lap:intermediate}.
    There exist $\epsilon_{0, v} \in (0, 1/8)$, $\kappa_v > 0 $, and $\sigma_{0}(M) \in (0,1)$ such that for all $\epsilon \in (0, \epsilon_{0, v})$, $k \in \Z \cap [1, \infty)$, $\alpha \in [-\sigma_{0}(M)\epsilon^{1/2}, 1]$, and $\lambda \in  [-M \epsilon^{1/2}, M \epsilon^{1/2}   ]$ we have 
\begin{equation}\label{eq:lap:viscous}
    \norm{f + T_{\Theta}f}_{X^{\sigma_1, \sigma_2}(\mathfrak M)  } \geq \kappa_v \norm{f}_{X^{\sigma_1, \sigma_2} (\mathfrak M)} 
\end{equation}
for all $f \in X^{\sigma_1, \sigma_2}(\mathfrak M) $.
\end{proposition}
\begin{remark}
    In order to prove Proposition \ref{prop:lap:viscous} we will need to choose $\sigma_{0}(M)$ small based on  $M$ (see \eqref{viscous:lap:key}). However,  the proof of Proposition \ref{prop:lap:intermediate} only requires  smallness of $\sigma_{0}$ in a way such that Proposition \ref{Airy_main}  holds.
\end{remark}
If we set $\kappa := \min(\kappa_I, \kappa_v)$, $\epsilon_0 = \min(\epsilon_{0,I}, \epsilon_{0,v}   )$, and $\sigma_\sharp := \min (\sigma_0, \sigma_0(M))$  then Propositions \ref{prop:lap:intermediate} and \ref{prop:lap:viscous} immediately imply the following result.
\begin{proposition}
     \label{prop:lap:final}
     Fix $\gamma \in [7/4, 2)$. 
     Assume that 
     \begin{equation}
         (\sigma_1, \sigma_2) \in \{ (0, -\gamma), (1, -\gamma + 1), (1 - 2(2 - \gamma)), -\gamma + 2)    \}.
     \end{equation}
    Then there exists $\epsilon_{0}, \kappa, \sigma_\sharp > 0 $, depending on $\gamma$, such that for all $\epsilon \in (0, \epsilon_{0, })$, $k \in \Z \cap [1, \infty)$, $\alpha \in [-\sigma_\sharp \epsilon^{1/2}, 1]$, and $\lambda \in \Sigma_{j,\delta_0}$ with $j\in\{1,2\}$, we have 
\begin{equation}\label{eq:lap:final}
    \norm{f + T_{\Theta}f}_{X^{\sigma_1, \sigma_2}(\mathfrak M)} \geq \kappa \norm{f}_{X^{\sigma_1, \sigma_2}(\mathfrak M)} 
\end{equation}
for all $f\in X^{\sigma_1, \sigma_2}(\mathfrak M)$.
\end{proposition}

In the rest of the section we give the proof of Proposition \ref{prop:lap:intermediate} and Proposition \ref{prop:lap:viscous}, beginning with the proof of Proposition \ref{prop:lap:intermediate}.
\begin{proof}[Proof of Proposition \ref{prop:lap:intermediate}]
Suppose \eqref{eq:lap:intermediate} does not hold. Then there exist 
$$(f_\ell, \kappa_{I, \ell}, \epsilon_\ell, \lambda_\ell, k_\ell, \alpha_{\ell}, M_\ell),\quad {\rm with}\,\,(\kappa_{I,\ell}, \epsilon_\ell, M_\ell^{-1}) \to 0+, \lambda_\ell \in \Sigma_{j,\delta_0}\setminus \big[-M_\ell \epsilon_\ell^{1/2}, M_\ell \epsilon_\ell^{1/2}\big],$$
and  $\norm{f_\ell }_{X^{\sigma_1, \sigma_2}(\mathfrak M)} = 1$, such that  
    \begin{equation}
    \label{eq:lap00}
        \norm{f_\ell + T_{\Theta_\ell} f_\ell}_{X^{\sigma_1, \sigma_2}(\mathfrak M)} \leq \kappa_{I, \ell}.
    \end{equation}
 In \eqref{eq:lap00} we have introduced $\Theta_\ell$ to represent dependence on  $\epsilon_\ell, \lambda_\ell, k_\ell, \alpha_{\ell}$.
 
 We claim that we may assume $(\alpha_\ell, M_\ell\epsilon_\ell^{1/2}, \lambda_\ell) \to 0$. To see this, we first note that 
 $$\liminf_{\ell \to \infty}{M_\ell \epsilon_\ell^{1/2}} < \infty,$$ 
  otherwise there is nothing to prove by our assumption on the range of the spectral parameter $\lambda_\ell$. We may therefore assume, passing to subsequences if necessary, that  $(\alpha_\ell, M_\ell \epsilon_\ell^{1/2}, \lambda_\ell)$ converges. 
 If $\lim_{\ell \to \infty} \alpha_\ell$ or $\lim_{\ell \to \infty} \lambda_\ell$ are not zero, then $\lim_{\ell \to \infty} \delta(\Lambda_\ell) > 0$.
  In this case we may follow similar arguments as in the proof Proposition \ref{prop:lap:non}, using the higher order regularity proved for $\tc{I}{1}$, $\tc{I}{2}$ in Propositions \ref{prop:higher:regularity1} and \ref{prop:higher:reg:2}, and the bounds proved for $\tc{v}{1}$ and $\tc{v}{2}$ on $X^{\sigma_1, \sigma_2}(\mathfrak M)$ in Proposition \ref{prop:k:bounds}, which ensure the needed convergence for $T_{\Theta_\ell} f_\ell $. 
  
  We can therefore assume that  $(\alpha_\ell,\epsilon_\ell, M_\ell \epsilon_\ell^{1/2}, \lambda_\ell)$ converges to 0, which implies that $\delta(\Lambda_\ell) $ converges to 0. Proposition \ref{prop:k:bounds}  also implies that there exists a subsequence such that $$k_\ell \delta(\Lambda_\ell) \lesssim 1.$$ 

There are two main cases to consider: in the first case, a uniform-in-$\ell$ amount of $f_{\ell}$ stays $\delta_0$ distance away from the critical point (the ``Non-concentrated" case); in the second case, the sequence $f_\ell$ concentrates near the critical point (the ``Concentrated" case). 
More precisely, we distinguish two scenarios: 
\begin{enumerate}
    \item The Non-Concentrated case
     \begin{equation}
     \label{eq:nonconentration}
        \limsup_{\ell \to \infty}\big\|\varrho^{-\sigma_1}\varrho_{k_\ell}^{-\sigma_2}(\cdot;\Lambda)f_{\ell}\big\|_{L^\infty(\T_{\tp} \backslash S_{ \frac{\delta_0}{2}})} > 0,
    \end{equation}

    \item  Concentrated case
     \begin{equation}
     \label{eq:concentration}
        \limsup_{\ell \to \infty}\big\|\varrho^{-\sigma_1}\varrho_{k_\ell}^{-\sigma_2}(\cdot;\Lambda)f_{\ell}\big\|_{L^\infty(\T_{\tp} \backslash S_{ \frac{\delta_0}{2}})} = 0.
    \end{equation}

\end{enumerate}

\subsubsection{The non-concentrated case}
\label{sec:noncon}
It follows from \eqref{X1.5}, \eqref{eq:lap00} and \eqref{eq:nonconentration} that $|k_\ell|\lesssim1$. 
We introduce the notation
   \begin{subequations}
   \label{def:T}
       \begin{align}
    \tc{I}{1\ell}f(y) &:=  \int_{\T_\tp} \grl \ail ((1- \varphi_{\delta_0}) b'' f ) \, dz, \\
    \tc{I}{2\ell}f(y) &:=  \int_{\T_\tp} \grl \ail ( \varphi_{\delta(\Lambda_\ell)} b'' f) \, dz, \\
    \tc{v}{1 \ell}f(y) &:=  \int_{\T_\tp} \grl \varphi_{ \frac{\delta(\Lambda_\ell)}{{3}}  }(z) \ail( \eta_{\delta(\Lambda_\ell)} b'' f)\, dz,  \\
    \tc{v}{2\ell}f(y) &:= \int_{\T_\tp} \grl (1 -\varphi_{\frac{\delta(\Lambda_\ell)}{{3}}  }(z) )\frac{\epsilon \p_z^2\ail( \eta_{\delta(\Lambda_\ell)} b'' f)}{b(z) - \lambda_\ell - i  \alpha_\ell }\, dz .
\end{align}
\end{subequations}

By Propositions   \ref{prop:higher:regularity1} and \ref{prop:higher:reg:2}, 
it follows that for all $\lambda \in \Sigma_{j,\delta_0}$ with $k_\ell \delta(\Lambda_\ell) \lesssim 1$  and $y \in \T_\tp$, we have
    \begin{align}
    \label{eq:lap1}
      |\tc{I}{2 }f(y)| +  |\tc{v}{1}f(y)| + |\tc{v}{2}f(y)| &\lesssim \frac{\delta(\Lambda_\ell)  \norm{f}_{X^{\sigma_1, \sigma_2}(\mathfrak M) } }{\varrho(y; \Lambda_\ell) \delta(\Lambda_\ell)^{\sigma_1 + \sigma_2} }, 
      \end{align}
      where we used the fact that   $\epsilon_\ell \lesssim \delta^4(\Lambda_\ell)$. 
      
      The pointwise inequality \eqref{eq:lap1} directly implies that   
  \begin{equation}
\label{eq:lap0}
    \lim_{\ell \to \infty} \left(\norm{\tc{I}{2 \ell }f_\ell}_{L^p(\T_\tp)} + \norm{\tc{v}{1\ell}f_\ell}_{L^p(\T_\tp)} + \norm{\tc{v}{2\ell}f_\ell}_{L^p(\T_\tp)} \right) =  0,
\end{equation}
for $p>1$ with $(\gamma-1)p<1$. 

We now show that we can extract a convergent sequence whose limit satisfies Rayleigh's equation. By weak compactness of $L^p$, we can find $f_* \in L^p(\T_\tp)$ such that 
$f_\ell \rightharpoonup f_*$. This weak convergence can be upgraded to strong convergence using \eqref{eq:lap00}, \eqref{eq:lap0}, and the higher regularity of $\tc{I}{1\ell}$.  
Passing $\ell\to\infty$, we obtain that 
\begin{equation}
\label{eq:lap3}
    f_* + \lim_{\ell \to \infty}\tc{I}{1 \ell}f_{*}  = 0 \quad \text{in } L^p(\T_\tp). 
\end{equation}
Using Lemma \ref{LAMweight}, we obtain that $f_*$ is bounded and moreover, for $y\in\T_\tp$
\begin{equation}
\label{eq:lap4}
    |f_*(y)| = |\lim_{\ell \to \infty}\tc{I}{1\ell}f_{*}(y)| \lesssim y^2. 
\end{equation}

To compute $\lim_{\ell \to \infty} \tc{I}{1\ell}f_*$, we observe that  $\mathcal{G}_\ell$ and $A_\ell^{-1}( (1 - \varphi_{\delta_0}) b'' f_*)$ have convergent subsequences in $L^{\infty}(\T_\tp)$ from Arzela-Ascoli, Lemma \ref{mGk50}, and  Lemma \ref{lem:limiting:behavior}. Therefore 
\begin{equation}
    \label{eq:lap4.5}
    \lim_{\ell \to \infty} \tc{I}{1\ell}f_* = \int_{\T_\tp} \mathcal{G}_*(y,z)A_*^{-1} ((1 - \varphi_{\delta_0}) b'' f_*) \, dz
\end{equation}
where we have defined  $\mathcal{G}_*:= \lim_{\ell \to \infty} \mathcal{G}_\ell$ and $A_*^{-1} := \lim_{\ell \to \infty} A_\ell^{-1}$. By the proof of Lemma \ref{lem:limiting:behavior} we have 
\begin{equation}
\label{eq:lap6}
    A_*^{-1}((1 - \varphi_{\delta_0}) b'' f_*)  =  \frac{(1 - \varphi_{\delta_0}(y))b''(y) f_* (y)}{b(y)} 
\end{equation}
Additionally, for $z \neq 0$,
\begin{equation}
\label{eq:lap5}
    \mathcal{L}_* \mathcal{G}_*(y,z)   := \left(-\p_y^2 + k_*^2 + \frac{b''(y)}{b(y)}\varphi_{\delta_0}(y) \right) \mathcal{G}_*(y,z) = \delta(y -z) 
\end{equation}
in the sense of distributions.
Applying $\mathcal{L}_*$ to \eqref{eq:lap3} gives
\begin{equation}
\label{eq:lap7}
    \left(-\p_y^2 + k_*^2 + \frac{b''(y)}{b(y)}\right)f_*(y) = 0,
\end{equation}
for $y\in\T_\tp$, in the sense of distributions. This
 contradicts our main spectral assumption \ref{MaAs}.







\subsubsection{The concentrated case}
\label{sec:concentrated}
For the remainder of the proof we adopt the shorthand $\delta_\ell = \delta(\Lambda_\ell)$ with a slight abuse of notation when $\ell\in\{0,1,2\}$. Here our main concern is when $\ell\to\infty$ so there should be no possibility of confusion.

By \eqref{X3}, \eqref{eq:lap00}, \eqref{eq:concentration} and  Lemma \ref{LAMweight}, we can find an $N$ such that for $\ell \geq  N$ we have 
\begin{equation}
    \label{eq:lap8}
   \norm{f_\ell + T_{\Theta_\ell} f_{\ell}}_{X^{\sigma_1, \sigma_2}(\mathfrak M)} +  \norm{(\tc{v1}{\ell} + \tc{v}{2\ell})
 f_\ell}_{X^{\sigma_1, \sigma_2}(\mathfrak M)} + \norm{\tc{I}{1 \ell}f_\ell}_{X^{\sigma_1, \sigma_2}(\mathfrak M)} \leq \frac{2}{3},
\end{equation}
which implies that
\begin{equation}
    \label{eq:lap10}
    \norm{\tc{I}{2 \ell}f_\ell}_{X^{\sigma_1, \sigma_2}(\mathfrak M)}\ge  \frac{1}{3}. 
\end{equation}
Since $\tc{I}{2\ell}f_\ell = \tc{I}{2\ell}( \varphi_{2 \delta_\ell}f_\ell)$ and $\tc{I}{2\ell}$ are uniformly bounded on $X^{\sigma_1, \sigma_2}(\mathfrak M)$ by Proposition \ref{prop:k:bounds},
it follows from \eqref{eq:lap10} that there exists $\theta >0$ which is independent of $\ell$, such that 
\begin{align}
    \label{eq:lap11}
&  
 \delta_\ell^{\sigma_1 + \sigma_2  } \Big[ \delta_\ell^{-\frac{1}{2}}\Big(\norm{f_\ell}_{L^2(S_{\delta_\ell})} + \norm{\delta_\ell \p_yf_{\ell}}_{L^2(S_{\delta_\ell})} \Big) +  \norm{ f_\ell }_{L^{\infty}(S_{4\delta_\ell} \setminus S_{\delta_\ell }) } + \norm{ \delta_\ell\p_y f_\ell }_{L^{\infty}(S_{4 \delta_\ell}\setminus S_{\delta_\ell } )} \Big]\ge \theta. 
\end{align}

In view of the fact that a non-vanishing amount of $f_\ell$ concentrates near the critical point, we use a blow-up argument. 
To this end, define the re-scaled quantities
\begin{subequations}
\label{def:rescaled}
    \begin{align} 
    \label{eq:coordinates:rescaled}
    (Y,Z) &:= \left(\frac{y}{\delta_\ell}, \frac{z}{\delta_\ell}\right) \quad (y,z) \in \T_\tp^2,\\
    \label{eq:vel:rescaled}
    B_\ell(Y)&:= \delta_\ell^{-2}b(y), \\
    \label{eq:seq:rescaled}
    F_\ell(Y) &:= \delta_\ell^{\sigma_1 + \sigma_2} f_\ell(y), \\
    \label{eq:gr:rescaled}
    G_{\ell}(Y,Z) &:= \delta_\ell^{-1}\grl, \\
     \label{eq:ai:rescaled}
    \widetilde{A}_{\ell}^{-1}(Q_\ell)(Y) &:= \delta_\ell^{2 + \sigma_1 + \sigma_2 } \ail(q)(y), \quad \text{ where } Q_\ell(Y) := \delta_\ell^{\sigma_1 + \sigma_2}q(y).
\end{align}
\end{subequations}
It will be convenient to extend the rescaled functions so that they are defined on $\R$. For  the $Y$ values not in the range of  $\frac{y}{\delta_\ell}, y \in [-\tp/2,  \tp/2]$, we simply extend all of the functions as zero. In this way, we can use all of the previously obtained bounds to control the sequence. 

It follows from rescaling that $G_\ell$ satisfies for $|Y|< \tp/(2\delta_\ell)$ and $|Z|< \tp/(2\delta_\ell)$,
\begin{equation}
    -\p_Y^2 G_\ell(Y,Z) + (k_\ell \delta_\ell)^2 G_\ell(Y,Z) + \frac{B_\ell''(Y)(\varphi(\delta_0^{-1}\delta_\ell Y) - \varphi(Y)    )}{B_\ell(Y) - (\lambda_\ell + i \alpha_\ell)\delta_\ell^{-2}} G_\ell(Y,Z) = \delta(Y-Z),
\end{equation}
and that $\widetilde{A}_\ell^{-1}(Q_\ell)(Y)$ satisfies for $|Y|< \tp/(2\delta_\ell)$ and $|Z|< \tp/(2\delta_\ell)$,
\begin{equation}
\label{eq:airy:op:rescale}
    \frac{\epsilon_\ell}{\delta_\ell^4}\p_Y^2 \widetilde{A}_\ell^{-1}(Q_\ell)(Y) - i\left(B_\ell(Y) - \frac{\lambda_\ell - i\alpha_\ell }{\delta_\ell^2} \right) \widetilde{A}_{\ell}^{-1}(Q_\ell)(Y) = Q_\ell(Y).
\end{equation}
Crucially, \eqref{eq:lap11} and \eqref{def:rescaled}   give uniform-in-$\ell$ lower bounds on $F_\ell$ in $H^1[-1, 1]$ intersection with $$W^{1, \infty}\big([-4,4]\setminus [-1, 1]\big).$$
Finally, we define, with \eqref{eq:coordinates:rescaled},
\begin{subequations}
\begin{align}
\label{def:t:rescaled}
  K_{\ell}(Y) &:= \delta_\ell^{\sigma_1 + \sigma_2}T_{\Theta_\ell} f_{\ell}(y),  
    \\
     K_{1 \ell}(Y) &:= \delta_\ell^{\sigma_1 + \sigma_2}T_{I1 \ell} f_\ell(y),    \\
    K_{2 \ell}(Y) &:= \delta_\ell^{\sigma_1 + \sigma_2 }T_{I2 \ell}f_\ell(y), \\
    K_{3 \ell}(Y) &:= \delta_\ell^{\sigma_1 + \sigma_2 }T_{v1 \ell}f_\ell(y),\\
    K_{4 \ell}(Y) &:= \delta_\ell^{\sigma_1 + \sigma_2 }T_{v2 \ell}f_\ell(y).
\end{align}
\end{subequations}
Using \eqref{def:rescaled} and \eqref{def:T}, we obtain that
\begin{subequations}
\label{K:bounds}
   \begin{align}
   K_{1 \ell}(Y) &= \int_\R G_\ell(Y,Z) \widetilde{A}_{\ell}^{-1}\left(\left(1 -\varphi_{\frac{\delta_0}{\delta_\ell}} \right) B_\ell'' F_\ell\right)(Z) \, dZ,\\ 
   K_{2 \ell}(Y) &= \int_\R G_\ell(Y,Z) \widetilde{A}_{\ell}^{-1}\left(\varphi  B_\ell'' F_\ell\right)(Z) \, dZ, \\
   K_{3 \ell}(Y) &= \int_\R G_\ell(Y,Z) \varphi\left(  \frac{Z}{3}\right)
 \widetilde{A}_\ell^{-1}\left(  \Phi(\delta_\ell \cdot) B_\ell'' F_\ell  \right)(Z) \, dZ,\\ 
 K_{4 \ell}(Y) &= \int_\R G_\ell(Y,Z)  \left(1 -\varphi\left(  \frac{Z}{3}\right) \right) \frac{\epsilon_\ell}{\delta_\ell^4}   \frac{\p_Z^2 \widetilde{A}_\ell^{-1} (\eta_{\delta}(\delta_\ell \cdot) B_\ell'' F_\ell )(Z)   }{B_\ell(Z) - \lambda_{\ell}\delta_\ell^{-2} - i \alpha_\ell \delta_\ell^{-2} }   \, dZ.
\end{align}
\end{subequations}
Propositions \ref{prop:higher:regularity1} and \ref{prop:higher:reg:2} imply for $\beta \in \{0, 1 \}$ and $1\leq|Y|\ll \delta_\ell^{-1}$ that 
\begin{subequations}
\label{eq:rescaled:T:bounds}
    \begin{align}
    \min\{1 +|Y|, (k_\ell \delta_\ell)^{-1}  \}^\beta|\p_Y^{\beta}K_{1\ell}(Y)| &\lesssim \delta_\ell^{2  - |\sigma_2| }( 1 + |Y|^2), \\ 
    \min\{1 +|Y|, (k_\ell \delta_\ell)^{-1}  \}^\beta|\p_Y^{\beta}K_{2\ell}(Y)| &\lesssim \frac{1}{1 +|Y| },   \\
     \min\{1 +|Y|, (k_\ell \delta_\ell)^{-1}  \}^\beta|\p_Y^{\beta}K_{3\ell}(Y)| &\lesssim    \frac{\exp \left( -c_0 \delta_\ell^2 \epsilon_\ell^{-\frac{1}{2}} \right) }{1 + |Y|}, \\
     \min\{1 +|Y|, (k_\ell \delta_\ell)^{-1}  \}^\beta|\p_Y^{\beta}K_{4\ell}(Y)| &\lesssim \frac{\epsilon_\ell}{ \delta_\ell^4 (1 + |Y|)}. 
\end{align}
\end{subequations}
From the higher order regularity estimates proved in Propositions \ref{prop:higher:regularity1} and \ref{prop:higher:reg:2} 
it follows that $K_\ell$ is relatively compact in $H^1([-1, 1])$  and $W_{{loc}}^{1, \infty}(\R\setminus[-1, 1])$.
We conclude that $K_\ell$ has a convergent subsequence in $H_{\rm loc}^{1}(\R)$, the limit of which we call $K_*$. Combined with \eqref{eq:lap00} and \eqref{def:rescaled}, this implies that there exists a subsequence such that $F_\ell$ converges to $F_*$ in $H_{\rm loc}^{1}(\R)$.  From the convergence of $F_\ell, K_\ell$ and \eqref{eq:lap00} we have that 
\begin{equation}
\label{eq:lap12}
    F_* + K_* = 0.
\end{equation}
From the spatial decay of $K_\ell$ in \eqref{eq:rescaled:T:bounds}  we can deduce that $F_*$ is in $H^1(\R)$.
To further understand the limiting solution $F_*$, 
we observe that 
\begin{equation}
\label{eq:vanishing:viscosity}
   \lim_{\ell \to \infty} \frac{\epsilon_\ell}{\delta_\ell^4} \lesssim  \lim_{\ell \to \infty} \frac{1}{M_\ell^2}  = 0.
\end{equation}
Therefore, from \eqref{eq:rescaled:T:bounds}, for $Y,Z\in\R$,
\begin{equation}
    K_* = \lim_{\ell \to \infty} K_{2\ell} = \lim_{\ell \to \infty}\int_\R G_*(Y,Z) \widetilde{A}_{\ell}^{-1}\left(\varphi  b''(0) F_*\right) \, dZ, 
\end{equation}
 where $G_*$ satisfies for $Y,Z\in\R$,
\begin{align}
    \label{eq:lap13}
    \widetilde{\mathcal{L}_*}G_* := -\p_Y^2 G_*(Y,Z) + k_*^2 G_*(Y,Z) + \frac{b''(0) (1 - \varphi(Y))}{\frac{b''(0)}{2}Y^2 - \lambda_*  - i\alpha_*    }G_*(Y,Z) = \delta(Y-Z),
\end{align}
with $\lambda_* := \lim_{\ell \to \infty } \lambda_\ell\delta_\ell^{-2}$, $\alpha_* := \lim_{\ell \to \infty} \alpha_\ell \delta_\ell^{-2}$
and $k_* := \lim_{\ell \to \infty} k_\ell \delta_\ell $ (passing to  subsequences if necessary to ensure that all sequences converge). 
Given the definition of $\delta(\Lambda)$, one of the following must be true:
\begin{enumerate}
    \item $|\alpha_*| > 0 $,
    \item $ |\lambda_*| > 0$.
\end{enumerate}
By using the re-scaled bounds for the relevant functions, we can follow the proof of Lemma \ref{lem:limiting:behavior}, and obtain that $\widetilde{A}_\ell^{-1}(Q)(Y)$ has the following limiting behavior:
\begin{itemize}
    \item If $\alpha_* \neq 0$ or if $\frac{\lambda_*}{b''(0)} < 0 $, then in the sense of distributions for $Y\in\R$,
    \begin{equation}
    \label{eq:alpha+:rescale}
         \lim_{\delta_\ell^{-4}\epsilon_\ell \to 0, \, \delta_\ell^{-2}\alpha_\ell \to \alpha_*, \, \delta_\ell^{-2} \lambda_\ell \to \lambda_* } \widetilde{A}_{\ell}(Q)(Y) = \frac{ -Q(Y)}{ \alpha_* +   i( \frac{b''(0) Y^2}{2} - \lambda_*)}, 
    \end{equation}
    \item If $\alpha_* = 0$ and $\frac{\lambda_*}{b''(0)} > 0 $, then in the sense of distributions for $Y\in\R$,
    \begin{align}
    \label{eq:alpha0:rescale}
   \lim_{ \delta_\ell^{-4}\epsilon_\ell \to 0,\,  \delta_\ell^{-2}\alpha_\ell \to  0, \, \delta_\ell^{-2}\lambda_\ell \to \lambda_*} \widetilde{A}_\ell(Q)(Y) = \text{P.V.} \frac{ -Q(Y)}{ i(\frac{b''(0)Y^2}{2} - \lambda_*)} + \pi \sum_{\frac{b''(0)c^2}{2} = \lambda_*}  \frac{Q(c)}{|b''(0)c|}  \delta( Y - c).
\end{align}
\end{itemize}

If $\alpha_* \neq 0$ or if $\frac{\lambda_*}{b''(0)} < 0 $, then \eqref{eq:alpha+:rescale} implies that for $Y\in\R$,
\begin{equation}
    \lim_{\ell \to \infty}\int_\R G_*(Y,Z) \widetilde{A}_{\ell}^{-1}\left(\varphi  b''(0) F_*\right) \, dZ  = \int_\R G_*(Y,Z) \frac{\varphi(Z)  b''(0) F_*(Z)}{\frac{b''(0)Z^2}{2} -\lambda_* -i \alpha_* }  \, dZ. 
\end{equation}
Applying $\widetilde{\mathcal{L}}_*$ to \eqref{eq:lap12}
 we get that 
\begin{equation}
    -\p_Y^2 F_* + k_*^2 F_* + \frac{b''(0)}{\frac{b''(0)}{2}Y^2 - \lambda_*  -i\alpha_*    }F_* = 0,
\end{equation}
for $Y\in\R$ in the sense of distributions. This contradicts the well-known fact that the linearized operator associated with the Poiseuille flow has no discrete eigenvalues. 

Now assume that $\alpha_* = 0$ and $\frac{\lambda_*}{b''(0)} > 0 $.  Using \eqref{eq:alpha0:rescale}
we obtain that for $Y\in\R$,
\begin{equation}
 F_*(Y) + \int_\R \frac{G_*(Y,Z) (1 - \varphi(Z))b''(0)F(Z)}{\frac{b''(0)}{2}Z^2 - \lambda_*   }dZ  + i\pi\sum_{ \frac{b''(0)c^2}{2} = \lambda_*}  \frac{G_*(Y,c)  F_*(c) }{|b''(0)c|} = 0,
\end{equation}
 which implies that for $Y\in \R$,
\begin{equation}\label{rescaleJ1}
    -\p_Y^2 F_*(Y) + k_*^2 F_*(Y) + \text{P.V.}\frac{b''(0)}{\frac{b''(0)}{2}Y^2 - \lambda_*    }F_*(Y) +  i \pi \sum_{\frac{b''(0)c^2}{2} = \lambda_*}  \frac{ F_*(c) }{|b''(0)c|} \delta(Y - c) = 0,
\end{equation}
in the sense of distributions. Multiplying $\overline{F_\ast}$ with \eqref{rescaleJ1}, integrating over $\R$ and taking the imaginary part, we see that $F_\ast(c)=0$. This again contradicts with the fact that the linearized operator associated with the Poiseuille flow has no discrete eigenvalues. The proof is now complete.
\end{proof}

Now that we have quantified the boundary between the intermediate  viscous regime through $M$, we can give the proof of Proposition \ref{prop:lap:viscous}.
\begin{proof}[Proof of Proposition \ref{prop:lap:viscous}]
The proof proceeds identically to that of Proposition \ref{prop:lap:intermediate} except in the concentrated case in section \ref{sec:concentrated}. For $\lambda \in [-M\epsilon^{1/2}, M \epsilon^{1/2}]$ we no longer have \eqref{eq:vanishing:viscosity}. 
Since $\frac{\epsilon_\ell}{\delta_\ell^4}$ is bounded, it has a convergent subsequence, the limit of which we denote as $\epsilon_* $. If $\epsilon_* = 0$, then we argue as in Proposition \ref{prop:lap:intermediate}. Therefore, we assume that   $\epsilon_* >0 $. From \eqref{eq:lap12},  
 $K_*$ is now of the form
\begin{align}
     K_* &= \lim_{\ell \to \infty} (K_{2\ell} +K_{3\ell} + K_{4\ell})  \notag \\
     &= \int_\R G_*(Y,Z) \widetilde{A}_{*}^{-1}\left(\varphi  b''(0) F_*\right) dZ + \int_\R G_*(Y,Z) \varphi \left(\frac{Z}{3} \right)\widetilde{A}_{*}^{-1}\left( b''(0) F_*\right) dZ \notag\\
     &+  \int_\R G_*(Y,Z)  \left(1 -\varphi\left(  \frac{Z}{3}\right) \right) \epsilon_*   \frac{\p_Z^2 \widetilde{A}_*^{-1} ( b''(0) F_* )   }{\frac{b''(0)}{2}Z^2 - \lambda_{*} - i \alpha_*  }   \, dZ.
\end{align}
Using the regularity of $K_*$, we can bootstrap the regularity of $F_*$ to prove that $F_* \in H^2(\R)$ and satisfies for $Y\in\R$,
\begin{align}
\label{eq:os:poiseuille}
    &\epsilon_* \p_Y^2 W_* - \alpha_* W_* -i\left(\frac{b''(0)Y^2}{2}  - \lambda_* \right) W_*  +ib''(0)F_* = 0, \notag\\
   &W_* = (\p_Y^2 - |k_*|^2) F_* .
\end{align}
Our goal is to derive a contradiction by showing the solution of \eqref{eq:os:poiseuille} must be zero. To this end, we use a direct energy method.  
Multiplying by $\overline{W_*}$ and integrating over $\R$, we obtain that
    \begin{align}
    \label{eq:poiseuille:energy}
    \epsilon_*\int_\R |\p_Y W_*(Y)|^2dY + \alpha_*\int_\R |W_*(Y)|^2 dY = 0,\notag\\
    \int_\R \left(\frac{1}{2} Y^2 - \frac{\lambda_*}{b''(0)} \right) |W_*(Y)|^2 dY + \int_\R |\p_YF_*(Y)|^2 + |k_*F_*(Y)|^2 dY = 0. 
\end{align}
If $\frac{\lambda_*}{b''(0)}\leq0$ we have a contradiction so we assume that it is positive. Similarly, if $\alpha_* \geq 0$ we immediately have a contradiction  so we also assume that  $  -\sigma_0(M) \epsilon_*^{1/2} \leq \alpha_* < 0$. Combining  \eqref{eq:poiseuille:energy} with the inequality
\begin{equation}
    \norm{f}_{L^2(\R)}^2 \leq 2 \norm{Yf}_{L^2(\R)}  \norm{\p_Yf}_{L^2(\R)},
\end{equation}
we obtain that 
\begin{equation}
    \norm{W_*}_{L^2(\R)}^2 \leq 2 \left(\frac{\lambda_* |\alpha_*|}{\epsilon_* b''(0)} \right)^{\frac{1}{2}} \norm{W_*}_{L^2(\R)}^2.
\end{equation}
This is  a contradiction because 
\begin{equation}
\label{viscous:lap:key}
    \left|\frac{\lambda_* |\alpha_*|}{\epsilon_* b''(0)} \right| \leq \frac{M\sigma_0(M)}{|b''(0)|},
\end{equation}
and we can choose $\sigma_0(M)>0$ sufficiently small so that  $M\sigma_0(M) < \frac{|b''(0)|}{4}$.

\end{proof}

\section{Proof of Proposition \ref{prop:k:bounds} - Proposition \ref{prop:higher:reg:2} }
\label{sec:T}

In the section we provide detailed proofs for Proposition \ref{prop:k:bounds} - Proposition \ref{prop:higher:reg:2}. Recall that $\epsilon \in (0, 1/8), j\in\{1,2\}$, $(\alpha, \lambda) \in R_{dj}$ (see \eqref{defdjJ} for the definition of $R_{dj}$), and $k \in \Z \cap [1, \infty)$. Since $j$ and $y_{j\ast}$ are fixed, we often drop the dependence of various quantities on $j$ and assume that $y_{j\ast}=0$, when there is no possibility of confusion. We also allow the implied constants to depend on the indices $\sigma_1,\sigma_2$.



Before proving the bounds on $T_\Theta$, we first prove the following technical lemma which will be useful for quantifying the heuristic that $\ai(f)$ is  exponentially localized to the support of $f$. Recall the definition \eqref{DeC1.92} for $d_{j,k}$, which we denote as $d_k$ below. 
\begin{lemma}
\label{lem:exp:decay}
Assume that  $(\alpha, \lambda) \in (-\sigma_0 \epsilon^{1/2}, 1)  \times \Sigma_{j,\delta_0}$, and $k \in \Z \cap [1, \infty)$. Furthermore, suppose that  $c\in[2,4]$, and that $D \subset \T_\tp$ is an interval. Then for all $\epsilon \in (0, 1/8)$ and  $f \in X^{\sigma_1, \sigma_2}(\mathfrak{M})$  with $\text{supp}(f) \subset D$, 
$  \sigma_1\geq 0,\, \sigma_2\in \R, $ the following bounds hold. 
\begin{itemize}
    \item For
$y \in S_{\frac{\delta(\Lambda)}{c}} ,D \subset \T_\tp\setminus S_{\delta(\Lambda)}$, we have 
\begin{equation}
\label{eq:exp:close}
    \big|\ai(f)(y)\big|\lesssim_{\sigma_1, \sigma_2} 
     \norm{ f}_{X^{\sigma_1, \sigma_2}(\mathfrak M)}
    \frac{  \exp\big(-c_1\delta(\Lambda)^2 \epsilon^{-\frac{1}{2}}\big) }{ d_{ k}^{\sigma_2}   \delta(\Lambda)^{2 + \sigma_1 }};
\end{equation}

  \item For $y \in \T_\tp \setminus S_{c\delta(\Lambda)},  D \subset S_{\delta(\Lambda)}$, we have
\begin{equation}
\label{eq:exp:far}
\begin{split}
     |\ai(f)(y)| &\lesssim  \norm{f}_{L^{\infty}(S_{\delta(\Lambda)}   )}\epsilon^{-1/2} \exp(-c_1|y|\delta(\Lambda)\epsilon^{-\frac{1}{2}}),
     \end{split}
\end{equation}
\end{itemize}
where $c_1 > 0$ is a small constant independent of $(\epsilon, \alpha, k, \lambda)$. 
\end{lemma}
\begin{proof}
We first prove \eqref{eq:exp:close}. 
   For $z \in \text{supp}(f)$ and $y \in S_{\frac{\delta(\Lambda)}{c}}$ we have   $|y - z| = |y - \text{sgn}(y) \delta(\Lambda) | + |\text{sgn}(y)\delta(\Lambda) - z|$. Then, we can write
   \begin{equation}
       \exp(-L(y,z)|y - z|) = \exp(-L(y,z)\left|y - \text{sgn}(y ) \delta(\Lambda) \right|) \times \exp(-L(y,z)\left|z - \text{sgn}(y ) \delta(\Lambda) \right|)    
       \end{equation}
   where $L(y,z)$ represents the coefficient of $|y-z|$ in the exponential in \eqref{eq:kernel:bounds:1} and \eqref{eq:kernel:bounds:2}. For $z \in \T_\tp \setminus S_{\delta(\Lambda)}$ and for  $y_0$ satisfying $b(y_0) = \lambda$, $|y_0| \leq \frac{\delta(\Lambda)}{c}$  
   we have $|b(z) - \lambda| \gtrsim |b(z)|$ and  $L(y,z) \gtrsim \frac{\delta(\Lambda)}{\epsilon^{\frac{1}{2}}} $, with an implicit constants only depending on $b$.  This implies there exists $c_1$, depending only on $b$ such that
\begin{align}
   \label{eq:exp1}
       \exp(-L(y,z)\left|y - \text{sgn}(y ) \delta(\Lambda) \right|) &\lesssim \exp(-c_1 \delta(\Lambda)^2 \epsilon^{-\frac{1}{2}}), \quad y \in S_{\frac{\delta(\Lambda)
       }{c}}, \notag\\
       \exp(-L(y,z)\left|z - \text{sgn}(y ) \delta(\Lambda) \right|) &\lesssim \exp(-c_1 \delta(\Lambda) \epsilon^{-\frac{1}{2}} \left|z - \text{sgn}(y ) \delta(\Lambda) \right|), \quad z \in \T_\tp \setminus S_\delta(\Lambda),
   \end{align} 
   and consequently,
   \begin{align}
   \label{eq:exp2}
       &|\ai (f) (y)| \lesssim \exp(-c_1 \delta(\Lambda)^2 \epsilon^{-\frac{1}{2}}) \int |f(z)| \epsilon^{-\frac{1}{2}} \delta(\Lambda)^{-1} \exp(-c_1 \delta(\Lambda) \epsilon^{-\frac{1}{2}} \left|z - \text{sgn}(y ) \delta(\Lambda) \right|) \, dz, \notag\\
       &\quad y \in \T_{\tp}.
   \end{align}
The bound \eqref{eq:exp:close} will follow from \eqref{eq:exp2} if
\begin{equation}
\label{exp3}
    \int_{\T_\tp} \frac{ \epsilon^{-\frac{1}{2}} \delta(\Lambda)}{ \varrho^{\sigma_1} (z; \Lambda) \varrho_k^{\sigma_2} (z; \Lambda) } \exp(-c_1 \delta(\Lambda) \epsilon^{-\frac{1}{2}} \left|z - \text{sgn}(y) \delta(\Lambda) \right|) \, dz \lesssim 
    \frac{1}{\delta(\Lambda)^{\sigma_1} d_k^{\sigma_2} } .
\end{equation}

To prove \eqref{exp3} we consider several cases. First, assume that $k\delta \geq 1$. Then $\varrho_k^{\sigma_2}(z; \Lambda) = \frac{1}{k^{\sigma_2}}$. Therefore, it is enough to only consider $\varrho^{\sigma_1}(z, \Lambda)$. Using that  $y \in S_{\frac{\delta(\Lambda)}{c}}$  and $\varrho(z, \Lambda) \geq \delta(\Lambda), $ \eqref{exp3} follows. Similar considerations hold for $k \delta(\Lambda) \leq 1$ and $ \sigma_2 \geq 0$. Now,  consider the case where $k\delta(\Lambda) \leq 1$  and  $\sigma_2 < 0$.  In this situation 
\begin{equation}
    \frac{1}{ \varrho^{\sigma_1}(y; \Lambda) \varrho_k^{\sigma_2}(y, \Lambda) } \lesssim \frac{\varrho^{|\sigma_2|} (y; \Lambda) }{\varrho^{\sigma_1}(y; \Lambda)} \lesssim \frac{\varrho^{|\sigma_2|}(y; \Lambda)} {\delta(\Lambda)^{\sigma_1}}.
\end{equation}
Therefore, it's sufficient to prove that 
\begin{equation}
\label{exp4}
    \int_{\T_\tp}  \epsilon^{-\frac{1}{2}} \delta(\Lambda) \varrho_k^{|\sigma_2|} (z; \Lambda)  \exp(-c_1 \delta(\Lambda) \epsilon^{-\frac{1}{2}} \left|z - \text{sgn}(y ) \delta(\Lambda) \right|) \, dz \lesssim \delta(\Lambda)^{|\sigma_2|}.
\end{equation}
Since $c_1 \delta(\Lambda) \epsilon^{-\frac{1}{2}} \exp (-c_1 \delta(\Lambda) \epsilon^{-\frac{1}{2}})$ is an approximation of the identity and $\frac{\epsilon^{\frac{1}{2}}}{\delta(\Lambda)} \leq  \delta(\Lambda) $, \eqref{exp4} follows, completing the proof of \eqref{eq:exp:close}. 

For the proof of \eqref{eq:exp:far} we proceed as before except with the roles of $y$ and $z$ reversed from \eqref{eq:exp1}. This yields
 \begin{align}
     \label{eq:exp4}
      \exp(-L(y,z)\left|y - \text{sgn}(y) \delta(\Lambda) \right|) &\lesssim \exp(-c_1  \delta(\Lambda) |y |\epsilon^{-\frac{1}{2}}), \quad y \in \T_\tp \setminus S_{c \delta(\Lambda)}, \notag\\
       \exp(-L(y,z)\left|z - \text{sgn}(y) \delta(\Lambda) \right|) &\lesssim \exp(-c_1 \delta(\Lambda) \epsilon^{-\frac{1}{2}} \left|z - \text{sgn}(y ) \delta(\Lambda) \right|),\quad z \in S_{\delta(\Lambda)}.
 \end{align}
 Then, for $y \in \T_{\tp} \setminus S_{c\delta(\Lambda)}$
  \begin{align}
   \label{eq:exp5}
       |\ai (f)(y)|\lesssim  &\exp(-c_1  \delta(\Lambda) |y |\epsilon^{-\frac{1}{2}})  \norm{f}_{L^{\infty}( S_{\delta(\Lambda)} )} \epsilon^{-\frac{2}{3}} \min( \delta(\Lambda), \epsilon^{\frac{1}{4}})^{-\frac{1}{3}}\\  
       & \times \int_{\T_\tp}   \exp(-c_1 \delta(\Lambda) \epsilon^{-\frac{1}{2}} \left|z - \text{sgn}(y) \delta(\Lambda) \right|) \, dz ,
   \end{align}
   which implies the desired bounds.
\end{proof}

We will prove of Propositions \ref{prop:k:bounds} -  \ref{prop:higher:reg:2} by proving the needed bounds for each $\tc{I}{1}, \tc{I}{2}, \tc{v}{1}$, and $\tc{v}{2}$ individually. Each of the following subsections is dedicated to one of the operators and proves the bounds needed for the proof of the propositions.

\subsection{Bounds on $\tc{I}{1}$ }
Recall that the definition of $\tc{I}{1}$ with $y_{1*} = 0$ is
\begin{equation}
 \tc{I}{1}g(y)  =  \int_{\T_\tp} \gr \ai ((1- \varphi_{\delta_0}) b'' g )(z) \, dz,\quad y\in\T_\tp.
\end{equation}
Since the $(1- \varphi_{\delta_0}) b'' g $ is supported approximately $\delta_0$ away from the critical layer 
we can essentially view $\ai$ as a bounded multiplication operator. However, in order to avoid the logarithmic loss in $\delta(\Lambda)$ for $\gr$ applied to bounded functions, we use Lemma \ref{lem:exp:decay} to view $\ai((1 - \varphi_{\delta_0}) b''g)$ as if it was supported $\delta_0$ distance away from the origin.  The following Lemma implies Propositions \ref{prop:higher:regularity1}
 and \ref{prop:higher:reg:2} for $\tc{I}{1}$.
\begin{lemma}
\label{lem:ti1}
  Assume that $(\alpha, \lambda) \in (-\sigma_0 \epsilon^{\frac{1}{2}}, 1)  \times \Sigma_{j,\delta_0}$, and  $k \in \Z \cap [1, \infty)$. Then for all $\epsilon \in (0, 1/8)$ and $g \in X^{\sigma_1, \sigma_2}(\mathfrak M)$  with $\sigma_1 ,\sigma_2  \in \R$ we have the following  bounds for all $y \in \T_\tp$ and $\beta \in \{ 0,1,2 \}$,
   \begin{align}
   \label{eq:ti1:pw}
         \big| \varrho_k^\beta(y ; \Lambda) \p_y^\beta \tc{I}{1}g(y)\big| \lesssim_{\sigma_1, \sigma_2} \norm{g}_{X^{\sigma_1, \sigma_2} 
 (\mathfrak M)}\varrho_k^{2}(y ; \Lambda) k^{\sigma_2}\times
         \begin{cases}
             k^{- 1}, & {\rm if}\,\, \varrho_k(y ; \Lambda) = \varrho(y, \Lambda ), \\
            1, & {\rm if}\,\, \varrho_k(y ; \Lambda) = \frac{1}{k} .
         \end{cases}
        \end{align}
    \end{lemma}
\begin{proof}
First, suppose that $\varrho_k(y ; \Lambda) = \varrho(y ; \Lambda)$. 
Decompose $\tc{I}{1}g$ as 
    \begin{align}
     \int_{\T_\tp} \gr \ai ((1- \varphi_{\delta_0}) b'' g )(z) \, dz &=
           \int_{\T_\tp} \gr \varphi_{\frac{\delta_0}{3}}(z)\ai ((1- \varphi_{\delta_0}) b'' g )(z) \, dz \\\notag 
          &+ \int_{\T_\tp} \gr (1-\varphi_{\frac{\delta_0}{3}}(z))\ai ((1- \varphi_{\delta_0}) b'' g )(z) \, dz.
    \end{align}
Proposition \ref{Airy_main} and  \eqref{eq:exp:close} with $\delta = \delta_0$ imply
\begin{align}
\label{ti1:1}
|(1-\varphi_{\frac{\delta_0}{3}}(z))\ai ((1- \varphi_{\delta_0}) b'' g )(z)| &\lesssim_{\delta_0, \sigma_1, \sigma_2}   k^{\sigma_2}\norm{g}_{X^{\sigma_1, \sigma_2}(\mathfrak M)},
\notag\\
|\varphi_{\frac{\delta_0}{3}}(z)\ai ((1- \varphi_{\delta_0}) b'' g )(z)| &\lesssim_{\delta_0, \sigma_1, \sigma_2} \exp(-c_1 \delta_0^2 \epsilon^{-\frac{1}{2}}) k^{\sigma_2}\norm{g}_{X^{\sigma_1, \sigma_2}(\mathfrak M)}. 
\end{align}
 Lemma \ref{mGk50} combined with  \eqref{ti1:1} for $\beta \in \{ 0,1 \}$ yields
\begin{align}
\label{tbpw:1}
 &\left| \int_{\T_\tp} \varrho^\beta(y ; \Lambda)\p_y^\beta\gr \varphi_{\frac{\delta_0}{3}}(z)\ai ((1- \varphi_{\delta_0}) b'' g )(z) \, dz  \right|\notag \\
 &\lesssim_{\delta_0, \sigma_1, \sigma_2}  
 \norm{g}_{X^{\sigma_1, \sigma_2}(\mathfrak M)}\varrho^2(y ; \Lambda) \exp(-c_1 \delta_0^2 \epsilon^{-\frac{1}{2}}) \int_{\T_\tp}  \frac{k^{\sigma_2}\varrho_k(z, \Lambda)}{\varrho^2(z, \Lambda)}  \, dz \notag\\
     &\lesssim_{\delta_0, \sigma_1, \sigma_2}  \norm{g}_{X^{\sigma_1, \sigma_2}(\mathfrak M)} \varrho^2(y ; \Lambda) k^{\sigma_2 -1 } \delta(\Lambda)^{-1}\exp(-c_1 \delta_0^2 \epsilon^{-\frac{1}{2}}). 
\end{align}
Similarly, for $\beta \in \{0, 1 \}$ we obtain the bound
\begin{align}
\label{tbpw:2}
    \left| \int_{\T_\tp} \varrho^\beta(y ; \Lambda)\p_y^\beta\gr (1-\varphi_{\frac{\delta_0}{3}}(z))\ai ((1- \varphi_{\delta_0}) b'' g )(z) \, dz \right| \lesssim_{\delta_0, \sigma_1, \sigma_2} \frac{\varrho^{2}(y ; \Lambda)}{ k^{1 - \sigma_2  }}\norm{g}_{X^{\sigma_1, \sigma_2}(\mathfrak M)}. 
\end{align}
Since $\delta(\Lambda) \gtrsim \epsilon^{\frac{1}{4}}$ it follows
\begin{equation}
   \delta(\Lambda)^{-1} \exp(-c_1 \delta_0^2 \epsilon^{-\frac{1}{2}}) \lesssim 1,
\end{equation}
giving \eqref{eq:ti1:pw} for $\varrho_k(y ; \Lambda) = \varrho(y ; \Lambda)$ and $\beta \in \{ 0, 1 \}$. 
For $ \varrho_k(y ; \Lambda)  = \frac{1}{k}$, we bound $\ai ((1- \varphi_{\delta_0}) b'' g )$ as 
\begin{equation}
\label{ti1:bs}
    |\ai ((1- \varphi_{\delta_0}) b'' g )(z)| \lesssim_{\delta_0, \sigma_1, \sigma_2} k^{\sigma_2} \norm{g}_{X^{\sigma_1, \sigma_2}(\mathfrak M)}. 
\end{equation}
Combining \eqref{ti1:bs} with Lemma \ref{mGk50} then yields \eqref{eq:ti1:pw} for $\beta \in \{0,1 \}$. 
For $\beta = 2$, we have  
\begin{align}
    &\int_{\T_\tp} \p_y^2 \gr \ai ((1- \varphi_{\delta_0}) b'' g )(z) \, dz   \notag\\
    &= - \ai ((1- \varphi_{\delta_0}) b'' g )(y) + (V(y) + k^2)  \int_{\T_\tp}  \gr \ai ((1- \varphi_{\delta_0}) b'' g )(z) \, dz.
\end{align}
By \eqref{ti1:bs}, 
the rest of the proof proceeds as for $\beta = 0$.
\end{proof}
The following corollary implies Proposition \ref{prop:k:bounds} for $\tc{I}{1}$.
\begin{corollary}
\label{lem:ti1:X}
Assume that  $(\alpha, \lambda) \in (-\sigma_0 \epsilon^{\frac{1}{2}}, 1)  \times \Sigma_{j,\delta_0}$,  and $k \in \Z \cap [1, \infty)$. Then for all $\epsilon \in (0, 1/8)$  and $\sigma_1 \in [0,1]$, $\sigma_2 \in  [-2, 1/2]$, we have the following bounds for $g \in X^{\sigma_1, \sigma_2}(\mathfrak M)$,
     \begin{equation}\label{eq:ti1:X}
       \norm{\tc{I}{1}g}_{X^{\sigma_1, \sigma_2}(\mathfrak M)} \lesssim  \frac{\norm{g}_{X^{\sigma_1, \sigma_2}(\mathfrak M)}}{  k^{\frac{1}{2}} }. 
   \end{equation}
\end{corollary}     
\begin{proof}
Recalling the definition of $X^{\sigma_1, \sigma_2}(\mathfrak M)$, it is sufficient to prove that for all $y \in \T_\tp$ and $\beta \in \{0,1 \}$ we have
\begin{align}
       \varrho^{\sigma_1}(y ; \Lambda) \varrho_k^{\sigma_2}(y ; \Lambda) |\varrho_k^\beta(y ; \Lambda) \p_y^\beta \tc{I}{1}g(y) | \lesssim   \frac{\norm{g}_{X^{\sigma_1, \sigma_2}(\mathfrak M)} }{k^{\frac{1}{2}}}.
\end{align}
Since $\T_\tp$ is bounded and $\sigma_1 \geq 0$ it is enough to prove 
\begin{align}
        \varrho_k^{\sigma_2}(y ; \Lambda) |\varrho_k^\beta(y ; \Lambda) \p_y^\beta \tc{I}{1}g(y) | \lesssim   \frac{\norm{g}_{X^{\sigma_1, \sigma_2}(\mathfrak M)} }{k^{\frac{1}{2}}}. 
\end{align}
We first consider the case $\varrho_k(y ; \Lambda) = \varrho(y ; \Lambda)$. From \eqref{eq:ti1:pw} it follows that 
    \begin{align}
        \varrho^{\sigma_1}(y ; \Lambda) |\varrho^\beta(y ; \Lambda) \p_y^\beta \tc{I}{1}g(y) | \lesssim \varrho^{\sigma_1 +2} (y ; \Lambda) k^{\sigma_2 - 1}  \norm{g}_{X^{\sigma_1, \sigma_2}(\mathfrak M)} \lesssim  \frac{\norm{g}_{X^{\sigma_1, \sigma_2} (\mathfrak M)}}{k^{\frac{1}{2}}},
    \end{align}
    where we used that $\sigma_2  \in [-2, 1/2]$. Now assume that $\varrho_k(y ; \Lambda) = \frac{1}{k}$; \eqref{eq:ti1:pw} yields
     \begin{align}
        k^{-\sigma_2} |\varrho_k
        ^\beta(y ; \Lambda) \p_y^\beta \tc{I}{1}g(y) | \lesssim 
\frac{\norm{g}_{X^{\sigma_1, \sigma_2}(\mathfrak M)}}{k^2},
        \end{align}
        which completes the proof.
\end{proof}


\subsection{Bounds on $\tc{I}{2}$ }
Recall that the definition of $\tc{I}{2}$ with $y_{1*} = 0$ is given for $y\in\T_\tp$ by 
\begin{equation}
    \tc{I}{2} g(y) =   \int_{\T_\tp} \gr  \ai ( \varphi_{\delta(\Lambda)} b'' g)(z) \, dz.
\end{equation}
Unlike the other terms in the decomposition of $T_{\Theta}$,  the kernel 
of $\tc{I}{2}$ exhibits singular behavior due to the critical layer. This requires a more careful analysis near the critical layer than the pointwise estimates given by the fundamental solution bounds (see Lemma \ref{lem:Airybounds:intermediate}). However, the Airy fundamental solution bounds are still useful, even for $|y| \lesssim  \delta(\Lambda)$, {\it provided} that we have the bound
\begin{equation}
\label{eq:airy:delta}
|\ai(\varphi_{\delta(\Lambda)} b'' g)(y)| \lesssim \frac{\norm{g}_{L^{\infty}(S_{\delta})}}{\varrho^2(y ; \Lambda)}, 
\end{equation}
with the implicit constant independent of $\alpha, \epsilon, k, \lambda$. Comparing this with the bounds in Proposition \ref{Airy_main}, we determine that \eqref{eq:airy:delta} will happen in the following three situations:
\begin{subequations}
\label{reg:degen}
\begin{align}
    \delta(\Lambda) &\lesssim |\alpha|^{\frac{1}{2}}, \\
    \delta(\Lambda) &\lesssim \epsilon^{\frac{1}{4}}, \\
  \delta(\Lambda) &\lesssim |b(y) - \lambda|^{1/2}\,\,{\rm for}\,\,y\in\T_\tp,\quad{\rm and }\,\,\lambda \notin b(\T_\tp), 
\end{align}
\end{subequations}
with implicit constants independent of $\epsilon, \alpha, k,  \lambda$. We refer to the situations in \eqref{reg:degen} as  \emph{regularly degenerate}. In these situations, we have the following lemma.
\begin{lemma}
\label{lem:ti2:regularized}
Assume that $(\alpha, \lambda) \in (-\sigma_0 \epsilon^{\frac{1}{2}}, 1)  \times \Sigma_{j,\delta_0}$, and  $k \in \Z \cap [1, \infty)$. Furthermore, assume that at least one of the cases in $\eqref{reg:degen}$ holds. Then for all $\epsilon \in (0, 1/8)$,   $g \in L^{\infty}(S_{\delta})$ with ${\rm supp}\,g\subseteq S_{\delta}$, and  all $y \in \T_\tp$ the following bounds hold,
\begin{itemize}
\item If $k \delta(\Lambda) \leq 1$, then we have 
\begin{equation}
\label{eq:viscous:regime:delta}
          \varrho_k^\beta(y ; \Lambda)|\p_y^{\beta}\tc{I}{2}g(y) | \lesssim  \frac{\delta(\Lambda) \norm{g}_{L^{\infty}(\sd{})}}{\varrho(y;\Lambda)}\left[\frac{\varrho_k(y;\Lambda)}{\varrho(y;\Lambda)}\right]^{1/4}, \quad \beta \in \{ 0,1 ,2\};
\end{equation}

\item If $k\delta(\Lambda) \geq 1$, then
   \begin{align}
     \label{eq:viscous:regime:k}
        |\tc{I}{2}g(y) | +  \Big|\frac{1}{k}\p_y\tc{I}{2}g(y) \Big| \lesssim \frac{\norm{g}_{L^{\infty}(S_{\delta})}}{k^2 \varrho^2(y;\Lambda)}.
    \end{align}
    
   \end{itemize}
\end{lemma}

\begin{proof}
We provide the details only for \eqref{eq:viscous:regime:delta} since  the proof of \eqref{eq:viscous:regime:k} is simpler. We decompose $\tc{I}{2}$ as 
\begin{equation}
      \int_{\T_\tp} \gr \varphi_{3 \delta(\Lambda)}(z)  \ai ( \varphi_{\delta(\Lambda)} b'' g)(z) \, dz  + \int_{\T_\tp} \gr 
 (1- \varphi_{3 \delta(\Lambda)}(z) ) \ai ( \varphi_{\delta(\Lambda)} b'' g)(z) \, dz.
\end{equation}
By Lemma \ref{mGk50}  it follows that for $\beta \in  \{0,1\}$
\begin{align}
\label{eq:ti:pw:far}
&\Big| \varrho_k^{\beta}(y ; \Lambda)\int_{\T_\tp}  \p_y^{\beta}\gr 
 (1- \varphi_{3 \delta(\Lambda)}(z) ) \ai ( \varphi_{\delta(\Lambda)} b'' g)(z) \, dz  \Big|  \notag\\
 &\lesssim \big(\delta\wedge \frac{1}{k}\big)\min\Big\{\frac{\delta}{\varrho(y;\Lambda)},\,e^{-k\varrho(y;\Lambda)}\Big\}  \int_{\T_\tp}  (1- \varphi_{3 \delta(\Lambda)}(z) ) |\ai ( \varphi_{\delta(\Lambda)} b'' g)(z)| \, dz  .
\end{align}
On the support of $ (1- \varphi_{3 \delta(\Lambda)}(z) )$,  \eqref{eq:exp:far} implies
\begin{equation}\label{APIJ1}
\begin{split}
    &|\ai(\varphi_\delta(\Lambda) b''g)(z)|\\
    &\lesssim  \norm{g}_{L^{\infty}(S_{\delta(\Lambda)}   )} \epsilon^{-1/2} \exp(-\frac{1}{2}c_1|z |\delta(\Lambda)\epsilon^{-\frac{1}{2}}) \exp(-\frac{c_1}{2} \frac{\delta(\Lambda)^2}{\epsilon^{\frac{1}{2}}}),
    \end{split}
\end{equation}
where $c_1$ is as in Lemma \ref{lem:exp:decay}. Since  
\begin{align}\label{APIJ2}
    \int_{\T_\tp} (1- \varphi_{3 \delta(\Lambda)}(z) ) \delta(\Lambda)\,\epsilon^{-1/2} \exp(-\frac{c_1}{2}|z|\delta(\Lambda)\epsilon^{-\frac{1}{2}}) \,dz \lesssim 1,  
\end{align}
The desired bound \eqref{eq:viscous:regime:delta} for $\beta \in \{ 0,1 \}$ follows from \eqref{APIJ1}-\eqref{APIJ2} and \eqref{eq:ti:pw:far}.

For the second derivatives, \eqref{mGk1} yields
\begin{align}
\label{eq:second:derivative}
     &\p_y^2 \int_{\T_\tp} \gr  \ai ( \varphi_{\delta(\Lambda))} b'' g)(z) \, dz \notag\\
    & = -  \ai ( \varphi_{\delta(\Lambda)} b'' g)(y) + (V(y)+ k^2) \int_{\T_\tp} \gr \ai ( \varphi_{\delta(\Lambda)} b'' g) (z)\, dz .
\end{align}
The desired bound \eqref{eq:viscous:regime:delta} for $\beta=2$ then follows from \eqref{eq:ti:pw:far} and \eqref{APIJ1}-\eqref{APIJ2}.
For the proof of \eqref{eq:viscous:regime:k}, we do not decompose $\tc{I}{2}$: We directly use \eqref{eq:airy:delta}, Lemma \ref{mGk50}, and that $ke^{-k|z|}$ is an approximation of the identity.
\end{proof}
Due to Lemma \ref{lem:ti2:regularized}, we can assume the following conditions hold for the remaining bounds on $\tc{I}{2}$ :
\begin{subequations}
\label{sing:degen}
    \begin{align}
    & |\lambda| \gg \epsilon^{\frac{1}{2}} 
     \text{ with an implicit constant independent of } \epsilon, \alpha, k, \lambda, \\
    & \lambda \in b(\T_\tp),  \\
    & \delta(\Lambda) \approx \delta_1(\lambda),
    \end{align}
\end{subequations}
where we recall that $\delta_1(\lambda) = 8 \sqrt{|\lambda| / | b''(0)|}$.
We seek to isolate the part of $\ai(\varphi_{\delta(\Lambda)} b'' g)$ that is supported at a distance closer to the critical layer than the length scale $d_k$ of the modified Green's function. On the complement of this region, the smallness from the Green's function can balance the size of $\ai(\varphi_{\delta(\Lambda)} b'' g)$ exactly as in the ``regularly degenerate" situation covered in Lemma \ref{lem:ti2:regularized}.  

Following  \eqref{Tnd:decomp},  
 we decompose $\tc{I}{2}$ as 
 \begin{align}
    \tc{I}{2}g &= \sum_{\ell=0}^2 \tc{I}{2\ell}g .
    \end{align}
We define $\tc{I}{20}$ and  $\tc{I}{2\ell}, \ell\in\{1,2\}$  as 
\begin{align}
    \tc{I}{20}g(y) := \int_{\T_\tp} \gr  \ai(\varphi_{\delta(\Lambda)} b'' g)(z) \Big(1  - \sum_{\ell=1}^2 \varphi_{\theta d_k}^2(z - y_\ell) \Big) \, dz , 
\end{align}
and
 \begin{align}
         \tc{I}{2\ell}g(y) := \int_{\T_\tp} \gr \varphi_{\theta d_k}^2(z - y_\ell) \ai(\varphi_{\delta(\Lambda)} b'' g)(z)\, dz,
    \end{align}
where $\ell \in \{1, 2 \}$,  $b(y_\ell) = \lambda$ 
and $\theta \in (0, 1)$ is a constant only depending on $b(y)$ (see Lemma \ref{lem:Airybounds:intermediate}).
Since $\delta(\Lambda) = \delta_1(\lambda)$, for $\theta$ sufficiently small, $\varphi_{\theta d_k}(y-y_1)$ and $\varphi_{\theta d_k} (y-y_2)$ have disjoint supports. Using a similar proof as for Lemma \ref{lem:ti2:regularized}, we obtain the following bounds for $\tc{I}{20}$.
\begin{lemma}
    \label{lem:tio:pw}
Assume that, $(\alpha, \lambda) \in (-\sigma_0 \epsilon^{\frac{1}{2}}, 1)  \times \Sigma_{j,\delta_0}$,   $k \in \Z \cap [1, \infty)$, and let $\theta$ to be fixed in the range given by Lemma \ref{lem:Airybounds:intermediate}. Then for all $\epsilon \in (0, 1/8)$ and  $g \in L^{\infty}(S_{\delta(\Lambda)})$
    we have the following bounds for all $y \in \T_{\tp}$.
\begin{itemize}
    \item If $k \delta(\Lambda) \leq 1$, then
    \begin{align} \varrho_k^\beta(y ; \Lambda)|\p_y^\beta\tc{I}{20}g(y)|  \lesssim \frac{\delta(\Lambda) \norm{g}_{L^{\infty}(S_{\delta(\Lambda)}) }}{\varrho(y ; \Lambda)}\left[\frac{\varrho_k(y;\Lambda)}{\varrho(y;\Lambda)}\right]^{1/4}, \quad{\rm for}\,\, \beta \in \{0, 1, 2 \},
\end{align}

\item If $k \delta(\Lambda) \geq 1$, then
\begin{align}
      \Big|\frac{1}{k^\beta}\p_y^\beta\tc{I}{20}g(y) \Big| \lesssim \frac{\norm{g}_{L^{\infty}(S_{\delta(\Lambda)})}}{k  \varrho(y ; \Lambda)}, \quad{\rm for}\,\, \beta \in \{ 0,1 \}. 
\end{align}
\end{itemize}

\end{lemma}
\begin{proof}
    We only consider the case where $k \delta(\Lambda) \geq 1$ since the case for $k\delta(\Lambda) \leq 1$ is identical to Lemma \ref{lem:ti2:regularized}. 
We have the following:
\begin{align}
    |\ai(\varphi_{\delta(\Lambda)} b'' g)(z) (1  - \sum_{j=1}^2 \varphi_{\theta d_k}^2(z - y_\ell) )| \lesssim 
    \begin{cases}
        \frac{k\norm{g}_{L^{\infty}(S_\delta(\Lambda))}}{\delta(\Lambda)} & |z| \leq \delta(\Lambda) \\
        \frac{\norm{g}_{L^{\infty}(S_\delta(\Lambda))}}{\varrho^2(z, \Lambda)} & |z| \geq \delta(\Lambda)
    \end{cases}
\end{align}
Combining these estimates yields 
\begin{equation}
     \Big|\ai(\varphi_{\delta(\Lambda)} b'' g)(z) (1  - \sum_{j=1}^2 \varphi_{\theta d_k}^2(z - y_\ell) ) \Big| \lesssim \frac{k \norm{g}_{L^{\infty}(S_\delta(\Lambda))}}{\varrho(z, \Lambda)}.
\end{equation}
Therefore, by Lemma \ref{mGk50} we have
\begin{equation}
    \Big|\frac{1}{k^\beta}\p_y^\beta\tc{I}{20}g(y) \Big| \lesssim \frac{\norm{g}_{L^{\infty}(S_{\delta(\Lambda)})}}{k  \varrho(y ; \Lambda)}, \quad{\rm for}\,\, \beta \in \{ 0,1 \}, 
\end{equation}
which proves the result.
\end{proof}
For the region that lies in the support of   $\varphi_{\theta d_k}(y - y_\ell)$ with $b(y_\ell) = \lambda$, the size  of $\ai(\varphi_{\delta(\Lambda) }b'' g)$ is too large to be controlled by the pointwise bounds for the modified Green's function directly. As in the invisicd case we expect that integrating by parts to ``remove" a derivative from $\ai$ and placing it on $\gr$ would be beneficial to gain from the larger scale of variation of $\gr$ (see \cite{Iyer}, \cite{JiaL}). The following lemma makes this precise using a frequency based characterization of the inverse of the generalized Airy  operator.  
\begin{lemma}[]
\label{lem:Airybounds:intermediate}
Assume that $(\alpha, \lambda) \in (-\sigma_0 \epsilon^{\frac{1}{2}}, 1)  \times \Sigma_{j,\delta_0}$,  and  $k \in \Z \cap [1, \infty)$. Furthermore, assume that $\delta(\Lambda) \approx \delta_1(\lambda)$, $\lambda \in b(\T_\tp)$, and $ \epsilon^{\frac{1}{2}} \ll |\lambda| \ll 1$ with implicit constants independent of $(\epsilon, \alpha, k, \lambda)$. Then there exists $\theta_0 \in (0, 1/2)$, depending only on $b(y)$, such that for all $\theta \in (0, \theta_0]$ and $\epsilon \in (0, 1/8)$ the following statement holds.
Suppose that $f: \T_\tp \to \C$ satisfies 
\begin{equation}
\label{eq:f:H1:regularity}
    \norm{f}_{L^2(\sd{})} + d_k \norm{\p_y f}_{L^2(\sd{})} \leq M,
\end{equation}
and is 0 outside of $S_{ \delta(\Lambda)}$. We view the function
\begin{equation}
    \varphi_{ \theta d_k}( y - y_{\ell}) \ai( f)(y),
\end{equation}
 as a function on $\R$.
Then
\begin{equation}
\label{eq:Airy:freq:pw}
\big|\sup_{\xi \in \R }\mathcal{F}\left(\varphi_{ \theta d_k}( \cdot - y_{\ell}) \ai( f) \right)(\xi)\big| \lesssim 
\frac{M}{d_k^{\frac{1}{2}}\delta(\Lambda)},   
\end{equation}
and 
\begin{equation}
\label{eq:Airy:L2}
\norm{ (1 + d_k \p_y)^{-1} \left(\varphi_{ \theta d_k} (\cdot - y_{\ell})\ai(f) \right)}_{L^2(\R)}
\lesssim \frac{M}{d_k\delta(\Lambda)}, 
\end{equation}
where $y_{\ell}$ satisfies $b(y_{\ell}) = \lambda$.
\end{lemma}
The following inequality will be useful in the  proof of Lemma \ref{lem:Airybounds:intermediate}:
\begin{equation}
    \label{eq:H1:bdd}
    \norm{f}_{L^{\infty}(I)} \lesssim |I|^{-\frac{1}{2}}\norm{f}_{L^2(I)} + |I|^{\frac{1}{2}}\norm{\p_y f }_{L^2(I)}
\end{equation}
where $I$ is an interval  and $|\cdot|$ denotes the length of the interval. 

\begin{proof}
Assuming \eqref{eq:Airy:freq:pw}, by Plancherel's theorem we have 
    \begin{align*}
      &\norm{ (1 + d_k \p_y)^{-1} \left(\varphi_{\theta d_k}(\cdot - y_{\ell})\ai( f) \right)}_{L^2(\R)} = \norm{ (1 - id_k \xi )^{-1} \mathcal{F}(\varphi_{\theta d_k}(\cdot - y_{\ell}) \ai(f))(\xi) }_{L_{\xi}^2(\R)} \notag \\
      &\lesssim  \frac{M}{d_k^{\frac{1}{2}}\delta(\Lambda)} \norm{(1 - id_k \xi )^{-1}}_{L_{\xi}^2(\R)}\lesssim \frac{M}{d_k\delta(\Lambda)^{\frac{1}{2}}}. 
    \end{align*}
Therefore, it suffices to prove $\eqref{eq:Airy:freq:pw}$. We proceed as in Lemma \ref{lem:Airybounds}. For non-critical values of $\lambda$, $b^{-1}(\lambda)$ has cardinality 2: $b^{-1}(\lambda)= \{y_1, y_2 \}$. Define $\varphi_\ell(y) := \varphi\left( \frac{2(y- y_{\ell})}{\theta d_k} \right)$ for $\ell \in \{ 1,2\} $. Choose $\theta_0>0$ so that the supports of $\varphi_1$ and $\varphi_2$ are separated by distance  $\theta_0 d_k$. Furthermore, define $w_{c,\ell}$ as the solution to the equation for $y\in\R$, 
\begin{equation}
    \epsilon \p_y^2 w_{c, \ell}(y) -\alpha w_{c, \ell}(y)   + ib'(y_{\ell})(y_{\ell} - y)w_{c , \ell}(y)  = f(y_{\ell}).
\end{equation}
We can solve for  $w_{c,\ell}$ explicitly as
\begin{equation}
    w_{c,\ell}(y,y_{\ell}) = \frac{\ f(y_{\ell})  }{b'(y_{\ell})} \Big(\frac{b'(y_{\ell})}{\epsilon} \Big)^{\frac{1}{3}}
    W \Big(\Big( \frac{ b'(y_{\ell})}{\epsilon} \Big)^{\frac{1}{3}}   (y-y_{\ell} )\Big),\quad y\in\R,
\end{equation}
where $W$ is given in the Fourier variable $\xi\in\R$ as
\begin{equation}\label{exW1}
    \widehat{W}(\xi) = -\sqrt{2\pi} \exp \left( \frac{\xi^3}{3} + \frac{\ept \alpha \xi }{(b'(y_{\ell}))^{\frac{2}{3}}}  \right) \mathbbm{1}_{(-\infty, 0]}(\xi). 
\end{equation}
We define $w_R: \T_\tp \to \C$ as 
\begin{align}
    w_R := \ai(f) - \sum_{\ell \in \{ 1,2\} }\varphi_{\ell} w_{c,\ell},  
\end{align}
where we have view $\varphi_{\ell} w_{c,\ell}$ as a function on $\T_\tp$. $w_R$ can be bounded similarly as in Lemma \ref{ake6} (see \eqref{ake6.5}), and is more regular than $W$.
Therefore, we obtain from \eqref{exW1} that
\begin{equation}
   \big|\mathcal{F}(\varphi_{ \theta d_k} (\cdot - y_{\ell}) \ai(f)) (\xi)\big|\lesssim \frac{M}{d_k^{\frac{1}{2}}\delta(\Lambda)}  \quad j \in \{ 1, 2\}
\end{equation}
which completes the proof.
\end{proof}
For the remainder of the section on $\tc{I}{2}$ we will take $\theta$ to be fixed.  Before proceeding with the proof of the bounds for $\tc{I2}{j}$, it will be convenient to introduce the following notation. 
For $\ell \in \{1 ,2 \}$ define for $y\in\R$,
\begin{equation}
    \Xi_{\ell}(y) := (1 + d_k \p_y)^{-1}\left[\varphi_{\theta d_k}(\cdot - y_{\ell}) \ai(\varphi_{\delta(\Lambda)} b'' g) \right].
\end{equation}
where $b(y_{\ell}) = \lambda$ and $\varphi_{\theta d_k}(y - y_{\ell})$ is as in Lemma \ref{lem:Airybounds:intermediate}.  We now proceed with obtaining the remaining bounds for $\tc{I2}{\ell}$.
\begin{lemma}
\label{lem:tI2j:pw}
Assume that $\ell \in \{1, 2 \}$, $(\alpha, \lambda) \in (-\sigma_0 \epsilon^{\frac{1}{2}}, 1)  \times \Sigma_{j,\delta_0}$,  and  $k \in \Z \cap [1, \infty)$. Furthermore, assume that $\delta(\Lambda) = \delta_1(\lambda)$, $\lambda \in b(\T_\tp)$, and $ \epsilon^{\frac{1}{2}} \ll |\lambda| \ll 1$ with implicit constants independent of $\epsilon, \alpha, k, \lambda$. Take $\theta$ to be fixed in the range given by Lemma \ref{lem:Airybounds:intermediate} 
Then for all $\epsilon \in (0, 1/8)$ and $g \in L^{\infty}(S_{\delta(\Lambda)})$the following bounds hold  for all $y \in \T_\tp$,
\begin{itemize}
    \item If $k \delta(\Lambda) \leq 1$, then
\begin{equation}\label{eq:tij:second:derivatitve}
\begin{split}
  &|\tc{I}{2\ell} g(y)| + \varrho_k(y ; \Lambda)\big|\p_y \tc{I}{2\ell} g(y) + d_k \varphi_{\theta d_k}(y - y_{\ell}) \Xi_{\ell}(y)\big| \\
  &\lesssim   \left[\frac{\varrho_k(y;\Lambda)}{\varrho(y;\Lambda)}\right]^{1/4}\frac{\delta(\Lambda) }{\varrho(y ; \Lambda)} \norm{g}_{L^{\infty}(\sd{}) },\\
& \varrho_k^2(y, \Lambda )\big| \p_y\big(\p_y \tc{I}{2j} g(y) + d_k \varphi_{\theta d_k} (y -y_{\ell})\Xi_{\ell}(y)\big) +  \varphi_{\theta d_k} (y-y_{\ell})\Xi_{\ell}(y)    \big| \\
  & \lesssim  \left[\frac{\varrho_k(y;\Lambda)}{\varrho(y;\Lambda)}\right]^{1/4}\frac{\delta(\Lambda) }{\varrho(y ; \Lambda)}\norm{g}_{L^{\infty}(\sd{})}.
   \end{split}
\end{equation}
  
\item If $k \delta(\Lambda) \geq 1$, then
  \begin{align}
  \label{eq:tij:k}
  &|\tc{I}{2\ell} g(y)| + \frac{1}{k}\big|\p_y \tc{I}{2\ell} g(y) + d_k \varphi_{\theta d_k}(y - y_{\ell}) \Xi_{\ell}(y)\big| \notag\\ 
  &\lesssim   \frac{ \exp\big(- k\max(|y| - \delta_1(\lambda), 0) \big)  }{k\delta(\Lambda)}\norm{g}_{L^{\infty} (\sd{})}. 
  \end{align}
\end{itemize}

\end{lemma}

\begin{proof}
First, observe that $\tc{I2}{\ell}$  can be written as 
\begin{align}
\label{eq:tI:noder}
    \tc{I}{2\ell}g(y) &=  \int_{\T_\tp}   \gr  \varphi_{\theta d_k}^2(z - y_{\ell}) \ai(\varphi_{\delta(\Lambda)} b'' g)(z) \, dz \notag \\  
    &=  \int_{\T_\tp}  (1 - d_k \p_z)\left(  \varphi_{\theta d_k}(z - y_{\ell}) \gr  \right)\Xi_{\ell}(z) \,dz. 
\end{align}

Similarly, $\p_y\tc{I2}{\ell}$ can be written as
\begin{align}
\label{eq:tI:derivative}
&\p_y\tc{I}{2\ell}g(y) = \notag \p_y \int_{\T_\tp}   \gr  \varphi_{\theta d_k}^2(z - y_{\ell}) \ai(\varphi_{\delta(\Lambda)} b'' g)(z) \,dz  \notag\\
&= \p_y \int_{\T_\tp}  (1 - d_k \p_z)\left(  \varphi_{\theta d_k}(z - y_{\ell}) \gr  \right)\Xi_{\ell}(z) \,dz ,   \notag\\
    &=  d_k\int_{\T_\tp} \varphi_{\theta d_k}(z - y_{\ell}) \p_y^2\gr \Xi_{\ell}(z) \,dz  \notag\\
    &-d_k \int_{\T_\tp} \p_y (\p_y + \p_z)\left(  \varphi_{\theta d_k}(z - y_{\ell}) \gr  \right)\Xi_{\ell}(z) \,dz  \notag\\
    &+ \int_{\T_\tp}  \left(  \varphi_{\theta d_k}(z - y_{\ell}) \p_y\gr  \right)\Xi_{\ell}(z) \, dz. 
    \end{align}
The equation for $\gr$ implies 
\begin{align}
\label{eq:green:second:order}
    &\int_{\T_\tp} \varphi_{\theta d_k}(z - y_{\ell}) \p_y^2\gr \Xi_{\ell}(z) \, dz  \notag \\
    &= -\varphi_{\theta d_k} (y - y_{\ell})\Xi_{\ell}(y) + (V(y)+ k^2) \int_{\T_\tp}  \left(  \varphi_{\theta d_k}(z - y_{\ell}) \gr  \right)\Xi_{\ell}(z) \, dz. 
\end{align}
To summarize, we write 
\begin{align}
\label{eq:tI:derivative:summary}
 \p_y \tc{I2}{\ell}g(y) =
  - d_k\varphi_{\theta d_k}(y - y_{\ell})\Xi_{\ell}(y) + \int_{\T_\tp} R_\ell(y,z) \Xi_{\ell}(z)  \, dz,  
\end{align}
where $\text{supp}(R_\ell(y, \cdot)) \subset \left[y_{\ell} -\theta d_k, y_{\ell} -\theta d_k\right]$, and 
\begin{align}
    |R_\ell(y,z)|\lesssim 
    \begin{cases}
        e^{-k|y-z|},  \quad  & d_k = \frac{1}{k}, \\
          \min\Big\{\frac{\delta(\Lambda)}{\varrho(y ; \Lambda)},e^{-k|y-z|}\Big\}, & d_k = \delta(\Lambda).
    \end{cases}
\end{align}
Therefore,
\begin{align}
  &|\tc{I}{2\ell} g(y)| + \varrho_k(y)\left|\p_y \tc{I}{2\ell} g(y) + d_k\varphi_{\theta d_k}(y-y_{\ell})\Xi_{\ell}(y)\right|   \notag\\
  &\lesssim  
  \begin{cases}
     \frac{ \exp{- k\max\big(|y| - \delta_1(\lambda), 0\big) }  }{k\delta(\Lambda)}\norm{g}_{L^{\infty} (\sd{})}, & {\rm if}\,\,d_k = \frac{1}{k},\\
      \left[\frac{\varrho_k(y;\Lambda)}{\varrho(y;\Lambda)}\right]^{1/4}\frac{\delta(\Lambda) }{\varrho(y ; \Lambda)}\norm{g}_{L^{\infty}(\sd{}) }, &{\rm if}\,\, d_k = \delta(\Lambda),
  \end{cases}
\end{align}
as desired.
To derive \eqref{eq:tij:second:derivatitve}, observe that \eqref{eq:tI:derivative} and \eqref{eq:green:second:order} imply
\begin{align}
    &\p_y\big(\p_y \tc{I}{2\ell} g + d_k\varphi_{\theta d_k} (y - y_{\ell})\Xi_{\ell}(y) \big)  \notag\\
    &=   - \varphi_{\theta d_k} (y - y_{\ell})\Xi_{\ell}(y) + (V(y) + k^2) \int_{\T_\tp} \gr \Xi_{\ell}(z) \, dz
    \notag\\
    &\quad+d_k\p_y\Big( \big(V(y) + k^2\big) \int_{\T_\tp} \gr \Xi_{\ell}(z) \, dz \Big) \\
    &\quad- d_k \int_{\T_\tp} \p_y^2 (\p_y + \p_z)\left(  \varphi_{\theta d_k}(z - y_{\ell}) \gr  \right)\Xi_{\ell}(z) \,dz,
\end{align}
from which \eqref{eq:tij:second:derivatitve} follows. 
\end{proof}
Combining Lemmas \ref{lem:ti2:regularized}, \ref{lem:tio:pw}, and \ref{lem:tI2j:pw} we get the following bounds for $\tc{I}{2}$ on $X^{\sigma_1, \sigma_2}(\mathfrak M)$. 
\begin{lemma}
\label{lem:ti2:X}
Assume that  $(\alpha, \lambda) \in (-\sigma_0 \epsilon^{\frac{1}{2}}, 1)  \times \Sigma_0$,  and  $k \in \Z \cap [1, \infty)$. Then for all $\epsilon \in (0, 1/8)$, $\sigma_1 \in [0,1]$,  and $\sigma_2 \in [-2, 1/2]$ satisfying $\sigma_1 + \sigma_2 \leq 1$, we have the following bounds for $g \in X^{\sigma_1, \sigma_2}(\mathfrak M)$,
\begin{equation}\label{eq:tc:X}
       \norm{\tc{I}{2}g}_{X^{\sigma_1, \sigma_2}(\mathfrak M)} \lesssim \frac{\norm{g}_{X^{\sigma_1, \sigma_2}(\mathfrak M)}}{ \lb k \delta(\Lambda) \rb^{\frac{1}{2}} },
   \end{equation}
   and for $y\in\T_\tp$ with $|y|\gtrsim \delta_0$,
   \begin{equation}\label{eq:tc:X.5J}
        \varrho^{\sigma_1}(y ; \Lambda) \varrho_k^{\sigma_2}(y ; \Lambda)\big| \varrho_k^\beta(y ; \Lambda) \p_y^\beta\tc{I}{2}g(y)\big|\lesssim_{\sigma_1,\sigma_2}\varrho^{\sigma_1-1/4}(y ; \Lambda)\varrho_k^{\sigma_2+1/4}(y ; \Lambda)  \norm{g}_{X^{\sigma_1, \sigma_2}(\mathfrak M) }.
   \end{equation}
\end{lemma}    
\begin{proof}
We focus on \eqref{eq:tc:X}, since \eqref{eq:tc:X.5J} is easier as $|y|\gtrsim\delta_0$.  We first assume that we are in the regularly degenerate case so that one of the cases in \eqref{reg:degen} holds. From Lemma \ref{lem:ti2:regularized}  we have for $\beta\in\{0,1\}$,
\begin{align}
    \varrho^{\sigma_1}(y ; \Lambda) \varrho_k^{\sigma_2}(y ; \Lambda)| \varrho_k^\beta(y ; \Lambda) \p_y^\beta\tc{I}{2}g(y)| \lesssim 
    \begin{cases}
         \varrho^{\sigma_1}(y ; \Lambda)\varrho_k^{\sigma_2}(y ; \Lambda) \frac{ \delta(\Lambda)  \norm{g}_{X^{\sigma_1, \sigma_2}(\mathfrak M) }}{\delta(\Lambda)^{\sigma_1 + \sigma_2} \varrho(y ; \Lambda)},& {\rm if}\,\,  k \delta(\Lambda) \leq 1, \\
         \varrho^{\sigma_1}(y ; \Lambda) \frac{ (k \delta(\Lambda))^{\frac{1}{2}}  \norm{g}_{X^{\sigma_1, \sigma_2}(\mathfrak M) }}{ k^2 \varrho^2(y ; \Lambda) \delta(\Lambda)^{\sigma_1 } },& {\rm if}\,\,  k \delta(\Lambda) \geq 1.
    \end{cases}
   \end{align}
Using the fact that $\sigma_1 + \sigma_2 \leq 1$ when $k\delta(\Lambda) \leq 1$ and that $\sigma_1 \leq 1$ when $k \delta(\Lambda) \geq 1$ implies the desired bound.
Now consider the singularly degenerate case \eqref{sing:degen}. 
For $g \in X^{\sigma_1, \sigma_2}(\mathfrak M)$ we have from Lemma that
\ref{lem:Airybounds:intermediate}
\begin{equation}
    \delta(\Lambda)^{-\frac{1}{2}}d_k^{\sigma_2 +1} \delta(\Lambda)^{\sigma_1}\norm{d_k \varphi_{\theta d_k} (\cdot- y_{\ell}) \Xi_{\ell} }_{L^2(\T_\tp)} \lesssim 
    \begin{cases}
        \norm{g}_{X^{\sigma_1, \sigma_2}(\mathfrak M)},  &  {\rm if}\,\,  k \delta(\Lambda) \leq 1,\\
        \frac{\norm{g}_{X^{\sigma_1, \sigma_2}(\mathfrak M)} }{ k\delta(\Lambda) }, & {\rm if}\,\, k\delta(\Lambda) \geq 1,
    \end{cases}
    \end{equation}
where $d_k \varphi_{\theta d_k} ( \cdot - y_{\ell}) \Xi_{\ell} $ is as in Lemma \ref{lem:tI2j:pw}. Finally, \eqref{eq:tij:k} implies that for $y \in \T_\tp$ 
\begin{align}
&\varrho^{\sigma_1}(y ; \Lambda) \varrho_k^{\sigma_2}(y ; \Lambda)\big|\tc{I}{2j}g(y)\big|  + 
 \varrho^{\sigma_1}(y ; \Lambda)  \varrho_k^{1 + \sigma_2}(y ; \Lambda)\big|  \p_y\tc{I}{2j}g(y) + d_k \varphi_{\theta d_k} (\cdot- y_{\ell}) \Xi_{\ell}(y) \big| \notag \\
& \lesssim
\begin{cases}
\varrho^{\sigma_1}(y ; \Lambda)\varrho_k^{\sigma_2}(y ; \Lambda) \frac{ \delta(\Lambda)  \norm{g}_{X^{\sigma_1, \sigma_2}(\mathfrak M) }}{\delta(\Lambda)^{\sigma_1 + \sigma_2} \varrho(y ; \Lambda)},&  \quad{\rm if}\,\,  k \delta(\Lambda) \leq 1, \\
         \varrho^{\sigma_1}(y ; \Lambda) \frac{   \norm{g}_{X^{\sigma_1, \sigma_2} (\mathfrak M) }}{ (k\delta(\Lambda))^{\frac{1}{2}}\delta(\Lambda)^{\sigma_1 }  } \times 
         \exp{- k\max(|y| - \delta_1(\lambda), 0) } ,
         &  \quad{\rm if}\,\,  k \delta(\Lambda) \geq 1.
    \end{cases}
   \end{align}
   The fact that $\sigma_1 + \sigma_2 \leq 1$ when $k\delta(\Lambda) \leq 1$ implies the desired bound.
   \end{proof}

\subsection{Bounds on $\tc{v}{1}$ }
Recall the definition of $\tc{v}{1}$ (assuming that $y_{1*} = 0$) for $y\in\T_\tp$, 
\begin{equation}
     \tc{v}{1}g(y) = \int_{\T_\tp} \gr  
  \varphi_{\frac{\delta(\Lambda)}{3}}(z) \ai( \eta_{\delta(\Lambda)} b'' g )(z)\, dz.
\end{equation}
Heuristically, we expect that $\tc{v}{1}$ to be very small since  for functions $g$ that are supported at distances greater than  $\delta(\Lambda)$ away from the critical layer, $\ai(g)(y)$ is exponentially small for $|y| < \delta(\Lambda)$ from Lemma \ref{lem:exp:decay}. Crucially, Lemma \ref{lem:exp:decay} also has the correct scaling between $\delta(\Lambda)$ and $\epsilon$: when $\delta(\Lambda) \approx \epsilon^{1/4}$, the exponential term is no longer small enough to view $\tc{v}{1}$ as negligible and we need higher regularity estimates for $\tc{v}{1}$ in order to prove the required  compactness for the limiting absorption principle. 
\begin{lemma}
 \label{lem:tv1:pw}
Assume that  $(\alpha, \lambda) \in (-\sigma_0 \epsilon^{\frac{1}{2}}, 1)  \times \Sigma_{j,\delta_0}$,  and  $k \in \Z \cap [1, \infty)$. Then for all $\epsilon \in (0, 1/8)$ and  $g \in X^{\sigma_1, \sigma_2}(\mathfrak M)$
 with $\sigma_1 \in [0,1], \sigma_2 \in [-2, 1/2] $ satisfying $\, \sigma_1 + \sigma_2 \in (-2, 1)$,  the following bounds hold for all $y \in \T_\tp$.
 \begin{itemize}
     \item For $\beta \in \{ 0,1 ,2\} \text{ and } k\delta(\Lambda) \leq 1$,

     \begin{align}
     \label{tv1:pw}
\varrho_k^\beta(y ; \Lambda)|\p_y^{\beta}\tc{v}{1}g(y) | \lesssim  \frac{\exp(-c_1 \delta(\Lambda)^2 \epsilon^{-\frac{1}{2}})\delta(\Lambda) \norm{g}_{X^{\sigma_1, \sigma_2} (\mathfrak M) } }{\varrho(y ; \Lambda) \delta(\Lambda)^{\sigma_1 + \sigma_2}};
     \end{align}

     \item For $\beta\in\{0,1\}$, we also have
\begin{align}
\label{tv1;X}
\varrho_k^\beta(y ; \Lambda)|\p_y^{\beta}\tc{v}{1}g(y) | \lesssim 
\begin{cases}
\exp(-c_1 \delta(\Lambda)^2 \epsilon^{-\frac{1}{2}}) \varrho^{-\sigma_1-1/4}\varrho_k^{-\sigma_2+1/4}(y ; \Lambda)\norm{g}_{ X^{\sigma_1, \sigma_2}(\mathfrak M) },   & k\delta(\Lambda) \lesssim 1, \\
 k^{\sigma_2}\exp(-c_1 \delta(\Lambda)^2 \epsilon^{-\frac{1}{2}}) \varrho^{-\sigma_1}(y ; \Lambda) \big(k\delta(\Lambda)\big)^{-2}\norm{g}_{ X^{\sigma_1, \sigma_2} (\mathfrak M) }, & k \delta(\Lambda) \gtrsim 1, 
\end{cases}
\end{align}  
where $c_1$ is as in \eqref{eq:exp:close}. 

 \end{itemize}
     \end{lemma}
\begin{corollary}
    \label{lem:tv1:X}
    Assume that $(\alpha, \lambda) \in (-\sigma_0 \epsilon^{\frac{1}{2}}, 1)  \times \Sigma_{j,\delta_0}$,  and  $k \in \Z \cap [1, \infty)$. Then for all $\epsilon \in (0, 1/8)$  
    with $\sigma_1 \in [0,1]$, $\sigma_2 \in [-2, 1/2]$ satisfying $\sigma_1 + \sigma_2 \in (-2,1)$  the following bounds hold for $g \in X^{\sigma_1, \sigma_2}(\mathfrak M)$,
    \begin{align}
        \label{eq:tv1:X}
       \norm{\tc{v1}{}g}_{X^{\sigma_1, \sigma_2}(\mathfrak M)} \lesssim_{\sigma_1, \sigma_2} 
       \begin{cases}
           \frac{\epsilon \norm{ g}_{X^{\sigma_1, \sigma_2}(\mathfrak M)}}{\delta(\Lambda)^4}  & 
            k\delta(\Lambda) \leq 1 
           \\
           \frac{\norm{g}_{X^{\sigma_1, \sigma_2}(\mathfrak M)}}{ \lb k \delta(\Lambda) \rb^{2} } & k \delta(\Lambda) \geq 1,
           \end{cases}
    \end{align}
    where the implicit constants are is uniform in $(\alpha, \epsilon, k, \lambda)$.
\end{corollary}
\begin{proof}
   From \eqref{eq:exp:close} we have for $z \in \T_\tp$,
   \begin{equation}
       \Big|\varphi_{\frac{\delta(\Lambda)}{3}}(z)\ai( \eta_{\delta(\Lambda)} b'' g )(z)\Big| \lesssim_{\sigma_1, \sigma_2}  \exp(-c_1 \delta(\Lambda)^2 \epsilon^{-\frac{1}{2}})\frac{\norm{g}_{X^{\sigma_1, \sigma_2}(\mathfrak M)}}{d_k^{\sigma_1} \delta(\Lambda)^{2 + \sigma_2}}.  
   \end{equation}
   The proof of \eqref{tv1:pw} then follows from Lemma \ref{mGk50}. For \eqref{tv1;X}, observe that since $\sigma_1 + \sigma_2 < 1$ and  $k\delta(\Lambda) \leq 1$ it follows  for  $y \in \T_\tp$ that 
\begin{equation}
    \frac{\delta(\Lambda)\varrho^{\sigma_1}(y ; \Lambda) \varrho_k^{\sigma_2}(y ; \Lambda)  }{\varrho(y ; \Lambda) \delta(\Lambda)^{\sigma_1 + \sigma_2}} \lesssim 1.   
\end{equation}
 For the case of $k \delta(\Lambda) \geq 1$, we use that $ke^{-k|z|}$ is an approximation of the identity.
The proof of corollary \ref{lem:tv1:X} is immediate.
\end{proof}

\subsection{Bounds on $\tc{v}{2}$}
Recall from \eqref{eq:tv2} the definition of $\tc{v}{2}$ assuming that $y_{1*} = 0$, 
\begin{equation}
\label{tv2:def}
    \tc{v}{2}g(y) =  \int_{\T_\tp} \gr (1 -\varphi_{\frac{\delta(\Lambda)}{{3}}}(z) )\frac{\epsilon \p_z^2\ai( \eta_{\delta(\Lambda)} b'' g)(z)}{b(z) - \lambda - i  \alpha }\, dz.
\end{equation}
 For $\epsilon^{\frac{1}{2}} \ll |\lambda|$, we expect that $\tc{v}{2}$ to be small since $\ai(\eta_{\delta(\Lambda)} b'' g)$ acts like multiplication by $\frac{1}{|b(y) - \lambda| + \delta(\Lambda)^2 }$. However, some care is needed in the case when $ |\lambda|\lesssim  \epsilon^{\frac{1}{2}}$ in order to ensure the correct  behavior for $ |y| \gg \delta(\Lambda)$ (see \eqref{eq:tv2:pw1}). Unlike $\tc{v}{1}$, the kernel of $\tc{v}{2}$ is not localized to $S_{\delta(\Lambda)}$ which complicates matters since it is no longer sufficient to simply consider the size of the kernel on $S_{\delta(\Lambda)}$. Capturing the optimal decay in the kernel away from $S_{\delta(\Lambda)}$ is the reason why we integrate by parts in \eqref{eq:tv2:noder} and \eqref{eq:tv2:derivative}.
\begin{lemma}
    \label{lem:tv2:pw}
Assume that $(\alpha, \lambda) \in (-\sigma_0 \epsilon^{\frac{1}{2}}, 1)  \times \Sigma_{j,\delta_0}$,  and  $k \in \Z \cap [1, \infty)$. Then for all $\epsilon \in (0, 1/8)$ and  $g \in X^{\sigma_1, \sigma_2}(\mathfrak M)$
 with $\sigma_1 \in [0,1], \, \sigma_1 + \sigma_2 \in (-2, 1)$  the following bounds hold for all $y \in \T_\tp$.
 \begin{itemize}
 \item For all values of $k \delta(\Lambda)$ and all $\lambda \in \Sigma_{j,\delta_0}$ we have 
    \begin{align}
    \label{eq:tv2:pw}
|\varrho_k^\beta(y ; \Lambda)\p_y^\beta\tc{v}{2}g(y)|   \lesssim_{\sigma_1, \sigma_2} \frac{\epsilon^{1/2} \norm{g}_{X^{\sigma_1, \sigma_2} (\mathfrak M) } }{ \delta(\Lambda)^2 \varrho^{\sigma_1+1/2}(y;\Lambda)\varrho_k^{\sigma_2-1/2}(y ; \Lambda) }, \quad \beta \in\{0,1 \};
\end{align}
\item For $k \delta(\Lambda) \leq 1$,
\begin{align}
\label{eq:tv2:pw1}
    \big|\varrho_k^\beta(y ; \Lambda) \p_y^\beta\tc{v}{2}g(y)\big| &\lesssim_{\sigma_1, \sigma_2} 
 \frac{\epsilon^{1/2} \delta(\Lambda) \norm{g}_{X^{\sigma_1, \sigma_2} (\mathfrak M) }}{\delta(\Lambda)^2\varrho(y; \Lambda ) \delta(\Lambda)^{\sigma_1 + \sigma_2}}, \quad \beta \in \{ 0,1\} ;
\end{align}
\item For  $|\lambda| \lesssim \epsilon^{\frac{1}{2}}$, 
\begin{subequations}
\begin{align}
\label{eq:tv2:kdom}
\big|k^{-\beta}\p_y^\beta\tc{v}{2}g(y) \big|  &\lesssim_{\sigma_1, \sigma_2} \frac{1}{\big(k\delta(\Lambda)\big)^{1/2}} \frac{\epsilon^{\frac{1}{4}}  k^{\sigma_2} \norm{g
    }_{X^{\sigma_1, \sigma_2} (\mathfrak M) }  }{\delta(\Lambda) \varrho^{\sigma_1}(y ; \Lambda)},\quad k \delta(\Lambda) \ge 1, \,\,\beta \in \{0,1 \}, \\
    \label{eq:tv2:2der}
   \big |\varrho_k^2(y ; \Lambda) \p_y^2\tc{v}{2}g(y)\big| &\lesssim_{\sigma_1, \sigma_2} \frac{\norm{g}_{X^{\sigma_1, \sigma_2}(\mathfrak{M})  } }{ \varrho^{\sigma_1+1/2}(y ; \Lambda)\varrho_k^{\sigma_2-1/2}(y ,\Lambda)    }, \quad k \delta(\Lambda) \leq 1.
\end{align}
\end{subequations}

\end{itemize}
\end{lemma}
\begin{remark}
    For $ |\epsilon|^{\frac{1}{2}} \ll |\lambda| \ll 1 $, \eqref{eq:tv2:pw} is sufficient to view $\tc{v}{2}$ as negligible as part of the proof of the limiting absorption principle. However, when $|\lambda| \lesssim \epsilon^{\frac{1}{2}}$, we need to treat $\tc{v}{2}$ like $\tc{I}{2}$ and prove higher regularity estimates and favorable dependence on $k \delta(\Lambda)$.   
\end{remark}

\begin{proof}
Integration by parts in \eqref{tv2:def} yields that $T_{v2}$ can be written as 
\begin{align}
\label{eq:tv2:noder}
    \tc{v}{2}g(y)&= 
     \int_{\T_\tp} \gr (1 -\varphi_{\frac{\delta(\Lambda)}{{3}}}(z) )\frac{\epsilon \p_z^2\ai( \eta_{\delta(\Lambda)} b'' g)(z)}{b(z) - \lambda - i  \alpha }\, dz \notag \\
    &=  - \epsilon \int_{\T_\tp} \p_z \gr \frac{1 - \varphi_{\frac{\delta(\Lambda)}{{3}}} (z)}{b(z) - \lambda - i  \alpha} \p_z\ai( \eta_{\delta(\Lambda)} b'' g)(z)\, dz \notag \\
    &- \epsilon\int_{\T_\tp}  \gr \p_z\left(\frac{1 - \varphi_{\frac{\delta(\Lambda)}{{3}}} (z)}{b(z) - \lambda - i  \alpha} \right) \p_z\ai( \eta_{\delta(\Lambda)} b'' g)(z)\, dz.
\end{align}

Similarly, $\p_y\tc_{v2}$ can be written for $y\in\T_\tp$ as
\begin{equation}
\label{eq:tv2:derivative}
\begin{split}
    \p_y\tc{v}{2}g(y)&= 
     \p_y\int_{\T_\tp} \gr (1 -\varphi_{\frac{\delta(\Lambda)}{{3}}}(z) )\frac{\epsilon \p_z^2\ai( \eta_{\delta(\Lambda)} b'' g)(z)}{b(z) - \lambda - i  \alpha }\, dz  \\
    &=  - \epsilon \p_y\int_{\T_\tp} \p_z \gr \frac{1 - \varphi_{\frac{\delta(\Lambda)}{{3}}} (z)}{b(z) - \lambda - i  \alpha} \p_z\ai( \eta_{\delta(\Lambda)} b'' g)(z)\, dz  \\
    &- \epsilon\int_{\T_\tp} \p_y \gr \p_z\left(\frac{1 - \varphi_{\frac{\delta(\Lambda)}{{3}}} (z)}{b(z) - \lambda - i  \alpha} \right) \p_z\ai( \eta_{\delta(\Lambda)} b'' g)(z)\, dz, 
    \end{split}
    \end{equation}
and using \eqref{mGk1}, as
    \begin{equation}\label{eq:tv2:derivativeJ}
    \begin{split}
    \p_y\tc{v}{2}g(y) =& -\epsilon \frac{1 - \varphi_{\frac{\delta(\Lambda)}{{3}}} (y)}{b(y) - \lambda - i  \alpha} \p_y\ai( \eta_{\delta(\Lambda)} b'' g)(y) \\
    &+ \epsilon \big(V(y) + k^2\big)\int_{\T_\tp} \gr \frac{1 - \varphi_{\frac{\delta(\Lambda)}{{3}}} (z)}{b(z) - \lambda - i  \alpha} \p_z\ai( \eta_{\delta(\Lambda)} b'' g)(z) \, dz   \notag\\
    &-\epsilon \int_{\T_\tp} \p_y(\p_z + \p_y) \gr \frac{1 - \varphi_{\frac{\delta(\Lambda)}{{3}}} (z)}{b(z) - \lambda - i  \alpha} \p_z\ai( \eta_{\delta(\Lambda)} b'' g)(z) \, dz \notag \\
    &- \epsilon\int_{\T_\tp} \p_y \gr \p_z\left(\frac{1 - \varphi_{\frac{\delta(\Lambda)}{{3}}} (z)}{b(z) - \lambda - i  \alpha} \right) \p_z\ai( \eta_{\delta(\Lambda)} b'' g)(z)\, dz.
    \end{split}
\end{equation}
Since $\eta_{\delta(\Lambda)}$ is supported a distance of approximately $\delta(\Lambda)$ away from 0, Lemma \ref{ake1} implies that for $y\in\T_\tp$, 
\begin{align}
\label{tv2:1}
   \left|  \epsilon \frac{1 - \varphi_{\frac{\delta(\Lambda)}{{3}}} (y)}{b(y) - \lambda - i  \alpha} \p_y\ai( \eta_{\delta(\Lambda)} b'' g)(y)  \right| \lesssim 
       \frac{\epsilon^{1/2} \norm{g}_{X^{\sigma_1, \sigma_2}(\mathfrak M)}}{\varrho^{3 + \sigma_1}(y ; \Lambda) \varrho_k^{ \sigma_2}(y ; \Lambda)} .
   \end{align}
The bounds \eqref{eq:tv2:pw1}-\eqref{eq:tv2:pw1} then follow from \eqref{tv2:1} and Lemma \ref{LAMweight}. 

We now consider the case  $|\lambda| \lesssim  \epsilon^{\frac{1}{2}}$. We no longer integrate by parts to remove a derivative from  $\p_y^2\ai$ since the scale of variation of $\gr$ is comparable or smaller than that of $\ai$. 
Lemma \ref{ake1} implies that for $y \in \T_\tp$,
\begin{equation}
\label{eq:Airy:derivative1}
    \left|\epsilon \frac{1 - \varphi_{\frac{\delta(\Lambda)}{{3}}} (y)}{b(y) - \lambda - i  \alpha} \p_y^{2} \ai(\eta_{\delta(\Lambda)} b'' g)(y) \right|\lesssim \frac{\epsilon^{\frac{1}{2}} \norm{g}_{X^{\sigma_1, \sigma_2} (\mathfrak M)}}{\varrho^{ 3 + \sigma_1}(y; \Lambda)\varrho_k^{1 + \sigma_2  }(y; \Lambda)}  ,
\end{equation}
 which, in combination with Lemma \ref{LAMweight} 
 and \eqref{eq:Airy:derivative1}, implies that for $k \delta(\Lambda) \geq 1$, 
\begin{align}
    \left|k^{-\beta}\p_y^\beta\tc{v}{2}g(y) \right|  \lesssim \frac{1}{\big(k\delta(\Lambda)\big)^{1/2}} \frac{\epsilon^{\frac{1}{4}}  k^{\sigma_2} \norm{g
    }_{X^{\sigma_1, \sigma_2}(\mathfrak M)}  }{\delta(\Lambda) \varrho^{\sigma_1}(y;\Lambda)}, \quad \beta \in \{0,1 \}.
\end{align}
For the second derivative estimates, we simply use the identity that
\begin{align}
     \p_y^2\tc{v}{2}g(y)
     =& -\epsilon \frac{1 - \varphi_{\frac{\delta(\Lambda)}{{3}}} (y)}{b(y) - \lambda - i  \alpha} \p_y^2\ai( \eta_{\delta(\Lambda)} b'' g)(y) \notag\\
     &+ (V(y) + k^2)\int_{\T_\tp} \gr \frac{1 - \varphi_{\frac{\delta(\Lambda)}{{3}}} (z)}{b(z) - \lambda - i  \alpha} \epsilon\p_z^2\ai( \eta_{\delta(\Lambda)} b'' g)(z) \, dz.
\end{align}
Therefore for $k\delta(\Lambda) \leq 1$
\begin{equation}
    \varrho_k^2(y ; \Lambda)|\p_y^2\tc{v}{2}g(y)| \lesssim_{\sigma_1, \sigma_2} \frac{\norm{g}_{X^{\sigma_1, \sigma_2}(\mathfrak M)}}{ \varrho^{\sigma_1}(y; \Lambda)\varrho_k^{\sigma_2}(y; \Lambda)    },
\end{equation}
which concludes the proof. 
\end{proof}
The following is an immediate corollary of Lemma \ref{lem:tv2:pw}.
\begin{corollary}
    \label{lem:tv2:X}
    Assume that  $(\alpha, \lambda) \in (-\sigma_0 \epsilon^{\frac{1}{2}}, 1)  \times \Sigma_{j,\delta_0}$,  and  $k \in \Z \cap [1, \infty)$. Then for all $\epsilon \in (0, 1/8)$ the following bounds hold for $\tc{v}{2}: X^{\sigma_1, \sigma_2}(\mathfrak M) \to X^{\sigma_1, \sigma_2}(\mathfrak M)$ 
 with $\sigma_1 \in [0,1]$, $\sigma_1 + \sigma_2 \in (-2,1)$.
\begin{align}
        \label{eq:tv2:X}
       \norm{\tc{v2}{}}_{X^{\sigma_1, \sigma_2} (\mathfrak M)} \lesssim_{\sigma_1, \sigma_2} 
       \begin{cases}
           \frac{\epsilon^{1/2}}{\delta(\Lambda)^2},  & {\rm if}\,\,\lambda \in \Sigma_{j,\delta_0}, \\
           \frac{\epsilon^{\frac{1}{4}}}{ \delta(\Lambda)\lb k \delta(\Lambda) \rb^{1/2} }, & {\rm if}\,\,
           |\lambda| \lesssim  \epsilon^{\frac{1}{2}} \text{ and } k \delta(\Lambda) \geq 1.
           \end{cases}
    \end{align}
    
\end{corollary}

\bigskip

\begin{center}
   {\bf Acknowledgement}
\end{center}
We thank Peter Constantin for bringing the paper \cite{Sinai} to our attention. RB would also like to thank Dallas Albritton,  Tarek Elgindi, and Vlad Vicol for helpful discussions and encouragement.

\end{document}